\documentclass{article}

\usepackage{microtype}
\usepackage{graphicx}
\usepackage{subfigure}
\usepackage{booktabs} 
\usepackage{multirow}
\usepackage{makecell}
\usepackage{calc}
\usepackage{dirtytalk}

\newcommand\mytabspace[1]{%
 \parbox[c][\totalheightof{#1}+4\fboxsep+4\fboxrule][c]{\widthof{#1}}{#1}%
}

\usepackage{hyperref}

\usepackage[accepted]{icml2023}

\usepackage{amsmath}
\usepackage{amssymb}
\usepackage{mathtools}
\usepackage{amsthm}

\usepackage[toc,page,header]{appendix}
\usepackage{minitoc}

\usepackage[capitalize,noabbrev]{cleveref}

\theoremstyle{plain}
\newtheorem{theorem}{Theorem}[section]
\newtheorem{proposition}[theorem]{Proposition}
\newtheorem{lemma}[theorem]{Lemma}
\newtheorem{corollary}[theorem]{Corollary}
\theoremstyle{definition}
\newtheorem{definition}[theorem]{Definition}
\newtheorem{assumption}[theorem]{Assumption}
\newtheorem*{remark}{Remark}

\usepackage[textsize=tiny]{todonotes}

\newcommand{\assign}{:=}
\newcommand{\comma}{{,}}

\newcommand{\infixand}{\text{ and }}
\newcommand{\nobracket}{}
\newcommand{\tmcolor}[2]{{\color{#1}{#2}}}
\newcommand{\tmop}[1]{\ensuremath{\operatorname{#1}}}
\newcommand{\tmscript}[1]{\text{\scriptsize{$#1$}}}
\newcommand{\tmstrong}[1]{\textbf{#1}}

\renewcommand{\algorithmicrequire}{\textbf{Input:}}
\renewcommand{\algorithmicensure}{\textbf{Output:}}

\newcommand{\Atype}{\mathsf{A}}
\newcommand{\Gtype}{\mathsf{G}}
\newcommand{\Htype}{\mathsf{H}}

\newcommand{\tocite}{(\textcolor{orange}{citation needed}) }

\newcommand{\R}{\mathbb{R}}
\newcommand{\cS}{\mathcal{S}}
\newcommand{\X}{\mathcal{X}}

\newcommand{\K}{\mathcal{K}}
\newcommand{\bX}{\mathbb{X}}

\newcommand{\N}{\mathbb{N}}
\newcommand{\D}{\mathcal{D}}
\newcommand{\E}{\mathbb{E}}

\newcommand{\norm}[1]{\left\lVert#1\right\rVert}

\newcommand{\sprod}[2]{\left\langle #1, #2 \right\rangle}

\newcommand{\cris}[1]{\textcolor{red}{#1}}
\newcommand{\nico}[1]{\textcolor{orange}{#1}}

\newcommand{\iid}{\overset{\text{iid}}{\sim}}

\icmltitlerunning{Neural signature kernels as infinite-width-depth-limits of controlled ResNets}

\begin{document}

{
\twocolumn[
\icmltitle{Neural signature kernels as infinite-width-depth-limits of controlled ResNets}



\icmlsetsymbol{equal}{*}

\begin{icmlauthorlist}
\icmlauthor{Nicola Muça Cirone}{ICL}
\icmlauthor{Maud Lemercier}{OX}
\icmlauthor{Cristopher Salvi}{ICL}
\end{icmlauthorlist}

\icmlaffiliation{ICL}{Department of Mathematics, Imperial College London, London, United Kingdom}
\icmlaffiliation{OX}{Department of Mathematics, University of Oxford, Oxford, United Kingdom}

\icmlcorrespondingauthor{Nicola Muça Cirone}{nm2322@ic.ac.uk}

\icmlkeywords{Machine Learning, ICML}

\vskip 0.3in
]}



\printAffiliationsAndNotice{}  

\begin{abstract}
Motivated by the paradigm of reservoir computing, we consider randomly initialized \emph{controlled ResNets} defined as Euler-discretizations of \emph{neural controlled differential equations} (Neural CDEs){, a unified architecture which enconpasses both RNNs and ResNets}. We show that in the infinite-width-depth limit and under proper scaling, these architectures converge weakly to Gaussian processes indexed on some spaces of continuous paths and with kernels satisfying certain partial differential equations (PDEs) varying according to the choice of activation function $\varphi${, extending the results of \citet{hayou2022infinite, hayou2023width} to the controlled and homogeneous case}. In the special{, homogeneous, }case where $\varphi$ is the identity, we show that the equation reduces to a linear PDE and the limiting kernel agrees with the \emph{signature kernel} of \citet{salvi2021signature}. We name this new family of limiting kernels \emph{neural signature kernels}. 
Finally, we show that in the infinite-depth regime, finite-width controlled ResNets converge in distribution to Neural CDEs with random vector fields which, depending on whether the weights are shared across layers, are either time-independent and Gaussian or behave like a matrix-valued Brownian motion.

\end{abstract}

\section{Introduction}


The symbiosis between differential equations and deep learning has become an active research area in recent years, notably through the introduction of hybrid models named \emph{neural differential equations} \cite{kidger2022neural}. In fact, many standard neural network architectures may be interpreted as approximations to some differential equations. 

This approximation has been treated rigorously for finite-width ResNets, which in the infinite-depth limit converge in distribution to zero-drift \emph{neural stochastic differential equations} (Neural SDEs) with diffusion depending on the choice of activation function \cite{cohen2021scaling, hayou2022infinite, marion2022scaling, cont2022asymptotic}.

The \say{dual} scenario of finite-depth and infinite-width neural networks has also been the object of many recent studies \cite{neal2012bayesian, matthews2018gaussian, novak2018bayesian}. Notably, through the unifying algorithmic language of \emph{Tensor Programs} designed by \citet{yang2019wide}, many standard feedforward, convolutional and recurrent architectures of finite-depth can be shown to converge to Gaussian processes (GPs) in the infinite-width limit.

In the context of deep learning for sequential data, finite-width RNNs have been informally identified as approximations to \emph{neural controlled differential equations} (Neural CDEs) introduced by \citet{kidger2020neural, morrill2021neural} and inspired from the homonymous class of dynamical systems studied in \emph{rough analysis}, a branch of stochastic analysis providing a robust solution theory for differential equations driven by irregular signals \cite{lyons1998differential, lyons2007differential, friz2020course, FrizVictoir}.


However, contrarily to this widespread interpretation, the Euler discretization of a Neural CDE with vector fields\footnote{Typically $f$ is taken to be a randomly initialized feedforward neural network with Gaussian weights and biases.} $f$ produces a recursive relation for the hidden state $h$ in the form of (\ref{eq:C-ResNet-layer}), where the increments of the input signal $x$ enter the recursion in a multiplicative manner rather than via an additive interaction typically assumed in RNNs:
\begin{align}\label{eq:C-ResNet-layer}
h_{k+1} = h_k + f(h_k)(x_{k+1}-x_{k}).
\end{align}
Furthermore, the addition of the previous hidden state $h_k$ on the right-hand side of (\ref{eq:C-ResNet-layer}), commonly referred to as a skip connection, is characteristic of ResNets and absent in classical RNNs. 
We will refer to architectures defined by  (\ref{eq:C-ResNet-layer}) as \emph{homogeneous controlled ResNets}. We will also consider their  \emph{inhomogenous} counterparts where the map $f = f(k,h_k)$ depends on the iteration $k$.

Dynamical systems in the form of (\ref{eq:C-ResNet-layer}) are often called \emph{reservoirs} in the paradigm of \emph{reservoir computing} \cite{TANAKA2019100, lukovsevivcius2009reservoir, verstraeten2007experimental}. Contrarily to deep learning, in reservoir computing, only the final readout linear map is trained, while the  function $f$ is randomly sampled but remains untrained.

It is worth noting that Neural CDEs are deep learning models that map between infinite dimensional spaces of continuous paths. Therefore, if these continuous models converge, in the infinite-width limit, to some limiting GPs, the latter should be equipped with kernel functions indexed on the same spaces of continuous paths. In the sequel we will demonstrate that controlled ResNet indeed behave like such GPs in the large width-depth regime.

\subsection{Contributions}

Our objective here is to provide a rigorous mathematical analysis of the behavior of controlled ResNets that are randomly initialised with Gaussian weights and biases in the large width and depth regimes. More specifically: 

\begin{itemize}
    \item We prove that both in the infinite-width-depth limit these architectures converge weakly to GPs with limiting kernels satisfying certain (possibly non-linear) partial differential equations varying according to the (in)homogeneity of the network and to the choice of activation function $\varphi$ (see \cref{table:kernels}). Moreover, we show that under some further conditions on the regularity of the driving paths the limits commute, i.e. the limiting GP is unchanged upon reversing the order of the limits. We name this new class of kernel \emph{neural signature kernels}.
    \item In the case where the system is homogeneous and $\varphi$ is the identity, we show that the equation reduces to a linear PDE and the limiting kernel is proportional to the \emph{signature kernel} introduced in \cite{salvi2021signature}.
    \item We then prove that in the infinite-depth regime, finite-width controlled ResNets  converge in distribution to Neural CDEs with random vector fields. In the inhomogeneous case, these fields behave as a matrix-valued Brownian motion, while for homogeneous networks they are time-independent and Gaussian.
\end{itemize}

\subsection{Notation} 

Since we are mainly interested in studying multivariate time-series, our data space will be a space of continuous paths on the interval $[0,1]$\footnote{The choice of the interval $[0,1]$ is not at all restrictive and has been made to ease the notation.} and with values in $\mathbb{R}^d$, for some $d\in\mathbb{N}$. More specifically, we consider the space
\begin{equation*}
    \bX := \{ x \in C^0([0,1];\R^d) : x(0) = 0, \exists \Dot{x} \in L^2([0,1]; \R^d) \}
\end{equation*}
of continuous paths with a square integrable derivative. 

We denote by $x^j(t) \in \R$ the $j^{th}$ coordinate of a path $x\in\bX$ for $j\in\{1,...,d\}$. Given $n$ paths $\X = \{x_1, \dots, x_n\} \subset \bX$ and functions $f: \bX \to \R$, $G: \bX \times \bX \to \R$  we will write $f(\X) \in \R^n$ for the vector $[f(\X)]_{\alpha} = f(x_{\alpha})$ and $G(\X, \X) \in \R^{n \times n}$ for the matrix $[G(\X,\X)]_{\alpha}^\beta = G(x_{\alpha}, x_{\beta})$ for any $\alpha, \beta \in \{1,..,n\}$.

We will consider partitions $\D = \{0 = t_0 < \cdots < t_M = 1\}$ of the interval $[0,1]$ and write their length as $\norm{\D} := M$ and their mesh size as $|\D| := \max\limits_{i = 1,\dots, \norm{\D}} |t_i - t_{i-1}|$. 

For an activation function $\varphi : \R \to \R$ and a positive semidefinite matrix $\Sigma \in \R^{2 \times 2}$, in the paper we will repeatedly make use of the following function $$V_\varphi(\Sigma) = \mathbb{E}_{\textbf{z} \sim \mathcal{N} (\textbf{0}, \Sigma)} [ \varphi(z_1) \varphi(z_2) ].$$
The explicit form of $V_\varphi$ changes significantly depending on the activation function. We list it for a restricted class of them in \cref{app:charact-Phi} in the appendix.

Henceforth, we fix a probability space $(\Omega, \mathcal{F}, \mathbb{P})$.

The paper is organized as follows: in \cref{sec:related_work} we discuss some related work, in \cref{sec:Inhom} we study the \emph{inhomogeneous} version and in  \cref{sec:Hom} we analyse the \emph{homogeneous} version of controlled ResNets. We conclude in Section \ref{sec:experiments} with numerical results validating our claims. All proofs can be found in the appendix.

\section{Related Work}\label{sec:related_work}

Results relating infinite-width limits of neural networks to GPs have been extended from shallow networks \cite{neal2012bayesian} to richer architectures of feedforward \cite{lee2018deep, matthews2018gaussian}, convolutional \cite{novak2018bayesian, garriga2018deep} and recurrent \cite{alemohammad2020recurrent} type.
This line of work culminated with the framework of \emph{Tensor Programs} formulated in \citet{yang2019wide} which offers an algorithmic procedure to systematically compute the limiting GP kernels for a wide range of different architectures. One major advantage of this formalism is that it makes it possible to consider weight sharing between layers, something mostly avoided in previous literature but central in the types of systems we consider here.

The reverse scenario of infinite-depth limit of finite-width architectures has mainly been explored for ResNets. In this regime, appropriately rescaled ResNets have been shown to behave like stochastic differential equations (SDEs) \cite{chen2018neural, cohen2021scaling, marion2022scaling}. In particular, \citet{hayou2022infinite} considers the simpler setting where the architecture does not exhibit weight sharing (we refer to this setting as inhomogeneous) and single out limiting kernels of exponential type. In what follows, we will show that the exponential nature of the limiting kernels is somewhat kept intact even when the ResNets architectures are controlled by an external stream of information, but that the limiting kernels have more structure, particularly in the homogeneous case.

It is natural to investigate the behavior of neural networks when both the width and the depth are very large. The literature on the topic has particularly developed around the infinite-width-then-depth limit, most notably with the derivation of the Edge of Chaos \cite{poole2016expon,schoenholz2017deep} which has shown how feedforward networks suffer from a kind of information dispersal which limits the propagation of the input signals through their depth. Noteworthy are also the results in \citet{Mufan2021future}, where both limits are taken together for some fixed depth-to-width ratio and non-Gaussian behaviour at the limit of a specific kind of ResNets is discovered, and in \citet{mufan2022NeuralCovariance} where, in the feedforward setting, stochastic dynamics are given for the covariance between output layers in the same regime.
{In the existing literature, the study most closely aligned with our work is \cite{hayou2023width}
where the commutativity of the limits is shown for residual Networks corresponding to the simplest example of  controlled ResNets: \emph{inhomogeneous} ones, driven by linear controls.}
\footnote{
{To be precise one has to first expand our framework to include different initial values, of the form $W_{in} x_0$, where $x_0 \in \R^d$ is the "classical" model input. The corresponding control is then $x_t = x_0 + t \boldsymbol{e}_1$. Such richer models will be studied in future work.}
}


As anticipated in the introduction, ResNets that are controlled by sequential data streams are generalised forms of RNNs and correspond to Euler discretizations of Neural CDEs and variants \cite{kidger2020neural, morrill2021neural, salvi2022neural, fermanian2021framing}. These models offer a memory-efficient way to model functions of potentially irregular signals in continuous-time and have achieved state-of-the art performance on a wide range of time series tasks \cite{singh2022multiscale, bellot2021policy, morrill2021neural}. They stem from the well-understood mathematics of \emph{controlled differential equations}, which are the central objects studied in rough analysis.

\emph{Rough path theory} introduced by \citet{lyons1998differential} is a modern mathematical framework focused on making precise the interactions between highly oscillatory signals and non-linear dynamical systems. The theory provides a deterministic toolbox to recover many classical results in stochastic analysis without resorting to specific probabilistic arguments. Notably, it extends Itô's theory of SDEs far beyond the semi-martingale setting and it has had a significant impact in the development of the theory of \emph{regularity structures} by \citet{hairer2014theory}, providing a mathematically rigorous description of many stochastic PDEs arising in physics.
 
More recently, interest has grown rapidly to develop machine learning algorithms based on rough path theoretical tools, particularly in the context of time series analysis \cite{kidger2019deep, arribas2020sig, lemercier2021distribution}. The \emph{signature}, a centrepiece of the theory, provides a top-down description of a stream; it captures crucial information such as the order of different events occurring across different channels, and filters out potentially superfluous information, such as the sampling rate of the signal.

In reservoir computing, the trajectory of a dynamical system is described through its interaction with a random dynamical system that is capable of storing information. In rough path theory the random system is replaced by a deterministic system given by the signature. Recently \cite{cuchiero2021discrete, cuchiero2021expressive} have investigated empirically the idea of a continuous-time reservoir through the randomization of the signature yielding controlled residual architectures similar to the ones of interest to us.

A significant effort has been made to scale methods based on the signature to high dimensional signals. \emph{Signature kernels} are defined as inner products of signatures and provide an elegant solution to this challenge thanks to the recent development of specific kernel tricks \cite{kiraly2019kernels}. Notably, \citet{salvi2021signature} establish that the signature kernel can be computed efficiently by solving a linear PDE. Algorithms based on signature kernels have been used in a wide range of applications including hypothesis testing \cite{salvi2021higher},  cybersecurity \cite{cochrane2021sk}, and probabilistic forecasting \cite{toth2020bayesian, lemercier2021siggpde} among others.

\begin{table*}
\caption{PDEs satisfied by the limiting kernels in the infinite-width-depth limit of controlled ResNets. The \textbf{first row} is for the \emph{inhomogeneous case} while the \textbf{second row} for the \emph{homogeneous} one. The kernel $k_{sig}$ is the \emph{signature kernel} \cite{salvi2021signature}.
}
\begin{tabular}{cc}
\toprule
  \multicolumn{2}{c}{\textbf{Activation function}} \\
       General case &   Identity case              \\ 
   \midrule
\mytabspace{{$\!\begin{aligned}
    \partial_t \kappa^{x,y}_\varphi(t)=
    \Big[  
    \sigma_A^2 V_\varphi 
    \left(\begin{pmatrix}
       \kappa^{x,x}_\varphi(t) & \kappa_{\varphi}^{x, y}(t)\\
       \kappa_{\varphi}^{x, y}(t) & \kappa_{\varphi}^{y, y}(t)
     \end{pmatrix}\right) 
     +\sigma_b^2 
     \Big] 
     \sprod{\Dot{x}_{t}}{ \Dot{y}_{t}}        \end{aligned}$}}
     & 
     $\kappa^{x,y}_{id}(t)=\big( \sigma_a^2 + \frac{\sigma_b^2}{\sigma_A^2}\big) e^{ \int_{0}^t \sprod{\sigma_A\Dot{x}_{s}}{ \sigma_A\Dot{y}_{s}} ds } -\frac{\sigma_b^2}{\sigma_A^2}$                 \\
    \mytabspace{{$\!\begin{aligned}
    \partial_t \partial_s \K^{x,y}_{\varphi}(s, t)= 
    \Big[        
    \sigma_A^2 V_\varphi 
    \left(\begin{pmatrix}
       \K_{\varphi}^{x, x}(s,s) & \K_{\varphi}^{x, y}(s,t)\\
       \K_{\varphi}^{x, y}(s,t) & \K_{\varphi}^{y, y}(t,t)
     \end{pmatrix}\right) +    \sigma_b^2        
     \Big]  \sprod{\Dot{x}_{s}}{ \Dot{y}_{t}}   \end{aligned}$}}   
     & \mytabspace{$\K^{x,y}_{id}(s,t)=\big( \sigma_a^2 + \frac{\sigma_b^2}{\sigma_A^2}\big) k_{sig}^{\sigma_Ax, \sigma_Ay}(s,t)-\frac{\sigma_b^2}{\sigma_A^2}$}\\
    \bottomrule
    \end{tabular}
    \label{table:kernels}
\end{table*}

\section{Inhomogeneous controlled ResNets}\label{sec:Inhom}

We begin by considering the case of inhomogeneous controlled ResNets. Contrarily to what one might expect, although in this setting the residual map changes at each iteration, the limiting kernels will be governed by simpler differential equations than their homogeneous counterparts, as it can be observed in \cref{table:kernels}. At an intuitive level, this fact can be justified by noting that sharing common random weights and biases throughout all iterations introduces a more intricate dependence structure on the dynamics of the system than if the weights and biases were independently sampled at each iteration.

\subsection{The model}

Let $\D_M = \{0 = t_0 < \cdots < t_M = 1\}$ be a partition, $N\in \mathbb{N}$ be the width, and $\varphi: \R \to \R$ an activation function.
Define a \emph{randomly initialized, $1$-layer inhomogeneous controlled ResNet} $\Psi^{M, N} : \bX \to \R$ as follows
\begin{equation*}
    \Psi^{M, N}_\varphi(x) := \sprod{\psi}{\cS^{M,N}_{t_M}(x)}_{\R^N}
\end{equation*}
where $\sprod{\cdot}{\cdot}_{\R^N}$ is the Euclidean inner product on $\R^N$, $\psi \in \R^N$ is a random vector with entries $[\psi]_\alpha \iid \mathcal{N}(0,\frac{1}{N})$, and where the random functions $\cS^{M,N}_{t_i} : \bX \to \R^N$ satisfy the following recursive relation
\begin{equation*}
\begin{gathered}
    \cS^{M,N}_{t_{i+1}} = \cS^{M,N}_{t_i} + \sum_{j=1}^d \big( A_{j,i} \varphi(\cS^{M,N}_{t_i}) + b_{j,i} \big) \Delta x^j_{t_{i+1}}
    \\
    \cS^{M,N}_{t_0} = a \quad \text{and}  \quad \Delta x^j_{t_i} = (x^j_{t_i}-x^j_{t_{i-1}})
\end{gathered}
\end{equation*}
for $i=0,...,M$, with initial condition $[a]_{\alpha} \iid \mathcal{N}(0,\sigma_a^2)$, and Gaussian weights $A_{k,l} \in \R^{N \times N}$ and biases $b_{k,l} \in \R^N$ sampled independently according to
\begin{equation*}
    [A_{j,i}]_{\alpha}^{\beta} \iid \mathcal{N}\left(0,\frac{\sigma_{A}^2}{N\Delta t_i}\right), \quad 
    [b_{j,i}]_{\alpha} \iid \mathcal{N}\left(0,\frac{\sigma_b^2}{\Delta t_i}\right)
\end{equation*}
with time step $\Delta t_i = (t_i- t_{i-1})>0$.

Here $\sigma_{a}, \sigma_{A} > 0$ and $\sigma_{b} \geq 0$  are all model hyperparameters.



\begin{remark}
    The time scaling $\frac{1}{\Delta t_i}$ in the random weights and biases is crucial as it is exactly the scaling one needs to get an It\^{o} diffusion in the distributional infinite-depth limit, as we will prove in \cref{thm:inhom-inf-D-limit} below.
\end{remark}

\subsection{The infinite-width-depth regime}\label{subsec:inf-W-D-inhom}

The first problem we are interested in studying is that of characterizing the limiting behavior of these neural networks in the infinite-width-then-depth regime.

\Cref{thm:main-inhom-Kernels} states that in this regime, these architectures converge weakly to GPs indexed on the path space $\bX$ with kernels satisfying a one-parameter differential equation.


\begin{theorem}\label{thm:main-inhom-Kernels}
Let $\{\D_M\}_{M \in \N}$ be a sequence of partitions of $[0,1]$ such that $|\D_M| \to 0$. Let the activation function $\varphi: \R \to \R$ be linearly bounded, absolutely continuous and with exponentially bounded derivative.
Then the following weak convergence\footnote{By weak convergence we mean that for any subset of paths $\X = \{x_1, \dots, x_n\} \subset \bX$ the random vector $\Psi^{M, N}_\varphi(\X)$ converges in distribution to corresponding evaluations of the RHS limit.} holds 
\begin{equation}\label{eqn:limits_inhom_kernels}
    \lim_{M \to \infty} \lim_{N \to \infty}  \Psi^{M, N}_\varphi =
    \mathcal{GP}(0,\kappa_\varphi),
\end{equation}
where the positive semidefinite kernel $\kappa_\varphi : \bX \times \bX \to \R$ is defined for any two paths $x,y \in \bX$ as $\kappa_\varphi(x,y) = \kappa_\varphi^{x,y}(1)$,  where $\kappa_\varphi^{x,y} : [0,1] \to \R$ is the unique solution of the following differential equation
\begin{equation}\label{eqn:inhom_kernel}
      \partial_t\kappa_{\varphi}^{x,y} =   
      \Big[ 
      \sigma_A^2 V_\varphi 
    \left(\begin{pmatrix}
       \kappa^{x,x}_\varphi & \kappa_{\varphi}^{x, y}\\
       \kappa_{\varphi}^{x, y} & \kappa_{\varphi}^{y, y}
     \end{pmatrix}\right) 
     +\sigma_b^2 
      \Big] 
     \sprod{\Dot{x}_t}{ \Dot{y}_t}_{\R^d}
\end{equation}
with initial condition $\kappa_{\varphi}^{x,y}(0)=\sigma_a^2$.

\vspace{3pt}
If moreover $\varphi$ is Lipschitz with $\varphi(0) = 0$ and $x \in \bX \cap C^{1,\frac{1}{2}}$, where $C^{1,\frac{1}{2}}$ denotes the set of $C^1$ paths with $\frac{1}{2}$-H\"{o}lder derivative, the limits can be exchanged and
\[
\lim_{N\to \infty} \lim_{M\to \infty} 
\Psi^{M, N}_\varphi = \lim_{M\to \infty} \lim_{N\to \infty} \Psi^{M, N}_\varphi =
    \mathcal{GP}(0,\kappa_\varphi).
\]

\end{theorem}
\begin{proof}[Idea of proof]
We prove the weak convergence (\ref{eqn:limits_inhom_kernels}) in Appendix \ref{app:sec:inhom-inf-W-D-limit}. The first step consists in showing, for a fixed depth $M$, the existence of an infinite-width distributional limit using the techniques established in \cite{yang2019wide}; this limit will be shown to be Gaussian and with covariance kernels $\kappa^{x,y}_{\D_M}:\D_M \to \R$ satisfying a difference equation. The second step amounts to prove that given any sequence of partitions $\D_M$ with $|\D_M| \to 0$, the sequence $\{\kappa^{x,y}_{\D_M}\}_M$ is uniformly bounded and uniformly equicontinuous so that by the Ascoli-Arzel\`a theorem the sequence admits a uniformly convergent subsequence. Finally, we  prove that the limit of this subsequence is a solution of the differential equation (\ref{eqn:inhom_kernel}) and that this solution is actually unique. 

The statement about commutativity of limits is proved in Appendix \ref{app:sub:inhom_commut}, after a characterization of the infinite-depth limit under these more stringent regularity assumptions, by proving that the distributional limit in depth is uniform in width.  
This generalizes the results of \cite{hayou2023width} in our more complex case.

\end{proof}

In some cases we can explicitly characterize the limiting kernels by solving analytically the differential equation (\ref{eqn:inhom_kernel}), as stated in the following corollary
\footnote{We note that these characterizations expressed by means of an exponential are consistent with the results of \cite{hayou2022infinite}.}. 
\begin{corollary}\label{corollary:inhom}
    With the same notation and assumptions as in \cref{thm:main-inhom-Kernels}, upon taking $\varphi = id$ the limiting kernel admits the following explicit expression
    \[
    \kappa_{id}(x,y) = \big( \sigma_a^2 + \frac{\sigma_b^2}{\sigma_A^2}\big) \exp\big \{\sigma_A^2 \int_0^1 \sprod{\Dot{x}_t}{ \Dot{y}_t}_{\R^d} dt \big\} - \frac{\sigma_b^2}{\sigma_A^2}.
    \]
    If $\varphi = ReLU$ and $x=y$, then the limiting kernel satisfies 
    \[
    \kappa_\varphi(x,x) = \big( \sigma_a^2 + \frac{2 \sigma_b^2}{\sigma_A^2}\big) \exp\big \{\frac{\sigma_A^2}{2} \int_{ 0}^1 \norm{\Dot{x}_{t}}_{\R^d}^2 dt \big\} - \frac{2 \sigma_b^2}{\sigma_A^2}
    \]
\end{corollary}

\begin{remark}
    In Lemma \ref{app:lemma:inhom_simpler_form} in the appendix we show that in the kernels governed by the dynamics (\ref{eqn:inhom_kernel}), the parameters $\sigma_A$ and $\sigma_b$ satisfy the following path-rescaling symmetry
    $$\kappa_{\varphi}^{x,y}(t; \sigma_A, \sigma_b) = \kappa_{\varphi}^{\sigma_A x, \sigma_A y}\big(t; 1,\frac{\sigma_b}{\sigma_A}\big).$$
\end{remark}


Next we show that infinite-depth, finite-width networks are solutions of SDEs where the vector fields are controlled by the input stream. We will then specialise to the case $\varphi = id$ and identify the limiting kernel with $\kappa_{id}$ from \cref{corollary:inhom} 


\subsection{The finite-width, infinite-depth regime}\label{subsec:inf-D-W-inhom}

Our next result states that when their width $N$ is fixed, these networks converge to a well defined distributional limit as their depth $M$ tends to infinity. In particular, in this limit, the random weights behave like white noise, and thanks to the careful choice of time scaling we have made, the limit is in fact a  zero-drift It\^{o} diffusion with diffusion coefficient depending on the driving path. 

\begin{theorem}\label{thm:inhom-inf-D-limit}
    Let $\{\D_M\}_{M \in \N}$ be a sequence of partitions of $[0,1]$ such that $|\D_M| \to 0$ as $M \to \infty$. Assume the activation function $\varphi$ is  Lipschitz and linearly bounded. Let $\rho_M(t) := \sup \{ s \in \D_M : s \leq t\}$. For any path $x \in \bX \cap C^{1,\frac{1}{2}}$, where $C^{1,\frac{1}{2}}$ denotes the set of $C^1$ paths with $\frac{1}{2}$-H\"{o}lder derivative, the $\R^N$-valued process $t \mapsto \cS^{M,N}_{\rho_M(t)}(x)$ converges in distribution, as $M \to \infty$, to the solution $\cS^{N}(x)$ of the following SDE    \begin{equation}\label{eqn:inhom_inf_depth_onepath}
        d\cS^{N}_t(x) = 
        \sum_{j=1}^d 
        \frac{\sigma_A}{\sqrt{N}} \Dot{x}^j_t dW^j_t \varphi(\cS^N_t(x)) + \sigma_b \dot{x}^j_t dB^j_t 
    \end{equation}
    with $\cS^N_0(x)=a$ and where $W^j \in \R^{N \times N}$ and $B^j \in \R^N$ are independent Brownian motions for $j\in\{1,...,d\}$.
\end{theorem}

\begin{proof}[{Idea of proof}]
   { The idea is proving that the finite difference scheme defining the inhomogeneous architecture gets closer and closer, as the mesh size of the partition becomes finer, to a Euler discretization of Equation (\ref{eqn:inhom_inf_depth_onepath}). 
    One then concludes with standard results which guarantee the convergence of Euler discretizations to the relative SDE's solution.}
\end{proof}

\begin{remark}
    Equation (\ref{eqn:inhom_inf_depth_onepath}) can be easily rewritten in more standard SDE form as follows (see Appendix for more details)
     \begin{equation*}
        d\cS^N_t(x) =  \sigma_x(t,\cS^N_t(x)) d Z_t
    \end{equation*}
    where $Z_t \in \R^{dN(N+1)}$ is a standard Brownian motion, independent from $a$ and $\sigma_x : [0,1] \times \R^N \to \R^{N \times dN(N+1)}$ is an input-dependent matrix valued function. 
\end{remark}

Passing directly to the infinite-width limit is not as easy as it could seem, Tensor Program arguments do not apply any longer since they are built for discrete layers and "collapse" in the continuous case we have to now work with. In simpler cases the limit can be found using McKean-Vlasov arguments as in \cite{hayou2022infinite} and we conjecture that similar results can be found in this more general setting. We leave such a study to future work.


In any case it is possible to directly prove this in the simplest case, when
$\varphi=id$. The result is proved in Appendix \ref{app:subsec:inf_d_lim_id}.



\section{Homogeneous controlled ResNets}\label{sec:Hom}


In this section we consider the more complex setting of networks in which the weights are shared across layers. We will see that this weight-sharing feature will yield limiting kernels governed by two-parameter, non-local partial differential differential equations. We will follow a similar structure as in the previous section, commenting on the crucial differences along the way. 

\subsection{The Model}

Define a \emph{randomly initialized, $1$-layer homogeneous controlled ResNet} $\Phi_\varphi^{M, N} : \bX \to \R$ as follows
\begin{equation*}
    \Phi_\varphi^{M, N}(x) := \sprod{\phi}{S^{M, N}_{t_M}(x)}_{\R^N}
\end{equation*}
where $\phi \in \R^N$ is the random vector  $[\phi]_\alpha \iid \mathcal{N}(0,\frac{1}{N})$, and where the random functions $S^{M,N}_{t_i} : \bX \to \R^N$ satisfy the following recursive relation
\begin{equation*}
    S^{M,N}_{t_{i+1}} = S^{M,N}_{t_i} + \sum_{k=1}^d \big( A_{k} \varphi(S^{M,N}_{t_i}) + b_{k} \big) \Delta x^k_{t_{i+1}}
\end{equation*}
with initial condition $S_{t_0} = a$ with $[a]_{\alpha} \iid \mathcal{N}(0,\sigma_a^2)$, and Gaussian weights $A_{k} \in \R^{N \times N}$ and biases $b_{k} \in \R^N$ sampled independently according to
\begin{equation*}
    [A_{k}]_{\alpha}^{\beta} \iid \mathcal{N}\left(0,\frac{\sigma_{A}^2}{N}\right), \quad 
    [b_{k}]_{\alpha} \iid \mathcal{N}\left(0,\sigma_b^2\right).
\end{equation*}




As done in the homogeneous case, we now study the limiting behavior of homogeneous controlled ResNets in the infinite-width-depth limit; as in the in-homogeneous case of the previous section, we will show that the limits commute.

\subsection{The infinite-width-depth regime}\label{subsec:body-phi-infWDlimit}

\Cref{thm:phi-SigKer} states that in this regime, these architectures converge weakly to GPs indexed on the path space $\bX$ with kernels satisfying a two-parameters differential equation.


\begin{theorem}\label{thm:phi-SigKer}
Let $\{\D_M\}_{M \in \N}$ be a sequence of partitions of $[0,1]$ such that $|\D_M| \to 0$ as $M \to \infty$. Let the activation function $\varphi$ be linearly bounded, absolutely continuous and with exponentially bounded derivative.
Then the following weak convergence holds 
\begin{equation}\label{eqn:limits_hom_kernels}
    \lim_{M \to \infty} \lim_{N \to \infty} \Phi_\varphi^{M, N} =
    \mathcal{GP}(0,\K_\varphi)
\end{equation}

where the positive semidefinite kernel $\K_\varphi : \bX \times \bX \to \R$ is defined for any two paths $x,y \in \bX$ as $\K_\varphi(x,y) = \K_\varphi^{x,y}(1, 1)$ where the function $\K_\varphi^{x, y} : [0,1] \times [0,1] \to \mathbb{R}$ is the unique solution of the following differential equation
\begin{equation}\label{eqn:hom_kernel_main}
\partial_s \partial_t \K^{x,y}_{\varphi} = 
    \Big[        
    \sigma_A^2 V_\varphi 
    \left(\Sigma_\varphi^{x, y}(s,t) \right) +    \sigma_b^2        
     \Big]  \sprod{\Dot{x}_{s}}{ \Dot{y}_{t}}
\end{equation}
where 
\[ 
    \Sigma_\varphi^{x, y}(s,t) = 
     \begin{pmatrix}
     \K_{\varphi}^{x, x}(s,s) &     \K_{\varphi}^{x, y}(s, t)\\
     \K_{\varphi}^{x, y}(s, t) & \K_{\varphi}^{y, y}(t, t)
     \end{pmatrix}
\]
and with initial conditions for any $s,t \in [0, 1]$
$$\K_{\varphi}^{x,y}(0,0)=\K_{\varphi}^{x,y}(s,0)=\K_{\varphi}^{x,y}(0,t)=\sigma_a^2.$$

\vspace{3pt}
If moreover $\varphi$ is Lipschitz the limits can be exchanged and
\[
\lim_{M \to \infty} \lim_{N \to \infty} \Phi^{M, N}_\varphi = \lim_{N \to \infty} \lim_{M \to \infty} \Phi^{M, N}_\varphi = \mathcal{GP}(0,\kappa_\varphi).
\]
\end{theorem}

\begin{remark}\label{rmk:inhom_scaling}
It is non-trivial to show not only that the problem is well-posed but even that equation (\ref{eqn:hom_kernel_main}) is well defined because the \say{instantaneous rate of change} $\partial_s \partial_t \K^{x,y}_{\varphi}(s,t)$ at times $s < t$ depends on the \say{past} values $\K_{\varphi}^{x, x}(s, s)$, on the \say{present} values $\K_{\varphi}^{x, y}(s, t)$ and on the \say{future} values $\K_{\varphi}^{y, y}(t, t)$. 
The nonlocal nature of these dynamics is such that it is \emph{a priori} not clear that the RHS of (\ref{eqn:hom_kernel_main}) even has meaning since the matrix $\Sigma_{\varphi}^{x,y}(s,t)$ could be not positive semidefinite and $V_{\varphi}$ is only defined on PSD matrices. 
\end{remark}

\begin{proof}[Idea of proof]

Similarly to Theorem \ref{thm:main-inhom-Kernels}, due to the complexity of the arguments, this result is proved in several steps. In in Appendix \ref{appendix:hom-WD-limit}.
The first step consists of showing that the infinite width limit is well defined for any choice of $\D_M$. 
This will be a GP defined by a kernel $\K_{\D_M \times \D_M}$ found as the terminal value of a finite difference scheme having the same form as that of a Euler discretization, on $\D_M \times \D_M$, of equation (\ref{eqn:hom_kernel_main}). 
The second step consists in proving that the kernels $\{\K_{\D_M \times \D_M}\}_M$ constitute, in a suitable metric, a Cauchy sequence as $|\D_M| \to 0$, that the limit is independent from the chosen sequence of partitions and that it does indeed uniquely solve equation (\ref{eqn:hom_kernel_main}).
The final step, proved in Appendix \ref{app:sub:hom_commut}, concerns the exchange of limits. After a characterization of the infinite-depth limits, we will prove that the distributional limit as depth goes to infinity is uniform in the width, thus we will be able to use the classical Moore-Osgood theorem to justify the exchange.
\end{proof}

When the activation function $\varphi$ is the identity, equation (\ref{eqn:hom_kernel_main}) reduces to a linear hyperbolic PDE. Upon inspection, we unveil a surprising link with the \emph{signature kernel}, a well-studied object in rough analysis corresponding to an inner product between two \emph{path-signatures}, and that was shown by \citet{salvi2021signature} to satisfy a similar PDE. This is the content of the next corollary.
\begin{corollary}
    Using the same notation and assumptions as in \cref{thm:phi-SigKer}, choosing $\varphi = id$ the limiting kernel satisfies the following identity
    \begin{equation*}
        \K_{id}^{x,y}(s,t) = \big( \sigma_a^2 + \frac{\sigma_b^2}{\sigma_A^2}\big) k_{sig}^{\sigma_A x,\sigma_A y}(s,t) - \frac{\sigma_b^2}{\sigma_A^2}
    \end{equation*}
    where $k_{sig}^{x,y}$ is the signature kernel from \cite{salvi2021signature} which for any two paths $x,y \in \bX$ and $s,t \in [0,1]$ satisfies the following linear hyperbolic PDE
    \begin{equation}\label{eq:pde_sigker}
        \partial_s\partial_t k_{sig}^{x,y} = \langle \dot x_s, \dot y_t \rangle k_{sig}^{x,y}
    \end{equation}
    with initial conditions $k_{sig}^{x,y}(s,0) = k_{sig}^{x,y}(0,t) = 1$.
\end{corollary}

In other words we have unveiled a novel family of kernels indexed on continuous paths which generalizes the signature kernel in \cite{salvi2021signature}. We name this new class of kernels \emph{neural signature kernels}. We note that this generalization is done directly at the level of the driving PDE unlike the extensions studied in \cite{cass2021general} which use a different inner product structure on the space where signatures live.



\begin{remark}
    Analogously to the inhomogeneous case, the parameters $\sigma_A$ and $\sigma_b$ defining the neural signature kernels governed by the dynamics in equation (\ref{eqn:hom_kernel_main}) satisfy the following path-rescaling symmetry
    $$\K_{\varphi}^{x,y}(s,t; \sigma_A, \sigma_b) = 
    \K_{\varphi}^{\sigma_A x, \sigma_A y}\big(s,t; 1, \frac{\sigma_b}{\sigma_A}\big)$$
    as shown in Lemma \ref{app:lemma:hom_simpler_form} in the appendix.
\end{remark}

\begin{figure*}
    \centering
    \includegraphics[trim={0.cm 0cm 0 0},clip,width=.65\columnwidth]{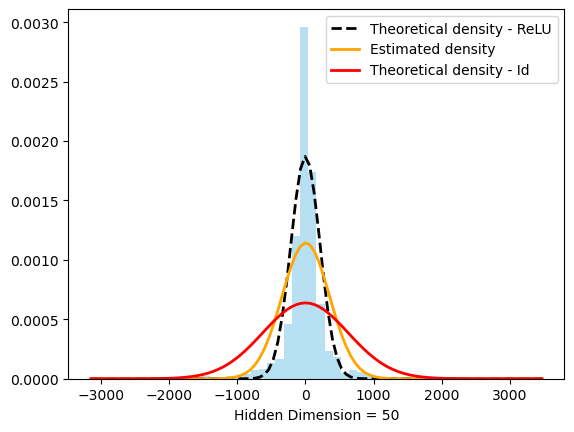}
    \hfill
    \includegraphics[trim={0.cm 0cm 0 0},clip,width=.65\columnwidth]{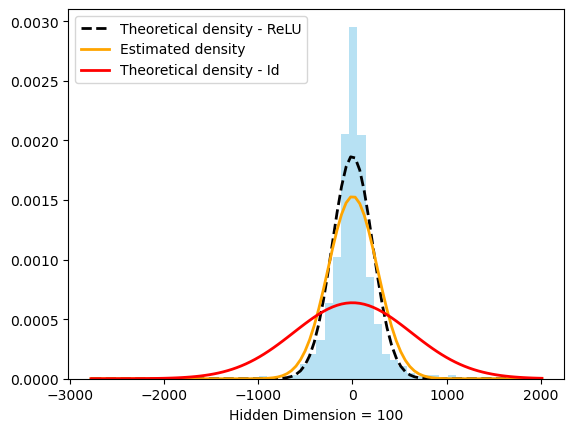}
    \hfill
    \includegraphics[trim={0.cm 0cm 0 0},clip,width=.65\columnwidth]{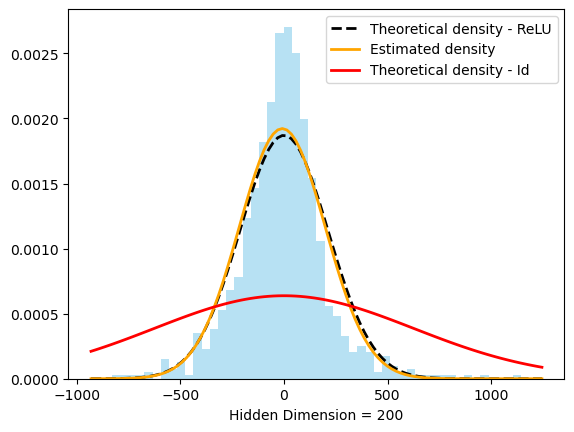}
    \caption{Fixed input $y(t) := cos(15t) + 3e^t$.  Histogram of $700$ independent realizations of $N^{-1}\sprod{S_1(y)}{S_1(y)}_{\R^N}$ for \emph{ReLU}-RandomizedSignatures $S(y) \in \R^{N}$, $N \in \{50, 100, 200\}$, plotted against Closest Gaussian Fit, ReLU-NeuralSigKer and id-NeuralSigKer.}
    \label{fig:hidden_200}
\end{figure*}

\begin{remark}
    For non-linear activation functions $\varphi$, the neural signature kernel non-linear PDE (\ref{eqn:hom_kernel_main}) and the $id$-neural signature kernel linear PDE (\ref{eq:pde_sigker}) might in principle admit the same solution, which would mean essentially that linear and non-linear controlled ResNets behave in the same way in the infinite-width-depth regime. In in Figure \ref{fig:hidden_200} we show empirically that this is not the case in general, by comparing the limiting empirical distributions for $\varphi=id$ and $\varphi=ReLU$.
\end{remark}


\subsection{The finite-width, infinite-depth regime}\label{subsec:inf-D-W-hom}

Our next result states that when their width $N$ is fixed, homogeneous controlled ResNets converge in distribution, in $[0,1]$, to a Neural CDE with random vector fields.  

\begin{theorem}\label{thm:randomised_sigs}
    Let $\{\D_M\}_{M \in \N}$ be a sequence of partitions of $[0,1]$ such that $|\D_M| \to 0$ as $M \to \infty$. Assume the activation function $\varphi$ is Lipschitz and linearly bounded. 
    Let $x \in \bX$ and let $\rho_M(t) := \sup \{ s \in \D_M : s \leq t\}$.
    Then, the $\R^N$-valued process $t \mapsto S^{M,N}_{\rho_M(t)}(x)$ converges in distribution\footnote{As random variables with values in $L^{\infty}([0,1];\R^N)$.}, as $M \to \infty$, to the solution $S^{N}(x)$ of the following Neural CDE
    \begin{equation}\label{eqn:Randomized-Sigs}
        dS^{N}_t(x) = \sum_{j=1}^d \big( A_j \varphi(S^{N}_t(x))  + b_j \big) dx_t^j
    \end{equation}
    where $A_j \in \R^{N \times N}$ and $b_j \in \R^N$ are sampled according in the definition of the homogeneous controlled ResNet. 
\end{theorem}

\begin{proof}[{Idea of proof}]
    If we fix $a$, the $A_k$s and the $b_k$s to be the same for all $\D_M$ then we have uniform convergence by classical results.
    The rate of convergence can be bounded with some constants depending on the entries of $a$, $A_k$, $b_k$ and which, thanks to Gaussianity, have finite expectation.
    It is just a matter of applying the classical \emph{portmanteau} lemma to conclude.
\end{proof}

\begin{remark}
    This can be naturally extended in order to take into consideration the joint distribution for different input choices.
\end{remark}


The solutions to equation (\ref{eqn:Randomized-Sigs}) have been informally introduced in \cite{cuchiero2021expressive, TeichmannMarketAnomaly} as lower dimensional approximations of path-signatures, and have been dubbed by the authors \emph{randomized signatures}.

Taking directly the infinite width limit is once again far from trivial, reasoning \emph{à la} Tensor Program quickly collapse and there is no clear possible future path corresponding to the McKean-Vlasov ideas for the inhomogeneous case.
The problem is that the randomness is in the vector fields themselves and not in the driving paths, courtesy of the cross-layer dependencies in the homogeneous networks. This is why it's necessary to sidestep the problem by proving the existence of uniform convergence bounds.
 
\subsection{The infinite-depth-then-width regime: $\varphi = id$}

As anticipated, contrary to the inhomogeneous case, in the current homogeneous setting, when $\varphi$ is the identity, we are able to prove directly that the limits in \cref{eqn:limits_hom_kernels} commute as well as explicit convergence bounds.

\begin{theorem}\label{thm:sig_id}
If $\varphi = id$ and for any $x,y\in\bX$ 
\begin{equation*}
    \frac{1}{N}\sprod{S^N_s(x)}{S^N_t(y)}_{\R^N} 
    \xrightarrow[N \to \infty]{\mathbb{L}^2}
    \K^{x,y}_{id}(s,t)
\end{equation*}
on $[0,1]^2$. Moreover the convergence is of order $\mathcal{O}(\frac{1}{N})$.
\end{theorem}

\begin{figure}[h]
    \centering
    \includegraphics[trim={0.cm 0cm 0 0},clip,width=\columnwidth]{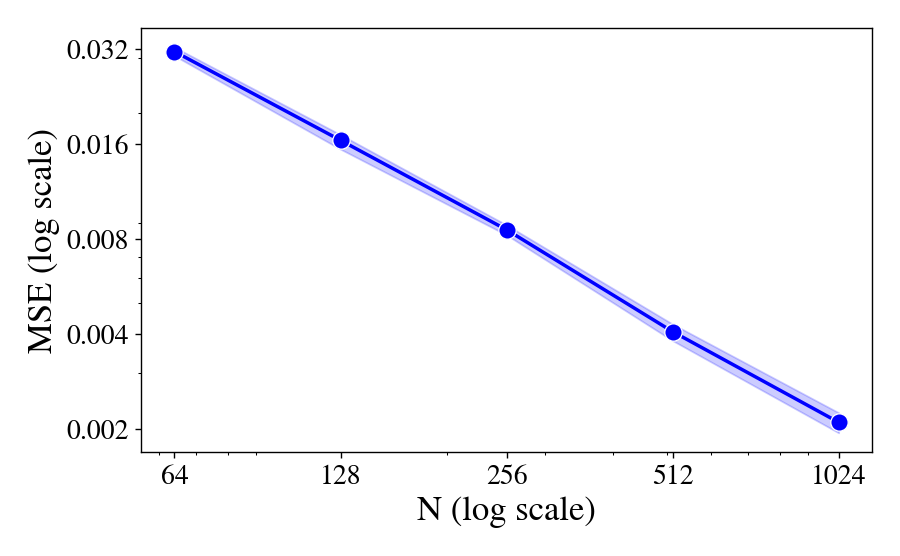}
    \caption{Mean squared error of $\frac{1}{N}\sprod{S^N_1(x)}{S^N_1(y)}_{\R^N}$, the estimator of $\K^{x,y}_{id}(1,1)$, as a function of the width $N$ on a logarithmic scale. Standard deviations were obtained by repeating the experiment $5$ times.}
    \label{fig:error_rate}
\end{figure}

\begin{figure*}
    \centering
    \includegraphics[width=\textwidth]{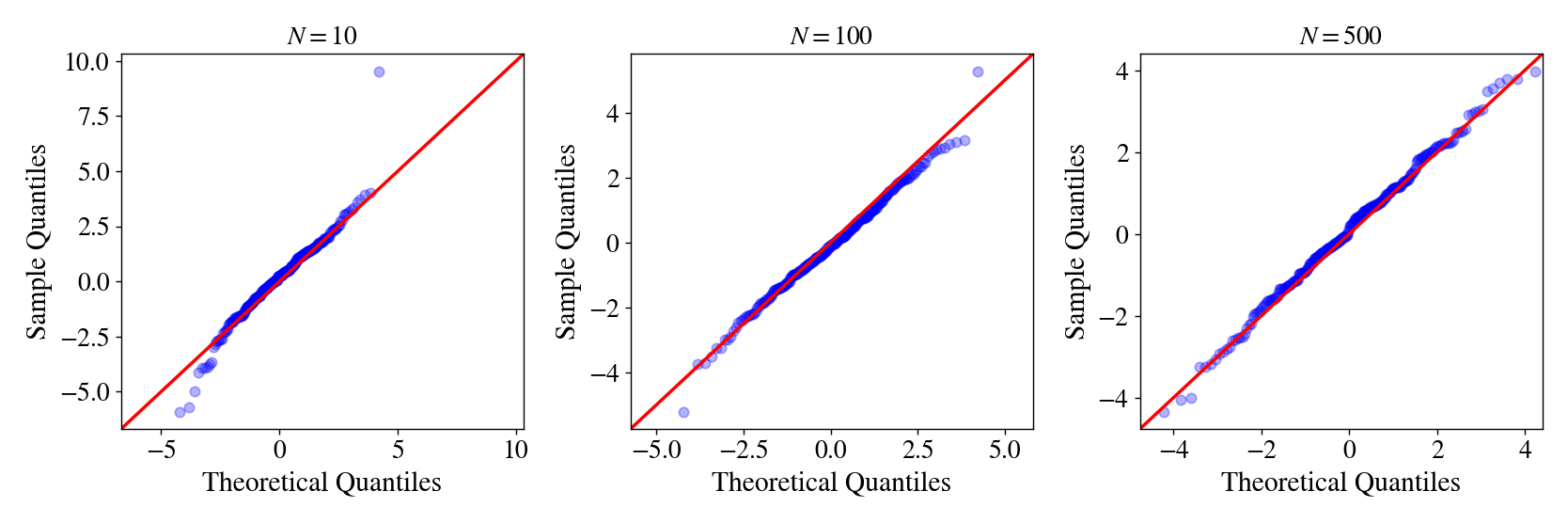}
    \caption{Empirical quantiles against the theoretical quantiles of $\mathcal{N}(0,\mathcal{K}^{x,x}_\varphi(1))$ for $N\in\{10,100,500\}$ with $\varphi=\text{ReLU}$}
    \label{fig:qq_plots}
\end{figure*}



\section{Numerics}\label{sec:experiments}

In this section, we first illustrate theoretical results established in \Cref{sec:Hom} and then outline numerical considerations to scale the computation of signature kernels. 

\subsection{Convergence of homogeneous controlled ResNets} 

We start by illustrating the convergence in distribution of a homogeneous controlled ResNet to a GP endowed with neural signature kernel as per \Cref{thm:phi-SigKer}. To this aim, we consider a homogeneous ResNet $\Phi_\varphi^{M, N}$ with activation function $\varphi=\text{ReLU}$, and $(\sigma_a, \sigma_A, \sigma_b)=(0.5, 1., 1.2)$. For $R=250$ realizations of the weights and biases, we run the model on a $2$-dimensional path $x:t\mapsto(\sin(15t),\cos(30t) + 3e^t)$ observed at $100$ regularly spaced time points in $[0,1]$. We then verify that, as $N$ increases, $\Phi_\varphi^{M,N}(x)$ converges to a Gaussian random variable with mean zero and variance $\K_\varphi(x,x)$. This limiting variance is computed by solving \cref{eqn:hom_kernel_main} on a fine discretization grid. As it can be observed on \Cref{fig:qq_plots} the Gaussian fit for this one-dimensional marginal gets better as $N$ increases. Further results can be found in the appendix.

We then provide empirical evidence for the order of convergence provided in \Cref{thm:sig_id}. Here, we consider a linear homogeneous ResNet, controlled by $x$ and $y$, two sample paths from a zero-mean GP with RBF kernel $r_{\text{RBF}}(s,t)=\exp{(-5(s-t)^2)}$ with $50$ observation points in $[-2, 2]$. Similarly to the previous setup, we run the model with $M=250$ different random initializations to estimate the mean squared error $\mathbb{E}[(\frac{1}{N}\langle S^N_1(x),S^N_1(y)\rangle - \K_{id}(x,y))^2]$ increasing the width $N$. Our empirical results, as displayed on \Cref{fig:error_rate}, align with the theoretical convergence rate.



\subsection{Scaling signature kernels}

The signature kernel of two paths is typically computed by approximating the solution of the PDE in (\ref{eq:pde_sigker}) on a $2$-dimensional time grid, which scales quadratically with the discretization step of the solver. Although an efficient numerical scheme leveraging GPU computations to update the solution at multiple time points on the grid in parallel has been proposed in \citet{salvi2021signature}, the maximum number of threads in a GPU block imposes a hard limit on the discretization step of the solver, limiting the applicability of signature kernel methods to long time series. \Cref{thm:sig_id} offers a new way to compute the signature kernel by solving two CDEs linearly in time instead of one PDE quadratically in time; one would first run a wide and infinite-depth ResNet on the two control paths of interest, and then compute the (rescaled) dot-product between the outputs of the penultimate layer. This approach allows for more flexibility regarding the choice of path interpolations and numerical solvers, as several options are made readily available in dedicated python packages such as $\mathsf{torchcde}$ \cite{kidger2020neural}. Next, we describe possible ways to increase further the scalability of this approach.

\paragraph{Log-ODE method} To further improve scalability of Neural CDEs for long time series \citet{morrill2021neural} made use of the so-called \emph{log-ODE scheme} to forward-solve the differential equation on much larger time intervals than the ones that would be expected given the sampling rate or length of the data. We leave the investigation of this numerical scheme for computing signature kernels as future work.

\paragraph{Sparse random matrices} The forward pass of a ResNet involves several ($M\times d$ where $M$ is the number of time steps, and $d$ the dimension of the input path) matrix-vector multiplications where the entries of each $N$-by-$N$ matrix are Gaussian distributed. As remarked in \cite{dong2020reservoir}, in the context of random RNNs, to speed-up these computations, the dense weight matrices can be replaced by structured random matrices given by the products of random (binary) diagonal matrices
and Walsh-Hadamard matrices. The complexity of the matrix-vector product can be reduced to $\mathcal{O}(N\log N)$ leveraging the fast Hadamard transform algorithm (without sampling the Walsh-Hadamard matrices).  

\paragraph{Random Fourier features} Several machine learning use cases of the signature kernel have provided empirical evidence that embedding the input paths pointwise in time in a feature space can be beneficial to increase the performance of kernel methods on sequential data. In particular, when the paths evolve in a Euclidean space, the RBF kernel often turns out to be a good choice. Although this embedding is infinite-dimensional, random Fourier features \cite{rahimi2007random} make it possible to approximate it by a finite-dimensional one. One could then investigate randomly initialized ResNets, controlled by sequences of such approximate embeddings.

\section{Conclusion and future work}

In this paper we considered controlled ResNets defined as Euler-discretizations of  Neural CDEs. We showed that in both the infinite-depth-then-width and in the infinite-width-then-depth limit, these converge weakly to the same GP indexed on path space endowed with neural signature kernels satisfying certain (possibly non-linear) PDEs varying according to the choice of activation function $\varphi$. In the special case where $\varphi$ is the identity, we showed that the equation reduces to a linear PDE and the limiting kernel agrees with the signature kernel. In this setting, we also provided explicit convergence rates. Finally, we showed that in the infinite-depth regime, finite-width controlled ResNets converge in distribution to Neural CDEs with random vector fields which are either time-independent and Gaussian, if the system is homogeneous, or behave like a matrix-valued Brownian motion, if the system is inhomogeneous.




We believe that a rigorous investigation of the functional analytic properties of the \emph{reproducing kernel Hilbert spaces} (RKHSs) associated to the new family of neural signature kernels is also a compelling future research direction. In particular, it would allow to build an understanding of the expressivity and generalization properties of these  kernels. 

In the homogeneous setting, the vector fields are constant functions while in the in-homogeneous setting they are described by white noise. Investigating the intermediate regularity cases is an interesting avenue for future research; for example considering matrices and biases sampled from of Fractional Brownian Motion increments with Hurst exponent $H \in [0,1]$ (the inhomogeneous case corresponds to the case $H = 0.5$ while the homogeneous one to $H = 1$).




Last but not least, establishing expressions and analyzing the associated \emph{Neural Tangent Kernels} (NTK) \cite{jacot2018neural, yang2020tensor} would provide quantitative insights on the training mechanism of Neural CDEs by gradient descent.

{All the experiments presented in this paper are reproducible following the code at \url{https://github.com/MucaCirone/NeuralSignatureKernels}}

\section{Aknowledgements}
The authors would like to thank Thomas Cass, James-Micheal Lehay and David Villringer for helpful discussions.

NMC was supported by EPSRC Centre for Doctoral Training in Mathematics of Random Systems: Analysis, Modelling and Simulation (EP/S023925/1) and the Department of Mathematics, Imperial College London, through a Roth Scholarship.
ML was supported by the EPSRC grant EP/S026347/1.


\bibliography{references}
\bibliographystyle{icml2023}

\newpage


\appendix
\onecolumn

\section{Preliminaries}

In this section we are going to state some preliminary results and considerations which we are going to refer to through the entirety of the text.

 Throughout the paper we fix a probability space $(\Omega, \mathcal{F}, \mathbb{P})$.

\subsection{Assumptions on path regularity}
Since we are mainly interested in studying time-series, our data space will be a space of paths, more specifically we are going to consider the space
\[
    \bX := \{ x \in C^0([0,1];\R^d) : x(0) = 0, \exists \Dot{x} \in L^2([0,1]; \R^d) \}
\]
\emph{i.e.} $\bX$ is the space of continous paths which have a square integrable derivative. 
$\bX$ is the closed subset of the Sobolev space $\big( W^{1,2}([0,1]) \big)^d$ made of those functions starting at the origin.

Note that every path $x \in \bX$ can be uniquely written as 
\[
    x_t = \int_0^t \Dot{x}_s ds 
\]
thus it naturally corresponds to the space $L^2([0,1]; \R^d)$ trough the identification $x \mapsto \Dot{x}$.
This identification gives the space the natural norm $\norm{x}_{\bX} = \norm{\Dot{x}}_{L^2}$.

This norm is equivalent to the induced norm from $\big( W^{1,2}([0,1]) \big)^d$ since 
\begin{equation*}
    \begin{gathered}
        \norm{\Dot{x}}^2_{L^2} \leq \norm{x}^2_{W^{1,2}} = \int_0^1 |x_s|^2 ds + \norm{\Dot{x}}^2_{L^2}
        =  \int_0^1 |\int_0^s \Dot{x}_r dr|^2 ds + \norm{\Dot{x}}^2_{L^2} \\
        \leq \int_0^1 \int_0^1 |\Dot{x}_r|^2 dr ds + \norm{\Dot{x}}^2_{L^2} = 2\norm{\Dot{x}}^2_{L^2}
    \end{gathered}
\end{equation*} for $x \in \bX$, hence $(\bX, \norm{\cdot}_{\bX})$ is a Banach space. Moreover one can easily see that every $x \in \bX$ has bounded variation and 
\[
\norm{x}_{1-var, [0,1]} = \int_0^1 |\Dot{x}_t| dt \leq \sqrt{\int_0^1 |\Dot{x}_t|^2 dt} = \norm{x}_{\bX}
\]

Given $n$ paths $\X = \{x_1, \dots, x_n\} \subset \bX$ and functions $f: \bX \to \R$, $G: \bX \times \bX \to \R$  we will write $f(\X) \in \R^n$ for the vector $[f(\X)]_{\alpha} = f(x_{\alpha})$ and $G(\X, \X) \in \R^{n \times n}$ for the matrix $[G(\X,\X)]_{\alpha}^\beta = G(x_{\alpha}, x_{\beta})$ for any $\alpha, \beta \in \{1,..,n\}$.

\subsection{Assumptions on the activation function}

Now we are going to state the main assumptions on the activation function $\varphi: \R \to \R$ and prove some important technical results of which we will frequently make use in the following sections. 
We will particularly be interested in how regularity assumptions made on $\varphi$ influence the regularity of the expectations of Equation (\ref{eqn:hom_kernel_main}) and (\ref{eqn:inhom_kernel}). 

Here are the crucial assumptions we make on the activation:

\begin{assumption}\label{assumption:linear-bound}
  The activation function $\varphi: \R \to \R$ is linearly bounded  \emph{i.e.} such that
  there exist some $M > 0$ such that  $| \varphi (x) | \leq M (1 + | x |)$.
\end{assumption}

\begin{assumption}\label{assumption:exp-bounded-der}
  The activation function $\varphi: \R \to \R$ is absolutely continuous and with exponentially bounded derivative.
\end{assumption}

\vspace{10pt}
\begin{lemma}\label{app:lipR_to_lipRn}
    If the activation function $\varphi: \R \to \R$ is $K$-Lipschitz then its componentwise extension $\varphi: \R^N \to \R^N$ is $K$-Lipschitz too.
    If the activation function $\varphi: \R \to \R$ is $M$-linealy-bounded then its componentwise extension is $\sqrt{2N}M$-linealy-bounded.
\end{lemma}
\begin{proof}
    Regarding the first proposition, for $x,y \in \R^n$ we have
    \begin{equation*}
        \begin{gathered}
            |\varphi(x)-\varphi(y)|^2_{\R^N} = \sum_{i=1}^N |\varphi(x_i)-\varphi(y_i)|^2 \leq \sum_{i=1}^N K^2|x_i-y_i|^2 = K^2 |x-y|^2_{\R^N}
        \end{gathered}
    \end{equation*}

    Concerning the second, notice how
    \begin{equation*}
        \begin{gathered}
            |\varphi(x)|^2_{\R^N} = \sum_{i=1}^N |\varphi(x_i)|^2 \leq \sum_{i=1}^N M^2(1+|x_i|)^2 \leq 2M^2 \sum_{i=1}^N (1+|x_i|^2)
            \\
            = 2M^2(N +|x|^2_{\R^n}) \leq 2NM^2(1+|x|^2_{\R^N})
        \end{gathered}
    \end{equation*}
    thus $|\varphi(x)|_{\R^N} \leq \sqrt{2N}M(1+|x|_{\R^N})$ since $\sqrt{1 + \epsilon} \leq 1 + \sqrt{\epsilon}$ for all $\epsilon \geq 0$.
\end{proof}

\begin{remark}\label{app:rem:linear_bound_N}
    Note how we have proved also that, under the linear boundedness assumption, using the same final bound
    \begin{equation}
        |\varphi(x)|_{\R^N} \leq \sqrt{2}M(\sqrt{N} + |x|_{\R^N})
  \end{equation}
    
\end{remark}

\subsection{Positive semidefinite matrices and the map $V_\varphi$}

\begin{definition}
    Let $PSD_2$ denote the set of $2\times 2$ positive semidefinite matrices 
    \[
    PSD_2 := \{ \Sigma \in \R^{2 \times 2} :  \Sigma =  \Sigma^{T} ; ([\Sigma]_1^2)^2 \leq [\Sigma]_1^1 [\Sigma]_2^2 ; 0 \leq [\Sigma]_1^1 \wedge [\Sigma]_2^2 \}
    \]
    For a fixed $R > 0$ we define the space
    \[
        PSD_2(R) := \{ \Sigma \in PSD_2 : \frac{1}{R} \leq [\Sigma]_1^1, [\Sigma]_2^2 \leq R\}
    \]
\end{definition}

\begin{lemma}\label{lemma:split_on_PSD_R}
Under {\tmstrong{Assumption \ref{assumption:exp-bounded-der}}} and for any $R>0$ the function $V_{\varphi} : \tmop{PSD}_2 (R) \rightarrow \R$ defined for any $\Sigma \in \tmop{PSD}_2 (R)$ as
  \begin{eqnarray*}
    V_{\varphi}(\Sigma) = \mathbb{E}_{(Z_x, Z_y) \sim \mathcal{N} (0,
    \Sigma)} [\varphi (Z_x) \varphi (Z_y)] & 
  \end{eqnarray*}
 is $\kappa_R$-Lipschitz for some $\kappa_R > 0$, i.e.
  \[
  |V_{\varphi} (\Sigma) - V_{\varphi} (\tilde{\Sigma})|  \leq  \kappa_R \| \Sigma -
    \tilde{\Sigma} \|_{\infty} 
  \]
\end{lemma}

\begin{proof}
    This is the content of Theorem F.4 in \cite{novak2018bayesian}.
\end{proof}

\begin{proposition}
  Under {\tmstrong{Assumption \ref{assumption:linear-bound}}}, there exists a {positive} constant $\tilde{M}>0$ such that 
  \[ | V_{\varphi} (\Sigma) | \leq \tilde{M} \left( 1 + \sqrt{[\Sigma]_1^1} \right)
     \left( 1 + \sqrt{[\Sigma]_2^2} \right) \]
\end{proposition}

\begin{proof}
  In fact given a PSD matrix $\Sigma \in \mathbb{R}^{2 \times 2}$ we have
  \[ \Sigma = A A^T \]
  where
  \[ A = \left(\begin{array}{c}
       \alpha, 0\\
       \beta \gamma, \beta \sqrt{1 - \gamma^2}
     \end{array}\right) \]
  with
  \[ \alpha = \sqrt{[\Sigma]_1^1}, \beta = \sqrt{[\Sigma]_2^2} \infixand
     \gamma = \frac{[\Sigma]_1^2}{\sqrt{[\Sigma]_1^1 [\Sigma]_2^2}} \]
  Then {it can be easily observed that}
  \begin{eqnarray*}
    \mathbb{E}_{(Z_x, Z_y) \sim \mathcal{N} (0, \Sigma)} [\varphi (Z_x)
    \varphi (Z_y)] & = & \mathbb{E}_{Z \sim \mathcal{N} (0, \tmop{Id})}
    \left[ \varphi (\alpha Z_1) \varphi \left( \beta \left( \gamma Z_1 +
    \sqrt{1 - \gamma^2} Z_2 \right) \right) \right]
  \end{eqnarray*}
  so that
\begin{align*}
    | V_{\varphi} (\Sigma) |  &=  \left| \mathbb{E}_{Z \sim \mathcal{N} (0,
    \tmop{Id})} \left[ \varphi (\alpha Z_1) \varphi \left( \beta \left( \gamma
    Z_1 + \sqrt{1 - \gamma^2} Z_2 \right) \right) \right] \right|
    \\
     &\leq  \mathbb{E}_{Z \sim \mathcal{N} (0, \tmop{Id})} \left[ |
    \varphi (\alpha Z_1) | \left| \varphi \left( \beta \left( \gamma Z_1 +
    \sqrt{1 - \gamma^2} Z_2 \right) \right) \right| \right]
    \\
     &\leq  M^2 \mathbb{E}_{Z \sim \mathcal{N} (0, \tmop{Id})} \left[
    (1 + | (\alpha Z_1) |) \left( 1 + \left| \left( \beta \left( \gamma Z_1 +
    \sqrt{1 - \gamma^2} Z_2 \right) \right) \right| \right) \right]
    \\
     &=  M^2 \Big[ 1 + \alpha \mathbb{E}_{Z \sim \mathcal{N} (0,
    \tmop{Id})} [| Z_1 |] + \beta \mathbb{E}_{Z \sim \mathcal{N} (0,
    \tmop{Id})} \left[ \left| \gamma Z_1 + \sqrt{1 - \gamma^2} Z_2 \right|
    \right] \\
    & \hspace{1cm} + \alpha \beta \mathbb{E}_{Z \sim \mathcal{N} (0, \tmop{Id})}
    \left[ | Z_1 | \left| \gamma Z_1 + \sqrt{1 - \gamma^2} Z_2 \right| \right]
    \Big]
    \\
     &\leq  M^2 \Big[ 1 + (\alpha + \beta | \gamma |) \mathbb{E}_{Z
    \sim \mathcal{N} (0, \tmop{Id})} [| Z_1 |] + \beta \sqrt{1 -
    \gamma^2} \mathbb{E}_{Z \sim \mathcal{N} (0, \tmop{Id})} [| Z_2 |] \\
    & \hspace{1cm} + \alpha \beta | \gamma | \mathbb{E}_{Z \sim \mathcal{N} (0,
    \tmop{Id})} [| Z_1 |^2] + \alpha \beta \sqrt{1 - \gamma^2} \mathbb{E}_{Z
    \sim \mathcal{N} (0, \tmop{Id})} [| Z_1 Z_2 |] \Big]
    \\
     &\leq  \bar{M} \left( 1 + \alpha + \beta | \gamma | + \beta \sqrt{1 -
    \gamma^2} + \alpha \beta | \gamma | + \alpha \beta \sqrt{1 - \gamma^2}
    \right)
\end{align*}
   
{where the first inequality follows from Jensen's inequality, the second from Assumption \ref{assumption:linear-bound}, the third from the triangle inequality, and the fourth from the fact that $\mathcal{N}(0,Id)$ has finite moments with $\bar{M}$ a constant incorporating $M^2$ and these bounds}.

Using ${\gamma^2 \leq 1}$, for some constant $\tilde{M} $ one has
  \begin{eqnarray*}
    | V_{\varphi} (\Sigma) | & \leq & \tilde{M} (1 + \alpha + \beta + \alpha \beta) =
    \tilde{M} (1 + \alpha) (1 + \beta)
  \end{eqnarray*}
  
\end{proof}

\vspace{10pt}
We end the section showing the explicit characterization of the maps $V_{\varphi}$ for some selected\footnote{by the availability in the literature.} activation functions. We write $V_{\varphi}$ with the obvious meaning.

\begin{proposition}\label{app:charact-Phi}
    Defining $\gamma(\Sigma) := \frac{[\Sigma]_1^2}{\sqrt{[\Sigma]_1^1[\Sigma]_2^2}}$ we have
    \begin{equation}
        \begin{gathered}
            V_{id}(\Sigma) = [\Sigma]_1^2 = [\Sigma]_2^1
            \\
            V_{ReLU}(\Sigma) = \frac{1}{2\pi}\left(\pi + \frac{\sqrt{1- \gamma(\Sigma)^2}}{\gamma(\Sigma)}  - arccos(\gamma(\Sigma))\right) [\Sigma]_1^2
            \\
            V_{erf}(\Sigma) = \frac{2}{\pi} arcsin \left( \frac{[\Sigma]_1^2}{\sqrt{(0.5 + [\Sigma]_1^1)(0.5 + [\Sigma]_2^2)}} \right)
        \end{gathered}
    \end{equation}
\end{proposition}

\begin{proof}
    See \cite{yang2019wide}[Facts B.2, B.3].
\end{proof}

\section{Proofs for inhomogeneous controlled ResNets}

In this section of the appendix we are going to prove all the results stated for the inhomogeneous case. 
The section will be subdivided in three main parts: in the first we consider the infinite-width-then-depth limit, in the second the infinite-depth-then-width one, in the final one we prove the commutativity of the integrals.

We start by recalling the defintion of the model.
\begin{definition}[Inhomogeneous controlled ResNets]
    Let $\D_M = \{0 = t_0 < \cdots < t_M = 1\}$ be a partition, $N\in \mathbb{N}$ be the width, and $\varphi: \R \to \R$ an activation function.
Define a \emph{randomly initialised, $1$-layer inhomogeneous controlled ResNet} $\Psi^{M, N} : \bX \to \R$ as follows
\begin{equation*}
    \Psi^{M, N}_\varphi(x) := \sprod{\psi}{\cS^{M,N}_{t_M}(x)}_{\R^N}
\end{equation*}
where $\sprod{\cdot}{\cdot}_{\R^N}$ is the $L^2$ inner product on $\R^N$, $\psi \in \R^N$ is a random vector with entries $[\psi]_\alpha \iid \mathcal{N}(0,\frac{1}{N})$, and where the random functions $\cS^{M,N}_{t_i} : \bX \to \R^N$ satisfy the following recursive relation
\begin{equation*}
\begin{gathered}
    \cS^{M,N}_{t_{i+1}} = \cS^{M,N}_{t_i} + \sum_{j=1}^d \big( A_{j,i} \varphi(\cS^{M,N}_{t_i}) + b_{j,i} \big) \Delta x^j_{t_{i+1}}
    \\
    \cS^{M,N}_{t_0} = a \quad \text{and}  \quad \Delta x^j_{t_i} = (x^j_{t_i}-x^j_{t_{i-1}})
\end{gathered}
\end{equation*}
for $i=0,...,M$, with initial condition $[a]_{\alpha} \iid \mathcal{N}(0,\sigma_a^2)$, and Gaussian weights $A_{k,l} \in \R^{N \times N}$ and biases $b_{k,l} \in \R^N$ sampled independently according to
\begin{equation*}
    [A_{j,i}]_{\alpha}^{\beta} \iid \mathcal{N}\left(0,\frac{\sigma_{A}^2}{N\Delta t_i}\right), \quad 
    [b_{j,i}]_{\alpha} \iid \mathcal{N}\left(0,\frac{\sigma_b^2}{\Delta t_i}\right)
\end{equation*}
with time step $\Delta t_i = (t_i- t_{i-1})>0$ and parameters $\sigma_{a}, \sigma_{A} > 0$ and $\sigma_{b} \geq 0$.
\end{definition}


\subsection{The infinite-width-then-depth regime}\label{app:sec:inhom-inf-W-D-limit}

The main goal in this subsection is to prove the first part of Theorem \ref{thm:main-inhom-Kernels}, which we restate here:

\begin{theorem}\label{thm:main-inhom-Kernels_appendix}
Let $\{\D_M\}_{M \in \N}$ be a sequence of partitions of $[0,1]$ such that $|\D_M| \downarrow 0$. Let the activation function $\varphi: \R \to \R$ be linearly bounded, absolutely continuous and with exponentially bounded derivative.
For any subset of paths $\X = \{x_1, \dots, x_n\} \subset \bX$ the following convergence in distribution holds
\begin{equation}\label{eqn:limits_inhom_kernels_appendix}
    \lim_{M \to \infty} \lim_{N \to \infty} \Psi_{\varphi}^{M, N}(\X) =
    \mathcal{N}(0,\kappa_\varphi(\X, \X))
\end{equation}
where the positive semidefinite kernel $\kappa_\varphi : \bX \times \bX \to \R$ is defined for any two paths $x,y \in \bX$ as $\kappa_\varphi(x,y) = \kappa_\varphi^{x,y}(1)$,  where $\kappa_\varphi^{x,y} : [0,1] \to \R$ is the unique solution of the following differential equation
\begin{equation}\label{eqn:inhom_kernel_appendix}
      \partial_t\kappa_{\varphi}^{x,y} =   
      \Big[ 
      \sigma_A^2 V_\varphi 
    \left(\begin{pmatrix}
       \kappa^{x,x}_\varphi & \kappa_{\varphi}^{x, y}\\
       \kappa_{\varphi}^{x, y} & \kappa_{\varphi}^{y, y}
     \end{pmatrix}\right) 
     +\sigma_b^2 
      \Big] 
     \sprod{\Dot{x}_t}{ \Dot{y}_t}_{\R^d}
\end{equation}
with initial condition $\kappa_{\varphi}^{x,y}(0)=\sigma_a^2$.
\end{theorem}

\begin{remark}
    As mentioned in the paper, this convergence is equivalent to say that the sequence of random functions $\Psi_{\varphi}^{M, N} : \bX \to \R$ convergence weakly to a GP with zero mean function and with kernel $\kappa_\varphi$.
\end{remark}

We split the lengthy proof in two parts:
\begin{enumerate}
    \item The infinite width-convergence of the finite dimensional distributions 
    \begin{equation*}
        \Psi_{\varphi}^{M, N}(\X) \xrightarrow[]{N \to \infty} \mathcal{N}(0, \kappa_{\D_M}^{}(\X,\X))
    \end{equation*}
    for a fixed depth $M$ to those of a Gaussian process $\mathcal{GP}(0,\kappa_{\D_M}^{})$ defined by a kernel computed as the final value of a finite difference scheme on the partition $\D_M$. This will be done using Tensor Programs \cite{yang2019wide}, and is the content of subsection \ref{subsubsec:discrete-inhom-Kernel}.
    
    \item The infinite-depth (uniform) convergence of the discrete kernels $\kappa_{\D_M}^{}$ to a limiting kernel $\kappa_{\varphi}$ which solves the differential equation (\ref{eqn:inhom_kernel_appendix}).  This will be done in subsection \ref{subsubsec:conv-inhom-Kernel}.
\end{enumerate}

\subsubsection{Infinite-width limit with fixed depth}\label{subsubsec:discrete-inhom-Kernel}

In the next theorem, we show that finite width inhomogeneous ResNets converge in distribution to GPs with discrete kernels satisfying some difference equations.

\begin{theorem}\label{thm:inhom-discrete-phi-SigKer}
Let $\D_M = \{0 = t_0, \dots, t_i, \dots, t_M=1\}$ be a fixed partition of $[0,1]$. Let the activation function $\varphi: \R \to \R$ be linearly bounded
\footnote{in the sense that $\exists C > 0.$ such that $|\phi(x)| \leq C(1+|x|)$.}.
For any subset $\X = \{x_1, \dots, x_n\} \subset \bX$ the following convergence in distribution holds
\begin{equation*}
    \lim_{N \to \infty} \Psi_{\varphi}^{M, N}(\X) 
    =
    \mathcal{N}(0,\kappa_{\D_M}(\X,\X))
\end{equation*}
where for any two paths $x,y \in \bX$ the discrete kernel $\kappa_{\D_M}(x,y):=\kappa_{\D_M}^{x,y}(1)$ where $\kappa_{\D_M}^{x,y}$ satisfies the following difference equation 
\begin{equation}\label{eqn:inhom_discrte_kernel}
     \kappa_{\D_M}^{x,y}(t_i) = \kappa_{\D_M}^{x,y}(t_{i-1}) 
    + \Big( \sigma_A^2 
    V_\varphi(\Sigma_{\D_M}^{x,y}(t_{i-1}))
    + \sigma_b^2 \Big) \frac{\sprod{\Delta x_{t_i}}{\Delta y_{t_i} }_{\R^d}}{\Delta t_i}
\end{equation}
with $\kappa_{\D_M}^{x,y}(0) = \sigma_a^2$ and where 
\begin{equation}\label{eqn:inhom_discrte_sigma}
    \Sigma_{\D_M}^{x,y}(t) =  
    \begin{pmatrix}
        \kappa_{\D_M}^{x,x}(t) & \kappa_{\D_M}^{x,y}(t) \\
        \kappa_{\D_M}^{x,y}(t) & \kappa_{\D_M}^{y,y}(t)
    \end{pmatrix}
\end{equation}
\end{theorem}

\begin{proof}
    We use \cite{yang2019wide}[Corollary 5.5] applied to the Tensor Program of Algorithm \ref{algo:inhom} where the input variables are independently sampled according to
\begin{equation*}
    [A_{j,i}]_{\alpha}^{\beta} \sim \mathcal{N}(0,\frac{\sigma_{A}^2}{N\Delta t_i}), 
    \quad 
    [v]_{\alpha}\sim \mathcal{N}(0,{1}),
    \quad
    [a]_{\alpha}\sim \mathcal{N}(0,\sigma_a^2),
    \quad 
    [b_{j,i}]_{\alpha} \sim \mathcal{N}(0,\frac{\sigma_b^2}{\Delta t_i})
\end{equation*}

The above sampling scheme follows \cite{yang2019wide}[Assumption 5.1]. Furthermore, linearly bounded functions are controlled in the sense of \cite{yang2019wide}[Definition 5.3] since for all $x \in \R$ one has $|\phi(x)| \leq C(1+|x|) \leq e^{|x| + log(C)}$. Thus, we are under the needed assumptions to apply \cite{yang2019wide}[Corollary 5.5]. This result states that the output vector of the discrete controlled ResNet in \Cref{algo:inhom}, on the partition $\D_M$ converges in law, as $N \to \infty$, to a Gaussian distribution $\mathcal{N}(0, K)$ where for $i,j=1,...,n$

{
\begin{equation*}
    [K]_i^j = \E_{Z \sim \mathcal{N}(\mu, \Sigma)}
    \bigg[
    Z^{\cS^{x_i}_{M}} Z^{\cS^{x_j}_{M}}
    \bigg]
    =
    \Sigma(\cS^{x_i}_{M}, \cS^{x_j}_{M})
\end{equation*}

with $\mu, \Sigma$ computed according to \cite{yang2019wide}[Definition 5.2] and defined on the set of all G-vars in the program \emph{i.e.}

\begin{equation*}
    \mu(g) = \begin{cases}
      \mu^{in}(g) & \text{if $g$ is \textbf{Input} G-var}\\
      \sum_k a_k \mu(g_k) & \text{if $g$ is introduced as $\sum_k a_k g_k$ via \textbf{LinComb}}\\
      0 & \text{otherwise}
    \end{cases}
\end{equation*}
\begin{equation*}
    \Sigma(g, g') =  \begin{cases}
      \Sigma^{in}(g, g') & \text{if both $g$ and $g'$ are \textbf{Input} G-var}\\
      \sum_k a_k \Sigma(g_k, g') & \text{if $g$ is introduced as $\sum_k a_k g_k$ via \textbf{LinComb}}\\
      \sum_k a_k \Sigma(g, g_k') & \text{if $g'$ is introduced as $\sum_k a_k g'_k$ via \textbf{LinComb}}\\
      \sigma_W^2 \E_{Z \sim \mathcal{N}(\mu, \Sigma)}[\phi(Z)\phi'(Z)]& \text{if $g = Wh$, $g' = Wh'$ via \textbf{MatMul} w/ same $W$}\\
      0 & \text{otherwise}
    \end{cases}
\end{equation*}

where $h = \phi((g_k)_{k=1}^{m})$ for some function $\phi$ and $\phi(Z) := \phi((Z^{g_k})_{k=1}^{m})$, similarly for $g'$.

In our setting $\mu^{in} \equiv 0$ since all \textbf{Input} variables are independent, from which $\mu \equiv 0$; furthermore $\Sigma^{in}(g, g') = 0$ except if $g = g'$ when it takes values in $\{ \sigma_a^2, \frac{\sigma^2_{b}}{t_l - t_{l-1}}, 1\}$ accordingly.

Following the rules of $\Sigma$, assuming $l_i, l_j \in \{1, \dots, M \}$, we obtain
\begin{eqnarray*}
  \Sigma(\cS^{x_i}_{l_i}, \cS^{x_j}_{l_j}) & = & \Sigma (\cS^{x_i}_{l_i - 1}, \cS^{x_j}_{l_j}) +
  \Sigma \left( \sum_{k = 1}^d \gamma_{k, l_i}^i \Delta (x_i)_{t_{l_i}}^k,
  \cS^{x_j}_{l_j} \right)\\
  & = & \Sigma (\cS^{x_i}_{l_i - 1}, \cS^{x_j}_{l_j}) + \Sigma \left( \sum_{k = 1}^d
  \gamma_{k, l_i}^i \Delta (x_i)_{t_{l_i}}^k, \cS^{x_j}_{l_j-1} \right)\\
  &  & + \Sigma \left( \sum_{k = 1}^d \gamma_{k, l_i}^i \Delta
  (x_i)_{t_{l_i}}^k, \sum_{l = 1}^d \gamma_{m, l_j}^j \Delta
  (x_j)_{t_{l_j}}^l \right)\\
  & = & \Sigma(\cS^{x_i}_{l_i - 1}, \cS^{x_j}_{l_j}) + \Sigma (\cS^{x_i}_{l_i}, \cS^{x_j}_{l_j-1})
  - \Sigma (\cS^{x_i}_{l_i - 1}, \cS^{x_j}_{l_j-1})\\
  &  & + \sum_{k, m = 1}^d \Sigma (\gamma_{k, l_i}^i, \gamma_{m, l_j}^j)
  \Delta (x_i)_{t_{l_i}}^k \Delta (x_j)_{t_{l_j}}^m
\end{eqnarray*}

Now
\begin{align*}
     \Sigma (\gamma_{k, l_i}^i, \gamma_{m, l_j}^j) 
    &= \delta_{k, m} \delta_{i, j} \frac{\sigma^2_{A}}{t_{l_i} - t_{l_i-1}}
    \E[\varphi(Z_1)\varphi(Z_2)]
    +
    \Sigma(b_{k, l_i}, b_{m, l_j})\\
    &= \frac{\delta_{k, m} \delta_{i, j}}{t_{l_i} - t_{l_i-1}}\big[ 
    \sigma_{A}^2 
    \E[\varphi(Z_1)\varphi(Z_2)]
    + \sigma_b^2
    \big]
\end{align*}
   
where $[Z_{1,} Z_2]^{\top} \sim \mathcal{N} (0, \tilde{\Sigma}_{l_{i-1}}(x_i, x_j))$  with

\begin{eqnarray*}
  \tilde{\Sigma}_{l}(x_i, x_j) & = 
  & \left(\begin{array}{c c}
    \Sigma (\cS^{x_i}_l, \cS^{x_i}_{l}) & 
    \Sigma (\cS^{x_i}_l, \cS^{x_j }_{l})\\
    \Sigma (\cS^{x_i}_l, \cS^{x_j }_{l}) & 
    \Sigma (\cS^{x_j }_{l}, \cS^{x_j }_{l})
  \end{array}\right)
\end{eqnarray*}

In particular we see that if $l_i \neq l_j$ then $ \Sigma(\cS^{x_i}_{l_i}, \cS^{x_j}_{l_j}) = \Sigma (\cS^{x_i}_{l_{i \wedge j}}, \cS^{x_j}_{l_{i \wedge j }})$.
Thus if we set, for $t_{l_i} \in \D_M$,  $$\kappa_{\D_M}^{ x_i, x_j}(t_{l_i}) := \Sigma(\cS^{x_i}_{l_i}, \cS^{x_j}_{l_i})$$ 
we get
\begin{equation*}
\begin{gathered}
  \kappa_{\D_M}^{ x_i, x_j}(t_{l_i})  =  \kappa_{\D_M}^{ x_i, x_j}(t_{l_i - 1}) \\
  + \sum_{k = 1}^d \big( \sigma_A^2 \mathbb{E}_{(Z_x, Z_y)
  \sim \mathcal{N} (0, \tilde{\Sigma}_{l_{i-1}}(x_i, x_j))} [\varphi (Z_x) \varphi
  (Z_y)]  + \sigma_b^2 \big) \frac{\Delta (x_i)_{t_{l_i}}^k \Delta (x_j)_{t_{l_i}}^k}{t_{l_i} - t_{l_i-1}}
  \end{gathered}
\end{equation*}
which is exactly what Equation \ref{eqn:inhom_discrte_kernel} states. 
Then note how 
\[
\Sigma (\cS^{x_i}_{0}, \cS^{x_j}_{0}) = \sigma_a^2
\]

Thus finally we can conclude and write the entries of the matrix $\Tilde{K}$ as
\[
    [{K}]_i^j = \kappa_{\D_M}^{ x_i, x_j}(1)
\]
}

\end{proof}

\begin{algorithm}[tbh]
    \caption{$\cS^{M, N}_1$ as Nestor program}
    \label{algo:inhom}
    \begin{algorithmic}
       \STATE {\bfseries Input:} $\cS_0: \Gtype(N)$ \hfill  $\triangleright$ \textit{initial value}
       \STATE {\bfseries Input:} $(b_{1}, \dots, b_{d}): \Gtype(N)$ \hfill $\triangleright$ \textit{biases}
       \STATE {\bfseries Input:} $(A_{1,l}, \dots, A_{d,l})_{l=1,\cdots,M}: \Atype(N, N)$ \hfill $\triangleright$ \textit{matrices}
       \STATE {\bfseries Input:} $v: \Gtype(N)$ \hfill $\triangleright$ \textit{readout layer weights}
      \\
      \FOR{$i = 1, \dots, n$}
          \STATE {\it // Compute $\cS^{M, N}_1 (x_i)$  (here $ \cS_0^{x_i}$ is to be read as $\cS_0$)}
          \\
          \FOR{$l = 1,\dots,M$}
              \FOR{$k = 1,\dots,d$}
                \STATE $\alpha^i_{k,l} := \varphi(\cS^{x_i}_{l-1}): \Htype(N)$ \hfill  $\triangleright$ \textit{by Nonlin;}
                \STATE $\beta^i_{k,l} := A_{k, l} \alpha^i_{k,l} : \Gtype(N)$ \hfill  $\triangleright$ \textit{by Matmul;}
                \STATE $\gamma^i_{k,l} := \beta^i_{k,l} + b_{k, l}  : \Gtype(N)$ \hfill  $\triangleright$ \textit{by LinComb;}
              \ENDFOR
              \STATE $\cS^{x_i}_{l} := \cS^{x_i}_{l-1} + \sum_{k=1}^d \gamma^i_{k,l} [(x_i)^k_{t_l} - (x_i)^k_{t_{l-1}}]  : \Gtype(N)$ \hfill  $\triangleright$ \textit{by LinComb;}
          \ENDFOR
          \\
      \ENDFOR
      \\
      \ENSURE $(v^T \cS_{M}^{x_i} / \sqrt{N})_{i=1,\dots,n}$
    \end{algorithmic}
\end{algorithm}

\begin{remark}
There are two things to notice, done in the above proof in order to satisfy the required formalism: 
\begin{itemize}
    \item In the program the output projector $v$ is sampled according to $\mathcal{N}(0,1)$ while the original $\phi \sim \mathcal{N}(0,\frac{1}{N})$. This does not pose any problems since the output of the formal programs uses ${v}/{\sqrt{N}}\sim \mathcal{N}(0,\frac{1}{N})$.
    \item The input paths $x_i$ enter program \ref{algo:inhom} not as \emph{Inputs} but as coefficients of \emph{LinComb}, this means that for any choice of input paths we must formally consider different algorithms. In any case, for any possible choice, the result has always the same functional form; hence \emph{a posteriori} it is legitimate to think about \emph{one} algorithm.
\end{itemize}
\end{remark}

{
Actually we have proved the even stronger statement, in the sense that the previous result holds for intermediate times too: 
\begin{corollary}
    For all $t_m,t_n \in \D_M$ one has the following distributional limit
    \[
    \sprod{\psi^N}{\cS^{N,M}_{t_m}(\X)}_{\R^N}
    \sprod{\psi^N}{\cS^{N,M}_{t_n}(\X)}_{\R^N}
    \xrightarrow[N \to \infty]{} 
    \mathcal{N}(0,\kappa_{\D_M}^{\X, \X}(t_m \wedge t_n))
    \]
    and the matrices $\Sigma_{\D_M}^{x,y}(t_n)$ are always in $PSD_2$.
\end{corollary}
}

\subsubsection{Uniform convergence of discrete kernels}\label{subsubsec:conv-inhom-Kernel}

We now prove the convergence of the kernels $\kappa_{\D_M}^{x,y} : \D_M \to \R$ established in the previous section to a unique limiting kernel $\kappa^{x,y}_\varphi: [0,1] \to \R$ as $|\D_M| \to 0$. We first extend the discrete kernels $\kappa_{\D_M}^{x,y}$ to maps defined on $[0,1]$ in two ways. 


\begin{definition}
  We extend the map $\kappa_{\D_M}^{x,y} : \D_M \rightarrow \mathbb{R}$ to the whole interval $[0,1]$ in two ways: for any $t \in [t_m, t_{m+1})$ 
  \begin{enumerate}
      \item (piecewise linear interpolation) using a slight abuse of notation that overwrites the previous one, define the map $\kappa_{\D_M}^{x,y} : [0, 1] \rightarrow \mathbb{R}$
      as
      \begin{equation*}
        \kappa_{\D_M}^{x,y} (t) =
        \kappa_{\D_M}^{x, y} (t_{m}) + 
        \left( \sigma_A^2 V_{\varphi} \left( \Sigma_{\D_M}^{x, y}(t_{m }) \right) + \sigma_b^2 \right)
        \sprod{\frac{x_t - x_{t_{m}}}{t - t_{m}}}{\frac{y_t - y_{t_{m}}}{t - t_{m}} }_{\R^d} (t - t_{m})
      \end{equation*}
      We extend in a similar way the matrix $\Sigma_{\D_M}^{x,y}$ defined in equation (\ref{eqn:inhom_discrte_sigma}).
      \item (piecewise constant interpolation) define the map  $\Tilde\kappa_{\D_M}^{x,y} : [0, 1] \rightarrow \mathbb{R}$ as
      \[
        \Tilde{\kappa}_{\D_M}^{x,y}(t) = \kappa_{\D_M}^{x,y}(t_{m}).
      \]
       and similarly for the matrix $\Tilde\Sigma_{\D_M}^{x,y}$.
   \end{enumerate}
\end{definition}

\begin{remark}
    It is important to consider both these types of extensions. The \emph{piecewise linear}, being continuous, is used to prove uniform convergence to a limiting map in the space of continuous functions $C^0([0,1];\R)$. The \emph{piecewise constant} is proved to converge to the same object, this time in $L^{\infty}([0,1];\R)$ since it's not continuous, and is well suited to prove how the positive semidefinitess properties of the discrete kernels pass to the limit. 
\end{remark}

\begin{theorem} \label{inhom_simple_conv}   
  Fix a sequence $\{\D_M \}_{M \in \N}$ of partitions of $[0,1]$ with $| \D_M | \rightarrow 0$ as $M \to \infty$. Then, for any two paths $x,y \in \bX$, the sequence of functions $\{\kappa_{\D_M}^{x,y}\}_M$ converges uniformly on $C^0([0,1]; \R)$ to the unique solution $\kappa_{\varphi}^{x, y}:[0,1] \to \R$ of the following differential equation
  \begin{equation}\label{app:eqn:inhom-simple-conv-PDE}
    \kappa_\varphi^{x,y} (t) =  \sigma^2_{a} +  \int_{0}^t (\sigma_A^2 V_{\varphi} (\Sigma_\varphi^{x, y}(s)) + \sigma_b^2) \sprod{\Dot{x}_s}{ \Dot{y}_{s}}_{\R^d}   d s
  \end{equation}
  with
  \[ \Sigma_\varphi^{x, y}(t) = \left(\begin{array}{c}
       \kappa_\varphi^{x, x} (t), \kappa_\varphi^{x, y} (t)\\
       \kappa_\varphi^{x, y} (t), \kappa_\varphi^{y, y} (t)
     \end{array}\right) \]
\end{theorem}

{
One of the main difficulties when dealing with Equation (\ref{app:eqn:inhom-simple-conv-PDE}) is making sure that the matrix $\Sigma_\varphi^{x,y}$ stays in $PSD_2$ for all times $t \in [0,1]$, in order to have $V_{\varphi}(\Sigma^{x,y}(t))$ well defined.  
This is clear if $x=y$ when $\Sigma_\varphi^{x,x}(t) = \kappa_\varphi^{x,x}(t)\mathbf{1}$, and one can just use \cite{FrizVictoir}[Theorem 3.7] to conclude, but it is not in the general case. 
One possible way to tackle this problem would be to consider the triplet $(\kappa_\varphi^{x,x},\kappa_\varphi^{x,y},\kappa_\varphi^{y,y}) : [0,1] \to  \R^3$ and find it as the solution of (\ref{app:eqn:inhom-simple-conv-PDE}) on a submanifold of $\R^3$. 
We decided to employ a more elementary technique which, unlike this "PDE on manifold" one, extends to the homogeneous case too, and we will find $\Sigma_\varphi^{x,y}: [0,1] \to \R^4$ as uniform limit of matrices in $PSD_2$.
}

The main idea for proving \Cref{inhom_simple_conv} will be to show that the sequence $\{ \kappa_{\D_M}^{x,y} \}_M$ is uniformly bounded and uniformly equicontinuous so that Ascoli-Arzelà theorem applies. One then proves that the limit of the resulting subsequence is the unique solution of \cref{app:eqn:inhom-simple-conv-PDE} and that the whole sequence converges to it. Before proving \Cref{inhom_simple_conv} we need several lemmas.

The first step is to establish a uniform lower bound.

\begin{lemma}[Uniform lower bounds]
For any $ x \in \bX$, any partition $\D$
$$\|\kappa_{\D}^{ x, x} \|_{\infty, [0, 1]} \assign \sup\limits_{t \in [0, 1]} \left|\kappa_{\D}^{ x, x} (t)\right| \geq \sigma_a^2 $$
and similarly for $\Tilde{\kappa}_{\D}^{x, x}$.
\end{lemma}

\begin{proof}
Setting, for any $t_m\in\mathcal{D}$ by definition of the kernel $\kappa_{\D}^{x, x}$
  \[ 
     \kappa_{\D}^{x,x} (t_m) = \sigma_a^2 +
     \sum_{\tmscript{\begin{array}{c}
       0 \leq l < m
     \end{array}}}\left( \sigma_A^2 V_{\varphi} \left(
     \Sigma_{\D}^{x,x} (t_{l})
     \right) + \sigma_b^2 \right)
     \left|{\frac{\Delta x_{t_{l + 1}}}{\Delta t_{l + 1}}}\right|^2  \Delta t_{l + 1} 
  \]

Recall the definition of the map $V_{\varphi}$ in \cref{lemma:split_on_PSD_R}. For any $t_l\in\D$ we note that 
$V_{\varphi} ( \Sigma_{\D}^{x,x} (t_{l})) = \E[\varphi(\sqrt{\kappa_{\D}^{x,x}(t_l)})^2]$ for $Z \sim \mathcal{N}(0,1)$ since we have that $\Sigma_{\D}^{x,x} (t_{l}) = \kappa_{\D}^{x,x}(t_l) \mathbf{1}$ with $\mathbf{1} \in R^{2\times 2}$ the matrix with all entries equal to $1$.
Thus $V_{\varphi} ( \Sigma_{\D}^{x,x} (t_{l})) \geq 0$ and we can conclude that $\Tilde\kappa_{\D}^{x, x}(t) \geq \sigma_a^2 $ since we are summing to $\sigma_a^2$ only the non-negative terms
\[
\left( \sigma_A^2 V_{\varphi} \left(
     \Sigma_{\D}^{x,x} (t_{l})
     \right) + \sigma_b^2 \right)
     \left|{\frac{\Delta x_{t_{l + 1}}}{\Delta t_{l + 1}}}\right|^2  \Delta t_{l + 1} 
\]

For any $t \in [0,1]$, by definition, we can write
  \begin{equation*}
      \kappa_{\D}^{x, x} (t) =
    \kappa_{\D}^{x,x} (t_{m}) + 
    \left( \sigma_A^2 V_{\varphi} \left( \Tilde{\Sigma}_{\D}^{x,x} (t_{m}) \right) + \sigma_b^2 \right) \left|\frac{x_t - x_{t_{m}}}{t-t_{m}}\right|^2_{\R^d} (t-t_{m})
  \end{equation*}
  thus $\textbf{}$ is the sum of $ \kappa_{\D}^{x,x} (t_{m }) = \Tilde{\kappa}_{\D}^{x,x}(t) \geq \sigma_a^2$ and the quantity
  \begin{equation*}
      \left( \sigma_A^2 V_{\varphi} \left( \Tilde{\Sigma}_{\D}^{x,x}(t_{m}) \right) + \sigma_b^2 \right) \left|\frac{x_t - x_{t_{m}}}{t-t_{m}}\right|^2_{\R^d} (t-t_{m})
  \end{equation*}
  which is $\geq 0$ by the same arguments as above.
\end{proof}

The second step is to establish a uniform upper bound.

\begin{lemma}[Uniform upper bounds]\label{lemma:upper_bound_xx_inhom}
  For any $x \in \bX$ and any partition $\D$ there exists a constant $C_x > 0$ independent of $\D$ such that
  \begin{equation*}
      {\| \Tilde\kappa_{\D}^{ x, x} \|_{\infty, [0, 1]} }
      \leq \| \kappa_{\D}^{ x, x} \|_{\infty,
     [0, 1]}  \leq C_x
  \end{equation*}
\end{lemma}

\begin{proof}
    {
    Note how all the values taken by $\Tilde\kappa_{\D}^{ x, x}$ are also taken by $\kappa_{\D}^{ x, x}$ in the corresponding partition points, thus the first inequality is trivial.
    Let us find an upper bound $\Tilde{C}_x$ for $\Tilde\kappa_{\D}^{ x, x}$ first, intuitively since $\kappa_{\D}^{ x, x}$ cannot be far from $\Tilde\kappa_{\D}^{ x, x}$ the constant $\Tilde{C}_x$ should help find the bound $C_x$.
    }
    
    Let $t \in \left[ t_m , t_{m + 1}  \right)$. Recall the definition of $\Tilde{\kappa}_{\D}^{x,x}$
  \begin{align*}
    \Tilde\kappa_{\D}^{x, x} (t) &=
     \kappa_{\D}^{x,x} (t_m) \\
     &= \sigma_a^2 +
     \sum_{\tmscript{\begin{array}{c}
       0 \leq l < m
     \end{array}}}\left( \sigma_A^2 V_{\varphi} \left(
     \Sigma_{\D}^{x,x} (t_{l})
     \right) + \sigma_b^2 \right) \sprod{\frac{\Delta x_{{t_{l + 1}} }}{\Delta t_{l + 1}}}{
     \frac{\Delta x_{t_{l + 1}}}{\Delta t_{l + 1}}}\Delta t_{l + 1}
  \end{align*}
We then have 
    \begin{align*}
            \sprod{\frac{\Delta x_{{t_{l + 1}} }}{\Delta t_{l + 1}}}{\frac{\Delta x_{t_{l + 1}}}{\Delta t_{l + 1}}}\Delta t_{l + 1} &=
            \left|\frac{\Delta x_{{t_{l + 1}} }}{\Delta t_{l + 1}}\right|_{\R^d}^2 \Delta t_{l + 1} \\
            &= 
            \left|\frac{1}{\Delta t_{l + 1}}  \int_{t_l}^{t_{l+1}} \Dot{x}_t dt \right|_{\R^d}^2 \Delta t_{l + 1}\\
            &\leq (\frac{1}{\Delta t_{l + 1}}  \int_{t_l}^{t_{l+1}} |\Dot{x}_t|_{\R^d} dt)^2 \Delta t_{l + 1}\\
            &\leq \frac{1}{\Delta t_{l + 1}}  \int_{t_l}^{t_{l+1}} |\Dot{x}_t|_{\R^d}^2 dt \Delta t_{l + 1}\\
           & = \int_{t_l}^{t_{l+1}} |\Dot{x}_t|_{\R^d}^2 dt
    \end{align*}
where the last inequality is by Jensen's inequality. In addition, by \cref{lemma:split_on_PSD_R} there exists a positive constant $\Tilde{M}$ such that
     \begin{equation*}
        \begin{gathered}
            |V_{\varphi}(\Sigma_{\D}^{x,x} (t_{l}))| \leq \Tilde{M}(1+\sqrt{\kappa_{\D}^{ x, x}(t_l)})^2 \leq 2\Tilde{M}(1+ \kappa_{\D}^{ x, x}(t_l)).
        \end{gathered}
    \end{equation*}
  Thus
  \begin{equation*}
      \begin{gathered}
          \Tilde\kappa_{\D}^{ x, x}(t) \leq \sigma_a^2 + \sum_{0 \leq l < m}
          (2\sigma_A^2 \Tilde{M}(1+ \kappa_{\D}^{ x, x} (t_l)) + \sigma_b^2 ) 
          \int_{t_l}^{t_{l+1}} |\Dot{x}_t|_{\R^d}^2 dt \\
          = \sigma_a^2 + \int_{0}^{t_{m}}
          (2\sigma_A^2 \Tilde{M}(1+ \Tilde\kappa_{\D}^{ x, x}(t)) + \sigma_b^2 ) 
           |\Dot{x}_t|_{\R^d}^2 dt
      \end{gathered}
  \end{equation*}
By Gronwall inequality (\cite{FrizVictoir}, Lemma 3.2) we have
  \begin{eqnarray*}
    1 + \Tilde \kappa_{\D}^{ x, x} (t) & \leq & (1 + \sigma_a^2 + \sigma_b^2 \norm{x}_{\bX}^2) \exp \{ 2 \sigma_A^2 \tilde{M} \norm{x}_{\bX}^2 \}
  \end{eqnarray*}
  hence the statement of this lemma holds for $\Tilde\kappa_{\D}^{ x, x}$ with the constant
  \begin{eqnarray*}
    \Tilde{C}_x & = & (1 + \sigma_a^2 + \sigma_b^2 \norm{x}_{\bX}^2) {e^{2 \sigma_A^2 \tilde{M} \norm{x}_{\bX}^2}}  - 1
  \end{eqnarray*}

  To prove a similar inequality for $\kappa_{\D}^{x,x}$, consider  
  \[
    \kappa_{\D}^{x,x} (t) =
    \kappa_{\D}^{x,x} (t_{m }) + 
    \left( \sigma_A^2 V_{\varphi} \left( \Sigma_{\D}^{x, x} (t_{m }) \right) + \sigma_b^2 \right)
    \left|\frac{x_t - x_{t_{m}}}{t - t_{m}}\right|^2 (t - t_{m})
  \]
  Then we have
\begin{align*}
    |\kappa_{\D}^{x,x} (t)| &\leq 
         |\kappa_{\D}^{x,x} (t_{m})| +  (2\sigma_A^2 \Tilde{M}(1 + \Tilde{C}_x) + \sigma_b^2 ) 
           \left|\frac{x_t - x_{t_{m}}}{t - t_{m}}\right|^2 (t - t_{m})\\
          &\leq \Tilde{C}_x + (2\sigma_A^2 \Tilde{M}(1+ \Tilde{C}_x) + \sigma_b^2 ) \norm{x}_{\bX}^2
\end{align*}
Hence the  statement follows by  setting 
  \[
    C_x = \Tilde{C}_x + (2\sigma_A^2 \Tilde{M}(1+ \Tilde{C}_x) + \sigma_b^2 ) \norm{x}_{\bX}^2
  \]
\end{proof}

\begin{remark}
    Note how $C_x$ only depends on $\norm{x}_{\bX}$ and is increasing in it, thus this bound is uniform on bounded subsets of $\bX$.
\end{remark}

The following lemma shows that the kernels are in fact elements of $PSD(R)$.

\begin{lemma}\label{lemma:psd_inhom}
  There exists a  constant $R = R_x \in \mathbb{R}$ such that $\tilde\Sigma_{\D}^{x,x}(t), \Sigma_{\D}^{x,x}(t) \in \tmop{PSD_2} (R)$ for every $\D$ and $t \in [0,1]$. Moreover, as before, this $R_x$ only depends on $\norm{x}_{\bX}$ and is increasing in it. 
\end{lemma}

\begin{proof}
  $\tilde\Sigma_{\D}^{x,x}(t)$ and $\Sigma_{\D}^{x,x}(t)$ are $PSD_2$ since they are of form $a\mathbf{1}$ for some $a > 0$.
The diagonal elements $\Tilde\kappa_{\D}^{ x, x}(t), \kappa_{\D}^{x,x}(t)$ are
  bounded above by a common constant $C_x$ and below by $\sigma^2_{a}$. 
  We thus can simply choose $R_x = C_x \vee \sigma_{a}^{-2}$.
\end{proof}

We now extend the results of \Cref{lemma:upper_bound_xx_inhom} and \Cref{lemma:psd_inhom} to the case $x\neq y$.

\begin{lemma}\label{lemma:inhom_bounded_set_bounds}
    Fix a $\alpha > 0$. There exist a constant $C_{\alpha}$ such that for all $x,y \in \bX$ with $\norm{x}_{\bX}, \norm{y}_{\bX} \leq \alpha$ and all partitions $\D$ it holds
    \[
        \| \kappa_{\D}^{x, y} \|_{\infty,
     [0, 1]} \leq C_{\alpha} 
    \]
    Moreover for $R_{\alpha} := C_{\alpha} \vee \sigma_{a}^{-2}$ we get 
    \[
          \Tilde\Sigma_{\D}^{x, y} (t) \in \tmop{PSD_2}(R_{\alpha}) 
    \]
\end{lemma}

\begin{proof}
Remember how the maps $\Tilde\Sigma_{\D}^{x, y}(t)$ are $PSD_2$ {since their values are found as covariance matrices of Gaussian random variables with Tensor Program arguments in Theorem \ref{thm:inhom-discrete-phi-SigKer} }, in particular by positive semidefinitiveness  and the previous bounds
\begin{eqnarray*}
\nobracket | \nobracket \Tilde{\kappa}_{\D}^{x, y} (t) | \leq \sqrt{\Tilde\kappa_{\D}^{ x, x} (t) \Tilde{\kappa}_{\D}^{y, y} (t)} \leq \sqrt{\Tilde{C_x }\Tilde{C_y}} \leq \Tilde{C_x} \vee \Tilde{C_y} 
\end{eqnarray*}
For $\kappa_{\D}^{x,y}(t)$ we proceed similarly to before: when $t \in [t_m, t_{m+1})$ one has
\begin{align*}
          |\kappa_{\D}^{x,y}(t)| & \leq 
          |\kappa_{\D}^{x,y} (t_{m})|  \\
          & +  \left|(\sigma_A^2 \Tilde{M}(1+ \sqrt{\Tilde\kappa_{\D}^{ x, x}(t)})(1+ \sqrt{\Tilde \kappa_{\D}^{y,y}(t)}) + \sigma_b^2 ) 
           \sprod{\frac{x_t - x_{t_{m}}}{t - t_{m}}}{\frac{y_t - y_{t_{m}}}{t - t_{m}} }_{\R^d} (t - t_{m})\right|\\
           &\leq
           \Tilde{C_x} \vee \Tilde{C_y} \\
           &+ (2 \sigma_A^2 \Tilde{M}(1+ \Tilde{C_x} \vee \Tilde{C_y}) + \sigma_b^2 )
           |\sprod{\frac{x_t - x_{t_{m}}}{t - t_{m}}}{\frac{y_t - y_{t_{m}}}{t - t_{m}} }_{\R^d}| (t - t_{m})\\
           &\leq \Tilde{C_x} \vee \Tilde{C_y} + (2 \sigma_A^2 \Tilde{M}(1+ \Tilde{C_x} \vee \Tilde{C_y}) + \sigma_b^2 )
           (|\frac{x_t - x_{t_{m}}}{t - t_{m}}|^2 + |\frac{y_t - y_{t_{m}}}{t - t_{m}}|^2)(t - t_{m})\\
           &\leq \Tilde{C_x} \vee \Tilde{C_y} + ( 2 \sigma_A^2 \Tilde{M}(1+ \Tilde{C_x} \vee \Tilde{C_y}) + \sigma_b^2 )(\norm{x}_{\bX} + \norm{y}_{\bX}) \leq C_{\alpha}
  \end{align*}
where we have used $|\sprod{a}{b}| \leq  |a||b| \leq (|a|^2 + |b|^2)$ for $a,b \in \R^d$. The second part follows from the definition of $\tmop{PSD_2}(R_{\alpha})$.
\end{proof}

We are now ready to prove \Cref{inhom_simple_conv}.

\begin{proof}[Proof of \Cref{inhom_simple_conv}]~\\

\paragraph{Part I (Convergence $\sup_{[0, 1] \times [0, 1]} | \kappa_{\D_M}^{x,y}
(t) - \Tilde\kappa_{\D_M}^{x,y} (t) | \to 0$ as $| \D | \rightarrow 0$)} The set $\{x,y\}$ is bounded in $\bX$, we can thus fix two constants $C_{x,y}, R_{x,y}$ with the properties given in \Cref{lemma:inhom_bounded_set_bounds}. We will often, for ease of notation, refer to them as $C$ and $R$.  We have, for $t \in [t_m, t_{m+1})$, that 
 \begin{align*}
     &| \kappa_{\D_M}^{x,y} (t) - \Tilde\kappa_{\D_M}^{x,y} (t) |  \\
     &\leq
          |(\sigma_A^2 \Tilde{M}(1+ \sqrt{\Tilde\kappa_{\D}^{ x, x}(t)})(1+ \sqrt{\Tilde \kappa_{\D_M}^{y,y}(t)}) + \sigma_b^2 ) 
           \sprod{\frac{x_t - x_{t_{m}}}{t - t_{m}}}{\frac{y_t - y_{t_{m}}}{t - t_{m}} }_{\R^d} (t - t_{m})|  \\
          &\leq  (2\sigma_A^2 \Tilde{M} (1 + C_{x,y}) + \sigma_b^2) 
          (|\frac{x_t - x_{t_{m}}}{t - t_{m}}|^2 + |\frac{y_t - y_{t_{m}}}{t - t_{m}}|^2)(t - t_{m}) \\
          &\leq (2\sigma_A^2 \Tilde{M} (1 + C_{x,y}) + \sigma_b^2) \int_{t_m}^t |\Dot{x}_s|^2 + |\Dot{y}_s|^2 ds
 \end{align*}

  In particular as $| \D | \rightarrow 0$ we
  have
    \begin{equation*} \begin{gathered}
    \sup_{[0, 1] \times [0, 1]} | \kappa_{\D_M}^{x,y}
     (t) - \Tilde\kappa_{\D_M}^{x,y} (t) | \to 0
     \end{gathered} \end{equation*}
  since, by dominated convergence, one has 
  \[
    \int_{t_m}^t |\Dot{x}_s|^2 + |\Dot{y}_s|^2 ds = \int_0^1 \mathbb{I}_{[t_m,t)} (|\Dot{x}_s|^2 + |\Dot{y}_s|^2) ds \to 0 
  \]

\paragraph{Part II (Ascoli-Arzelà)} By \Cref{lemma:inhom_bounded_set_bounds} the sequence of functions $\{ \kappa_{\D_M}^{x,y} \}_M$ is uniformly bounded. Assume $s<t$ then one has, with the same bounds just used, that
    \begin{equation*}
      \begin{gathered}
          | \kappa_{\D_M}^{x,y} (t) - \kappa_{\D_M}^{x,y} (s) |  
          \leq  2(2\sigma_A^2 \Tilde{M} (1 + C_{x,y}) + \sigma_b^2) \int_{s}^t |\Dot{x}_r|^2 + |\Dot{y}_r|^2 ds
      \end{gathered}
  \end{equation*}
  thus the sequence of functions $\{ \kappa_{\D_M}^{x,y} \}_M$ is also uniformly equicontinuous, in fact the bound is independent from the partition. We can apply Ascoli-Arzelà to conclude that there exist a subsequence $\{\D_{M_k}\}$ with $\kappa_{\D_M}^{x, y}$ uniformly converging to a limiting $\kappa^{x,y} \in C^0([0,1];\R)$.  Since this must be the uniform limit of $\Tilde\kappa_{\D_M}^{x, y}$ too, by the result of \emph{Part} I,  we obtain that the corresponding limiting matrices $\Sigma_\varphi^{x,y}(t) \in PSD_2(R_{x,y})$ for all $t \in [0,1]$ {since the $\Tilde{\Sigma}_\varphi^{x,y}(t)$ are and $PSD_2(R_{x,y})$ is closed}
  \footnote{This is important to have a candidate solution to Equation (\ref{app:eqn:inhom-simple-conv-PDE}) which otherwise would not be well defined.}.~\\

\paragraph{Part III (Uniqueness of solutions to equation (\ref{app:eqn:inhom-simple-conv-PDE}))} 
  
Here we prove that if the PDE (\ref{app:eqn:inhom-simple-conv-PDE}) admits a solution this must be unique. Assume the existence of different solutions $K = (K_{x,x}, K_{x,y}, K_{y,y})$ and $ G = (G_{x,x}, G_{x,y}, G_{y,y})$ with all the $\Sigma^K$ and $\Sigma^G$ in $PSD_2$ . From Eq (\ref{app:eqn:inhom-simple-conv-PDE}) it is clear that $K_{x,x}, K_{y,y}, G_{x,x}, G_{y,y} \geq \sigma_a^2$ and, by continuity, that they are bounded by some constant; thus all the $\Sigma_K$ and $\Sigma_G$ are in some in $PSD_2(\Bar{R})$. Then by the Lipschitz property of $V_{\varphi}$ one sees that
  \begin{align*}
      \norm{\Sigma_K^{x,y}(t) - \Sigma_G^{x,y}(t)}_{\infty}
      \leq \int_0^t \sigma_A^2 k_{\Bar{R}}\norm{\Sigma_K^{x,y}(t) - \Sigma_G^{x,y}(t)}_{\infty} (|\sprod{\Dot{x}_r}{\Dot{x}_r}| + |\sprod{\Dot{x}_r}{\Dot{y}_r}| + |\sprod{\Dot{y}_r}{\Dot{y}_r}|) dr
  \end{align*}
  thus $ \norm{\Sigma_K^{x,y}(t) - \Sigma_G^{x,y}(t)}_{\infty} = 0$ by Gronwall for all $t \in [0,1]$ \emph{i.e} $K = G$.

  \paragraph{Part IV (Limiting kernel solves equation (\ref{app:eqn:inhom-simple-conv-PDE}))} 

We now need to prove that the limit $\kappa^{x,y}_\varphi$ of the subsequence $\{\kappa_{\D_M}^{x,y}\}_k$ solves the PDE, it will then follow that any sub-sequence $\{\kappa_{\D_M}^{x,y}\}$ admits a further sub-sequence converging to the same map $\kappa^{x,y}_\varphi$, giving us the convergence of the whole sequence. Thus without loss of generality we can assume in the sequel that the whole sequence converges.
    
  Let us prove that the limit $\kappa_\varphi^{x, y}$ is, in fact, a solution of the PDE. Let $t \in [t_m, t_{m+1})$ for some fixed $\D$, then{
  \begin{align*}
    & \left| \kappa_\varphi^{x,y} (t) - \sigma_a^2 +  \int_{0}^t (\sigma_A^2 V_{\varphi} (\Sigma^{x, y}_\varphi(s)) + \sigma_b^2) \sprod{\Dot{x}_s}{ \Dot{y}_{s}}_{\R^d}   d s   \right|  
    \\ 
    \leq & 
    \left| \kappa_\varphi^{x,y} (t) - \kappa_{\D_M}^{x,y} (t) \right| + \sigma_A^2 \int_{0}^t | V_{\varphi}
    (\Sigma^{x, y}_\varphi (s))
    - V_{\varphi} (\tilde\Sigma_{\D_M}^{x,y} (s)) |  
    | \sprod{\Dot{x}_{s}}{\Dot{y}_{s}}_{\R^d} | ds  
    \\
    & 
    + \sum_{0 \leq l < m} \left|\left( \sigma_A^2 V_{\varphi} \left(
    \Sigma_{\D_M}^{x,x} (t_{l})
    \right) + \sigma_b^2 \right)\right| \left|\int_{t_l}^{t_{l+1}}  \sprod{\Dot{x}_s}{ \Dot{y}_{s}}_{\R^d} ds - \sprod{\frac{\Delta x_{{t_{l + 1}} }}{\Delta t_{l + 1}}}{
    \frac{\Delta y_{t_{l + 1}}}{\Delta t_{l + 1}}}\Delta t_{l + 1}\right| 
    \\
    & 
    + \left|\left( \sigma_A^2 V_{\varphi} \left(
    \Sigma_{\D_M}^{x,x} (t_{m})
    \right) + \sigma_b^2 \right)\right| \left|\int_{t_m}^{t}  \sprod{\Dot{x}_s}{ \Dot{y}_{s}}_{\R^d} ds - \sprod{\frac{x_t - x_{t_m}}{t - t_m}}{
    \frac{y_t - y_{t_m}}{t - t_m}}(t - t_m)\right| 
    \\
    \leq & 
    | \kappa_\varphi^{x,y} (t) - \kappa_{\D_M}^{x,y} (t) | + \sigma_A^2 \int_{0}^t | V_{\varphi}
    (\Sigma_\varphi^{  x, y} (s))
    - V_{\varphi} (\tilde{\Sigma}_{\D_M}^{x, y} (s)) |  
    | \sprod{\Dot{x}_{s}}{\Dot{y}_{s}}_{\R^d} |    ds  
    \\
    & 
    + \left( 2 \sigma_A^2 \Tilde{M} (1 + C_{x,y}) + \sigma_b^2 \right)
    \Big\{ \sum_{0 \leq l < m} \left|\int_{t_l}^{t_{l+1}}  \sprod{\Dot{x}_s}{ \Dot{y}_{s}}_{\R^d} ds - \sprod{\frac{\Delta x_{{t_{l + 1}} }}{\Delta t_{l + 1}}}{
    \frac{\Delta y_{t_{l + 1}}}{\Delta t_{l + 1}}}\Delta t_{l + 1}\right| 
    \\
    & 
    + \left|\int_{t_m}^{t}  \sprod{\Dot{x}_s}{ \Dot{y}_{s}}_{\R^d} ds - \sprod{\frac{x_t - x_{t_m}}{t - t_m}}{
    \frac{y_t - y_{t_m}}{t - t_m}}(t - t_m)\right| \Big\} &
  \end{align*}
    }
   but 
   
   {
    \begin{align*}
            \left|\int_{a}^{b}  \sprod{\Dot{x}_s}{ \Dot{y}_{s}}_{\R^d} ds
            - \sprod{\frac{x_b - x_{a}}{b-a}}{\frac{y_b - y_{a}}{b-a}}(b-a) \right| &= 
            \left|\int_{a}^{b}  \sprod{\Dot{x}_s}{ \Dot{y}_{s}}_{\R^d} 
            - \sprod{\frac{x_b - x_{a}}{b-a}}{\Dot{y}_s}_{\R^d} ds\right| \\
            &= 
            \left|\int_{a}^{b}  \sprod{\Dot{x}_s - \frac{x_b - x_{a}}{b-a}}{ \Dot{y}_{s}}_{\R^d}  ds \right| \\
            &\leq
            \int_{a}^{b} |\Dot{x}_s - \frac{x_b - x_{a}}{b-a}||\Dot{y}_{s}| ds  
    \end{align*}
    hence
   \begin{align*}
     &\left| \kappa_\varphi^{x,y} (t) - \sigma_a^2 +  \int_{0}^t (\sigma_A^2 V_{\varphi} (\Sigma^{x, y}_\varphi(s)) + \sigma_b^2) \sprod{\Dot{x}_s}{ \Dot{y}_{s}}_{\R^d}   d s   \right|  \\
    &\leq | \kappa_\varphi^{x,y} (t) - \kappa_{\D_M}^{x,y} (t) | + \sigma_A^2 \int_{0}^t | V_{\varphi}
    ({\Sigma}_\varphi^{  x, y} (s))
    - V_{\varphi} (\tilde{\Sigma}_{\D_M}^{x, y} (s)) |  
    | \sprod{\Dot{x}_{s}}{\Dot{y}_{s}}_{\R^d} |    ds  
    \\
    &+ \left( 2 \sigma_A^2 \Tilde{M} (1 + C_{x,y}) + \sigma_b^2 \right)
    \big\{ 
    \int_0^t \Big(\mathbb{I}_{[t_m, t)} |\Dot{x}_s - \frac{x_t - x_{t_{m}}}{t - t_m}| + \sum_{0 \leq l < m} \mathbb{I}_{[t_l, t_{l+1})} |\Dot{x}_s - \frac{\Delta x_{t_{l+1}}}{\Delta t_{l+1}}|\Big) |\Dot{y}_{s}| ds \big\}
    \end{align*}  
   
    Now, considering the sequence  $\D_M$, by convergence 
    \[
    | \kappa_\varphi^{x,y} (t) - \kappa_{\D_M}^{x,y} (t) | = o(1)
    \]
   and
    \[
        \int_{0}^t | V_{\varphi} ({\Sigma}^{  x, y} (s))  - V_{\varphi} (\tilde{\Sigma}_{\D_M}^{x, y} (s)) |  |\sprod{\Dot{x}_{s}}{\Dot{y}_{s}}_{\R^d}| ds
        \leq
        \int_{0}^t k_R  | \Sigma^{x, y}_\varphi (s) - \Tilde{\Sigma}_{\D_M}^{x, y} (s) |_{\infty}  |\sprod{\Dot{x}_{s}}{\Dot{y}_{s}}_{\R^d}| ds = o(1)
    \]

    Moreover, setting $m_M$ to be such that $t \in [t_{m_M}, t_{m_M + 1})$ in $\D_M$, we have that by Lebesgue differentiation theorem 
    \[
    \Big(\mathbb{I}_{[t^{\D_M}_{m_M}, t)}(s) |\Dot{x}_s - \frac{x_t - x_{t^{\D_M}_{m_M}}}{t - t^{\D_M}_{m_M}}| + \sum_{0 \leq l < m_M} \mathbb{I}_{[t^{\D_M}_l, t^{\D_M}_{l+1})}(s) |\Dot{x}_s - \frac{\Delta x_{t^{\D_M}_{l+1}}}{\Delta t^{\D_M}_{l+1}}|\Big) \xrightarrow[]{M\to\infty} 0
    \]
    almost surely as a function of $s$, thus using Dominated convergence we conclude that
    \[\int_0^t \Big(\mathbb{I}_{[t^{\D_M}_{m_M}, t)}(s) |\Dot{x}_s - \frac{x_t - x_{t^{\D_M}_{m_M}}}{t - t^{\D_M}_{m_M}}| + \sum_{0 \leq l < m_M} \mathbb{I}_{[t^{\D_M}_l, t^{\D_M}_{l+1})}(s) |\Dot{x}_s - \frac{\Delta x_{t^{\D_M}_{l+1}}}{\Delta t^{\D_M}_{l+1}}|\Big) |\Dot{y}_{s}| ds = o(1)\]
  
  thus
  \[ \left| \kappa_\varphi^{x,y} (t) - \sigma_a^2 + \int_{0}^t (\sigma_A^2 V_{\varphi} (\Sigma^{x, y}_\varphi (s)) + \sigma_b^2) 
     \sprod{\Dot{x}_{s}}{\Dot{y}_{s}}_{\R^d}     ds\right| = 0 \]
  i.e
  \[ \kappa_\varphi^{x,y} (t) = \sigma_a^2 + \sigma_a^2 + \int_{0}^t (\sigma_A^2 V_{\varphi} (\Sigma^{x, y}_\varphi (s)) + \sigma_b^2) 
     \sprod{\Dot{x}_{s}}{\Dot{y}_{s}}_{\R^d}     ds
 \]
  }
\end{proof}

\subsubsection{Proof of Theorem \ref{thm:main-inhom-Kernels}: Part 1}

It is finally time to prove the first part \Cref{thm:main-inhom-Kernels} in the main body of the paper, which we restated in the appendix as \Cref{thm:main-inhom-Kernels_appendix}.

\begin{proof}[Proof of \Cref{thm:main-inhom-Kernels_appendix}]
    The proof is now just a matter of combining Theorem \ref{thm:inhom-discrete-phi-SigKer} and Theorem \ref{inhom_simple_conv}. Under our hypotheses Theorem \ref{thm:inhom-discrete-phi-SigKer} tells us that, for any subset $\X = \{x_1, \dots, x_n\} \subset \bX$, we have in distribution
\[
    \lim_{N \to \infty} \Psi_{\varphi}^{M, N}(\X)
    =
    \mathcal{N}(0,\kappa_{\D_M}(\X, \X))
\]
where $\kappa_{\D_M}(x_\alpha,x_\beta) = \kappa_{\D_M}^{x_{\alpha}, x_{\beta}}(1)$ for all $\alpha, \beta = 1, \dots, n$. Thus to conclude we just have to prove that, still in distribution, it holds
\[
    \lim_{M \to \infty} \mathcal{N}(0,\kappa_{\D_M}(\X, \X))
    =
    \mathcal{N}(0,\kappa_\varphi(\X, \X))
\]
or equivalently that
$$\lim_{M \to \infty} \kappa_{\D_M}^{x_{\alpha}, x_{\beta}}(1) = \kappa_{\varphi}^{x_{\alpha}, x_{\beta}}(1)$$
This last needed limit follows from Theorem \ref{inhom_simple_conv} with the sequence $\{\D_M \}_{M \in \N}$.
\end{proof}

We can now study some particular cases.
\begin{corollary}
    If $\varphi = id$, then 
    \begin{equation}\label{eqn:app_inhom-id-case}
        \kappa_{id}^{x,y} (t) = \big( \sigma_a^2 + \frac{\sigma_b^2}{\sigma_A^2}\big) \exp \left\{\sigma_A^2 \int_{\eta = 0}^s \sprod{\Dot{x}_{\eta}}{ \Dot{y}_{\eta}}_{\R^d} d\eta \right\} - \frac{\sigma_b^2}{\sigma_A^2}
    \end{equation}

    If $\varphi = ReLu$ and $x=y$, then 
    \begin{equation}\label{eqn:app_inhom-ReLU-case}
    \kappa_{ReLu}^{x,x} (t) = \big( \sigma_a^2 + \frac{2 \sigma_b^2}{\sigma_A^2}\big) \exp\left \{\frac{\sigma_A^2}{2} \int_{\eta = 0}^s \norm{\Dot{x}_{\eta}}_{\R^d}^2 d\eta \right\} - \frac{2 \sigma_b^2}{\sigma_A^2}
    \end{equation}
\end{corollary}

\begin{proof}
    This is just a matter of computation. For the case $\varphi = id$ one has
    \[
        \E_{Z \sim \mathcal{N} (0, \Sigma)} 
      [ \varphi(Z_1) \varphi(Z_2) ]  =
      \E_{Z \sim \mathcal{N} (0, \Sigma)}
      [ Z_1 Z_2 ] =
      [\Sigma]_1^2
    \]
    hence
    \[
        \kappa_{id}^{x,y} (t) =  \sigma_a^2  + 
      \int_{0}^{t} 
      \Big[ 
      \sigma_A^2 \kappa_{id}^{x,y} (s)
      + 
      \sigma_b^2 
      \Big] 
     \sprod{\Dot{x}_{s}}{ \Dot{y}_{s}}_{\R^d}   d s 
    \]

    Notice then that substituting (\ref{eqn:app_inhom-id-case}) for $K^{id}_{s} (x, y)$ in the integral leads to
    {
    \begin{align*}
          &
          \sigma_a^2  + 
          \int_{0}^{t} 
          \Big[ 
          \sigma_A^2 
          \Big\{ ( \sigma_a^2 + \frac{\sigma_b^2}{\sigma_A^2}) \exp \big \{\sigma_A^2 \int_{\eta = 0}^s \sprod{\Dot{x}_{\eta}}{ \Dot{y}_{\eta}}_{\R^d} d\eta \big\} - \frac{\sigma_b^2}{\sigma_A^2} \Big\}
          + 
          \sigma_b^2 
          \Big] 
         \sprod{\Dot{x}_{s}}{ \Dot{y}_{s}}_{\R^d}   ds 
         \\
         = &
         \sigma_a^2  + 
          \int_{0}^{t} 
          \sigma_A^2 
          \big( \sigma_a^2 + \frac{\sigma_b^2}{\sigma_A^2}\big) \exp\big \{\sigma_A^2 \int_{\eta = 0}^s \sprod{\Dot{x}_{\eta}}{ \Dot{y}_{\eta}}_{\R^d} d\eta \big\} 
         \sprod{\Dot{x}_{s}}{ \Dot{y}_{s}}_{\R^d}   ds 
         \\
         = &
         \sigma_a^2  + 
          \int_{0}^{t} 
          ( \sigma_a^2 + \frac{\sigma_b^2}{\sigma_A^2} )
          \partial_T \Big( \exp\big \{\sigma_A^2 \int_{\eta = 0}^T \sprod{\Dot{x}_{\eta}}{ \Dot{y}_{\eta}}_{\R^d} d\eta \big\} \Big)|_{T=s}   ds 
          \\
          = &
          \sigma_a^2  + 
           ( \sigma_a^2 + \frac{\sigma_b^2}{\sigma_A^2} )
          \Bigg[
          \exp\big \{\sigma_A^2 \int_{\eta = 0}^s \sprod{\Dot{x}_{\eta}}{ \Dot{y}_{\eta}}_{\R^d} d\eta \big\}
          - 1
          \Bigg] 
          \\
          = &
          ( \sigma_a^2 + \frac{\sigma_b^2}{\sigma_A^2} )
          \exp\big \{\sigma_A^2 \int_{\eta = 0}^s \sprod{\Dot{x}_{\eta}}{ \Dot{y}_{\eta}}_{\R^d} d\eta \big\}
          - \frac{\sigma_b^2}{\sigma_A^2}
    \end{align*}
    }
    which means, by uniqueness of solutions, that the thesis holds.

    For what concerns the case $\varphi = ReLU$ notice that for one dimensional Gaussian variables centered in the origin one has 
    \[
    \E_{Z \sim \mathcal{N} (0, \sigma^2)} 
      [ ReLU(Z)^2]  =
      \frac{1}{2}\E_{Z \sim \mathcal{N} (0, \sigma^2)}
      [ Z^2 ] =
      \frac{1}{2}\sigma^2
    \]
    hence since 
    {
    \[
    \mathbb{E}_{Z \sim \mathcal{N} (0, \kappa_{ReLU}^{x,x}(s) \mathbf{1})} 
      [ ReLU(Z_1) ReLU(Z_2) ] 
      =
      \E_{Z \sim \mathcal{N} (0, \kappa_{ReLU}^{x,x}(s))} 
      [ ReLU(Z)^2] 
    \]
    one has
    \[
    \kappa_{ReLu}^{x,x} (t) =  \sigma_a^2  + 
      \int_{0}^{t} 
      \Big[ 
      \frac{\sigma_A^2}{2}\kappa_{ReLu}^{x,x} (s)
      + 
      \sigma_b^2 
      \Big] 
     \sprod{\Dot{x}_{s}}{ \Dot{x}_{s}}_{\R^d}   d s 
    \]
    which equals $\kappa_{id}^{x,x}(t;\sigma_a, \frac{\sigma_A}{\sqrt{2}}, \sigma_b)$.
    }
\end{proof}

We conclude this section by proving Remark \ref{rmk:inhom_scaling} about the path scaling symmetry mentioned in the main paper.

\begin{lemma}\label{app:lemma:inhom_simpler_form}
For all choices $(\sigma_a, \sigma_A, \sigma_b)$ and for all $\varphi$ as in Theorem \ref{thm:main-inhom-Kernels} we have, with abuse of notation and the obvious meaning, that
\[
    \kappa_{\varphi}^{x,y}(t; \sigma_a, \sigma_A, \sigma_b) = \kappa_{\varphi}^{\sigma_A x, \sigma_A y}(t; \sigma_a, 1,{\sigma_b}{\sigma_A^{-1}})
\]
\end{lemma}

\begin{proof}
    We have
    \begin{align*}
          &\kappa_{\varphi}^{x,y}(t; \sigma_a, \sigma_A, \sigma_b) \\
          &= 
          \sigma_a^2  + 
          \int_{\eta = 0}^s 
          \Big[ 
          \sigma_A^2 \mathbb{E}_{Z \sim \mathcal{N} (0, {\Sigma}_{\varphi}^{x,y}{(\eta; \sigma_a, \sigma_A, \sigma_b)}} 
          [ \varphi(Z_1) \varphi(Z_2) ] 
          + 
          \sigma_b^2 
          \Big] 
         \sprod{\Dot{x}_{\eta}}{ \Dot{y}_{\eta}}_{\R^d}   d \eta
         \\
           &= 
          \sigma_a^2  + 
          \int_{\eta = 0}^s 
          \sigma_A^2 \Big[ 
          \mathbb{E}_{Z \sim \mathcal{N} (0, {\Sigma}_{\varphi}^{x,y}{(\eta; \sigma_a, \sigma_A, \sigma_b)}} 
          [ \varphi(Z_1) \varphi(Z_2) ] 
          + 
          \frac{\sigma_b^2}{\sigma_A^2} 
          \Big] 
         \sprod{\Dot{x}_{\eta}}{ \Dot{y}_{\eta}}_{\R^d}   d \eta \\
         &=
         \sigma_a^2  + 
          \int_{\eta = 0}^s 
           \Big[ 
          \mathbb{E}_{Z \sim \mathcal{N} (0, {\Sigma}_{\varphi}^{x,y}{(\eta; \sigma_a, \sigma_A, \sigma_b)}} 
          [ \varphi(Z_1) \varphi(Z_2) ] 
          + 
          \frac{\sigma_b^2}{\sigma_A^2} 
          \Big] 
         \sprod{\sigma_A\Dot{x}_{\eta}}{\sigma_A \Dot{y}_{\eta}}_{\R^d}   d \eta
    \end{align*}
    and
    \begin{align*}
          &\kappa_{\varphi}^{\sigma_A x, \sigma_A y}(t; \sigma_a, 1,{\sigma_b}{\sigma_A^{-1}}) 
          \\
          &= \sigma_a^2  + 
          \int_{\eta = 0}^s 
           \Big[ 
          \mathbb{E}_{Z \sim \mathcal{N} (0, {\Sigma}_{\varphi}^{\sigma_A x, \sigma_A y}{(\eta; \sigma_a, 1,\frac{\sigma_b}{\sigma_A})}} 
          [ \varphi(Z_1) \varphi(Z_2) ] 
          + 
          \frac{\sigma_b^2}{\sigma_A^2} 
          \Big] 
         \sprod{\sigma_A\Dot{x}_{\eta}}{\sigma_A \Dot{y}_{\eta}}_{\R^d}   d \eta
    \end{align*}
    Thus the respective triplets solve the same equation and we can conclude by uniqueness.
\end{proof}

\subsection{The infinite-depth-then-width regime}

As mentioned in the paper, it is natural to ask what happens if the order of the width-depth limits in is reversed.

\subsubsection{Proof of Theorem \ref{thm:inhom-inf-D-limit}} 

We begin by proving Theorem \ref{thm:inhom-inf-D-limit} which we restate for the reader's convenience. We will follow arguments used in \cite{hayou2022infinite}, extending the results obtained therein.

\begin{theorem}\label{thm:inhom-inf-D-limit_appendix}
    Let $\{\D_M\}_{M \in \N}$ be a sequence of partitions of $[0,1]$ such that $|\D_M| \downarrow 0$ as $M \to \infty$. Assume the activation function $\varphi$ is  Lipschitz and linearly bounded. Let $\rho_M(t) := \sup \{ s \in \D_M : s \leq t\}$. For any path $x \in \bX \cap C^{1,\frac{1}{2}}$, where $C^{1,\frac{1}{2}}$ denotes the set of $C^1$ paths with $\frac{1}{2}$-H\"{o}lder derivative, the $\R^N$-valued process $t \mapsto \cS^{M,N}_{\rho_M(t)}(x)$ converges in distribution, as $M \to \infty$, to the solution $S^{N}(x)$ of the following SDE    
    \begin{equation}\label{eqn:app-inhom-inf-depth-1}
        d\cS^{N}_t(x) = 
        \sum_{j=1}^d 
        \frac{\sigma_A}{\sqrt{N}} \Dot{x}^j_t dW^j_t \varphi(\cS^N_t(x)) + \sigma_b \dot{x}^j_t dB^j_t 
    \end{equation}
    with $\cS^N_0(x)=a$ and where $W^j \in \R^{N \times N}$ and $B^j \in \R^N$ are independent Brownian motions for $j\in\{1,...,d\}$.
\end{theorem}



\begin{proof}
    We will transform Equation (\ref{eqn:app-inhom-inf-depth-1}) in a usual SDE form to then use the classical result \cite{kloeden1992numerical}[Theorem 10.2.2] to prove the convergence of Euler Discretizations to the unique solution. We first show that Equation (\ref{eqn:app-inhom-inf-depth-1}) can be re-written as follows
    \[
        \cS^N_t(x) = \sigma_a^2 + \int_0^t \sigma_x(s,\cS^N_s(x)) d Z_s
    \]
    for a Brownian Motion $Z_t \in \R^{\Gamma_{d,N}}$ with $\Gamma_{d,N} := d N (N+1)$.

    Take in fact 
    $$Z_s := [([W_s^1]^1)^T,([W_s^1]^2)^T,\dots,([W_s^1]^N)^T,(B^1_s)^T,([W_s^2]^1)^T,\dots,(B^d_s)^T]^T$$
    or equivalently
    \[
        [Z_s]_i = 
        \begin{cases}
            [W_s^k]^m & \text{if } m = 1,\dots,N
            \\
            B^k_s & \text{if } m = N + 1
        \end{cases}
    \]
    for $i = N(N+1)(k-1) + N(m-1) + 1,\dots, N(N+1)(k-1) + Nm$
    \footnote{meaning that $[Z_s]_i$ is considered as the element of $\R^N$ corresponding to these indices.}. In addition, define $\sigma_x : [0,1] \times \R^N \to \R^{N\times\Gamma_{d,N}}$ as
    \[
     [\sigma_x(s,y)]_i^j
     =
     \begin{cases}
            \frac{\sigma_A}{\sqrt{N}} \Dot{x}^k_s [\varphi(y)]_m Id_{N \times N} & \text{if } m = 1,\dots,N
            \\
            \sigma_b \Dot{x}^k_s Id_{N \times N} & \text{if } m = N + 1
    \end{cases}
    \]
    for $i = 1, \dots, N$ and $j = N(N+1)(k-1) + N(m-1) + 1,\dots, N(N+1)(k-1) + Nm$.
    It is then just a matter of checking the required conditions to apply \cite{kloeden1992numerical}[Theorem 10.2.2]:
    \begin{enumerate}
        \item There is a $K > 0$ such that $\forall t \in [0,1]. \forall y, y' \in \R^N$ one has 
        \[
        \norm{\sigma(t,y) - \sigma(t,y')}_{F} \leq K \norm{y - y'}_{\R^N}
        \]
        \item There is a $K' > 0$ such that $\forall t \in [0,1]. \forall y \in \R^N$ one has 
        \[
        \norm{\sigma(t,y)}_{F} \leq K' (1 + \norm{y }_{\R^N})
        \]
        \item There is a $K'' > 0$ such that $\forall t, s\in [0,1]. \forall y \in \R^N$ one has 
        \[
        \norm{\sigma(t,y) - \sigma(s,y)}_{F} \leq K'' (1 + \norm{y}_{\R^N}) |t-s|^{\frac{1}{2}}
        \]
    \end{enumerate}

    As a first step note how
    \[
    \norm{\sigma(t,y)}_{F}^2 = 
    N \sum_{k=1}^d (\Dot{x}^k_t)^2 \big( \sigma_b^2 +\sum_{m=1}^N  \frac{\sigma_A^2}{N} ([\varphi(y)]_m )^2\big) =
    N \norm{\Dot{x}_t}_{\R^d}^2 (\sigma_b^2 + \sigma_A^2\frac{\norm{\varphi(y)}_{\R^N}^2}{N})
    \]
    thus using sublinearity of $\varphi$ for some $M>0$ one has
    \[
    \norm{\sigma(t,y)}_{F}^2 \leq N \norm{\Dot{x}_t}_{\R^d}^2 (\sigma_b^2 + \frac{\sigma_A^2}{N} M (1 + \norm{y}_{\R^N})^2)
    \]
    from which we easily get condition (ii). 
    
    For condition (i) note that 
    \[
     \norm{\sigma(t,y) - \sigma(t,y')}_{F}^2 = N \norm{\Dot{x}_t}_{\R^d}^2 \sigma_A^2\frac{\norm{\varphi(y) - \varphi(y')}_{\R^N}^2}{N}
    \]
    thus one just uses Lipschitz property of $\varphi$ with Lemma \ref{app:lipR_to_lipRn}.

    Finally for (iii) one computes just as before
    \[
    \norm{\sigma(t,y) - \sigma(s,y)}_{F}^2 = N \norm{\Dot{x}_t - \Dot{x}_s}_{\R^d}^2 (\sigma_b^2 + \sigma_A^2\frac{\norm{\varphi(y)}_{\R^N}^2}{N})
    \]
    and using H\"{o}lder property of $x$ and linear bound on $\varphi$ one has, for some $\Tilde{M} > 0$
    \[
    \norm{\sigma(t,y) - \sigma(s,y)}_{F}^2 \leq N(\sigma_b^2 + \frac{\sigma_A^2}{N} M (1 + \norm{y}_{\R^N})^2)\Tilde{M}^2|t-s|
    \]
    hence the concluding inequality
    \[
    \norm{\sigma(t,y) - \sigma(s,y)}_{F} \leq  (N\Tilde{M}(\sigma_b^2 + \frac{\sigma_A^2}{N} M)) (1 + \norm{y}_{\R^N})|t-s|^{\frac{1}{2}}
    \]

    Because in general
    $$
        \frac{\Delta x^k_{t^{\D_M}_l}}{\Delta t^{\D_M}_l} \neq \Dot{x}^k_{t^{\D_M}_{l-1}}
    $$
    the recursive equation
    \begin{equation*}
        \begin{gathered}
            \cS^{M,N}_{t^{\D_M}_{l}}(x) = \cS^{M,N}_{t^{\D_M}_{l-1}}(x) + \sum_{k=1}^d \frac{\Delta x^k_{t^{\D_M}_l}}{\Delta t^{\D_M}_l} (\frac{\sigma_A}{\sqrt{N}}\sqrt{\Delta t^{\D_M}_l} W_{k,l} \varphi(\cS^{M,N}_{t^{\D_M}_{l-1}}(x)) + \sigma_b\sqrt{\Delta t^M_l} B_{k,l} ) 
        \end{gathered}
    \end{equation*}
    is not the Euler discretization of the above SDE.
    However, after classical though tedious calculations it is easy to show that 
    \[
        \lim_{M \to \infty} \sup\limits_{t\in [0,1]} \E[|\Tilde{\cS}^{M, N}_t(x) - \cS^{M, N}_t(x)|^2] = 0
    \]
    where $\Tilde{\cS}^{M, N}_t(x)$ is defined as ${\cS}^{M, N}_t(x)$ but substituting the difference quotients with the actual derivatives.

\end{proof}

\begin{remark}
    By considering the concatenated system 
    \[
        [\cS^{N}_t(x_i)]_{x_i \in \X} = [a]_{x_i \in \X} + \int_0^t [\sigma_{x_i}(s,\cS^N_s(x_i))]_{x_i \in \X} dZ_s
    \]
    we straightforwardly extend the previous result to the case with multiple inputs considered at the same time.
\end{remark}

\subsubsection{Infinite-depth-then-width limit: $\varphi = id$}\label{app:subsec:inf_d_lim_id}

Here we directly prove that, in the case $\varphi = id$, the covariances of the infinite-depth networks converge to $\kappa_{id}$ as the width increases without bounds.

\begin{proposition}
Let $\varphi = id$. For any subset $\X = \{x_1, \dots, x_n\} \subset \bX$.
Then for all integers $N\geq 1$ and for all $k,m = 1,\dots,d$ that the following convergence holds
\[
    \lim_{M \to \infty} \E\left[{\Psi_{id}^{M, N}(\X)}{\Psi_{id}^{M, N}(\X)}\right] = \kappa_{id}(\X,\X)
\]
\end{proposition}

\begin{proof}
    To start notice that
    \[
    \E[ \Psi_{id}^{M, N}(x_k)\Psi_{id}^{M, N}(x_m) ]=
    \frac{1}{N}\E\left[ \sprod{{\cS^{M,N}_{1}}(x_k)}{{\cS^{M,N}_{1}}(x_m)}\right]
    \]
    
    Thanks to the extension of Theorem \ref{thm:inhom-inf-D-limit_appendix} to the multi-input case we have 
    \[
        \lim_{M \to \infty} \frac{1}{N}\E\left[\sprod{{\cS^{M,N}_{1}}(x_k)}{{\cS^{M,N}_{1}}(x_m)}\right] = 
        \frac{1}{N}\E\left[ \sprod{\cS^{N}_{1}(x_k)}{\cS^{N}_{1}(x_m)}\right]
    \]
    {
    and for $x,y \in \X$
    \begin{align*}
        & \frac{1}{N}\E\left[ \sprod{\cS^{N}_{1}(x)}{\cS^{N}_{1}(y)}\right]
        \\ 
        = & \frac{1}{N}\E\left[ \sprod{a + \sum_{i=1}^d \int_0^1 
        \frac{\sigma_A}{\sqrt{N}} \Dot{x}^i_t dW^i_t \cS^N_t(x) + \sigma_b \dot{x}^i_t dB^i_t }{\cS^{N}_{1}(y)}\right]
        \\
        = & 
        \frac{1}{N} \E\left[  \norm{a}^2_{\R^N}\right] 
        + \frac{1}{N} \sum_{i,j=1}^d \sigma_b^2 \E\left[ \sprod{\int_0^1 \dot{x}^i_t dB^i_t}{\int_0^1 \dot{y}^j_t dB^j_t}\right]
        \\
        & + \frac{1}{N} \sum_{i,j=1}^d \frac{\sigma_A^2}{N} 
        \E\left[ \sprod{\int_0^1 \Dot{x}^i_t dW^i_t \cS^N_t(x)}{\int_0^1 \Dot{y}^j_t dW^j_t \cS^N_t(y)}\right]
        \\
        = & 
        \frac{1}{N} (N \sigma_a^2)
        + \sum_{i,j=1}^d \frac{\sigma_b^2}{N} \sum_{\alpha=1}^N 
        \E\left[ \int_0^1 \dot{x}^i_t d[B^i_t]_{\alpha} \cdot \int_0^1 \dot{y}^j_t d[B^j_t]_{\alpha}\right]
        \\
        & + \frac{1}{N} \sum_{i,j=1}^d \frac{\sigma_A^2}{N} \sum_{\alpha,\beta,\gamma = 1}^N
        \E\left[ 
        \int_0^1 \Dot{x}^i_t [\cS^N_t(x)]_{\alpha} d[W^i_t]_{\gamma}^{\alpha}  
        \cdot 
        \int_0^1 \Dot{y}^j_t [\cS^N_t(y)]_{\beta} d[W^j_t]^{\beta}_{\gamma} 
        \right]
        \\
        = & 
        \sigma_a^2 
        + \frac{\sigma_b^2}{N} \sum_{l=1}^N \sum_{i=1}^d \E\left[ \int_0^1 \dot{x}^i_t \dot{y}^i_t dt\right]
        \\
        & + \frac{\sigma_A^2}{N^2} 
        \sum_{\alpha,\gamma = 1}^N \sum_{i=1}^d 
        \E\left[ 
        \int_0^1 \Dot{x}^i_t [\cS^N_t(x)]_{\alpha}  \Dot{y}^i_t [\cS^N_t(y)]_{\alpha} dt
        \right]
        \\
        = &
        \sigma_a^2 + \sigma_b^2 \int_0^1 \sprod{\dot{x}_t}{\dot{y}_t} dt
        + \sigma_A^2 \int_0^1 \E\left[ \frac{1}{N} \sprod{\cS^N_t(x)}{\cS^N_t(y)}_{\R^N}\right] \sprod{\dot{x}_t}{\dot{y}_t} dt
    \end{align*}
    where the penultimate equality follows from It\^{o}'s isometry. Hence}
    \[
    \frac{1}{N}\E\left[ \sprod{\cS^{N}_{1}(x_k)}{\cS^{N}_{1}(x_m)}\right] =
    \sigma_a^2 + 
    \int_0^1
    \left( 
    \frac{\sigma_A^2}{N}  \E[ \sprod{\cS^{N}_{1}(x_k)}{\cS^{N}_{1}(x_m)}] + \sigma_b^2 \right)
    \sprod{(\Dot{x}_k)_s}{(\Dot{x}_m)_s}_{\R^d} ds
    \]
    Defining $K^N_t(x_k,x_m) :=  \frac{1}{N}\E[\sprod{\cS^{N}_{t}(x_k)}{\cS^{N}_{t}(x_m)}$ and noticing that we can repeat the previous arguments for all $t \in [0,1]$, not only for $t=1$, by considering for example piecewise constant extensions of the $\cS^{M,N}_{\cdot}$ we see that
    \[
        K^N_t(x_k,x_m) = \sigma_a^2 + 
    \int_0^t
    \big( 
    \sigma_A^2 K^N_s(x_k,x_m) + \sigma_b^2 \big)
    \sprod{(\Dot{x}_k)_s}{(\Dot{x}_m)_s}_{\R^d} ds
    \]
    meaning that $K^N_t(x_k,x_m)$ {solves Equation (\ref{app:eqn:inhom-simple-conv-PDE})}. 

    Since we showed in the proof of Theorem \ref{inhom_simple_conv} that the unique solution to the equation is $\kappa_{id}^{x_k,x_m}(1)$ we conclude.
    
\end{proof}

Directly proving the convergence in distribution of the infinite-width networks to centered Gaussians with covariance function $\kappa_{id}$ is difficult since the networks are not Gaussian processes. 
However we believe that it is possible to employ McKean-Vlasov arguments to prove that they are at the limit, to then obtain a commmutativity of the limits; we leave the exploration of this direction to future work.

\subsection{Commutativity of Limits}\label{app:sub:inhom_commut}

In this section we are going to prove the commutativity of limits in the inhomogeneous case. To do this we are going to proceed similarly to \cite{hayou2023width} which proves the same result in a much restricted case.

Note that if $\varphi$ is $K$-Lip \emph{i.e.} $|\varphi(x) - \varphi(y)| \leq K|x - y|$ and $\varphi(0) = 0$ then 
\[
\norm{\varphi(x)}^2 = \sum |\varphi(x_i)|^2 \leq  \sum K^2|x_i|^2 
\leq K^2\norm{x}^2
\]
hence 
\[
\norm{\varphi(x)} \leq K\norm{x}
\]

\begin{assumption}
    In this subsection $\varphi$ is considered Lipschitz and with $\varphi(0) = 0$.
\end{assumption}

Recall Theorem \ref{thm:inhom-inf-D-limit_appendix} and let $\tilde{\cS}^{M, N}_t(x)$ be the Euler discretization of (\ref{eqn:app-inhom-inf-depth-1}). 
Then one has (see the proof of \cite{kloeden1992numerical}[10.2.2])

\begin{equation}\label{app:eqn:euler_convergence}
\sup_{N \geq 1} \frac{1}{N} \E\left[ \sup\limits_{t \in [0,1]}\norm{\tilde{\cS}^{M, N}_t(x) - {\cS}^{N}_t(x)}^2 \right] 
\leq C |\D_M|
\end{equation}

We can say the same for ${\cS}^{M, N}_t(x)$ instead of $\tilde{\cS}^{M, N}_t(x)$ \emph{i.e} when the derivatives are replaced by difference quotients on the successive interval.

\begin{proposition}
    The following inequality holds
     \[
         \sup_{N \geq 1} \frac{1}{N}\sup\limits_{t\in [0,1]}\E\left[\norm{{\cS}^{M, N}_t(x) - \cS^{N}_t(x)}^2\right]  \leq \tilde{C} |\D_M|
    \]
\end{proposition}

\begin{proof}
    The bounding constant in equation (\ref{app:eqn:euler_convergence}) depends on the path $x$ only trough $\norm{\dot x}_{\infty, [0,1]}$. In particular the bound is uniform on bounded sets, with respect to this norm.

    In particular we can interpolate $\{(t_m, x_{t_m})\}_{m = 0, \dots, |\D_M|}$ in such a way to have the resulting map $\bar x$ with
    \[
        {\dot {\bar x}}_{t_m} = \frac{x_{t_{m+1}} - x_{t_m}}{{t_{m+1}} - t_m} = \frac{\Delta x_{t_{m+1}}}{\Delta t_{m+1}}
    \]
    such that definitely in $M$ $\dot x$ is $\frac{1}{2}$-Holder continuous and
    \[
     \norm{\dot {\bar x} - \dot x}_{\infty,[0,1]}^2 \lesssim |\D_M|
    \]

    One can for example interpolate the points with a polynomial close enough to $\dot x$, being the polynomial defined on the compact $[0,1]$ it is Lipschitz on $[0,1]$ hence $\frac{1}{2}$-Holder continuous.
    The existence of such a polynomial follows from 
    \[
        |\frac{\Delta x_{t_{m+1}}}{\Delta t_{m+1}} - {\dot x}_{t_m}| \leq \frac{1}{\Delta t_{m+1}} \int_{t_m}^{t_{m+1}} |\dot x_s - {\dot x}_{t_m}| ds \leq \frac{2}{3} C_{\frac{1}{2}-Ho}\sqrt{\Delta t_{m+1}}
    \]

    Then $\cS^{M,N}(x)$ will be exactly the Euler discretisation of $\cS^N(\bar x)$ on $\D_M$, hence
    \[
        \sup_{N \geq 1} \frac{1}{N} \E\left[ \sup\limits_{t \in [0,1]}\norm{{\cS}^{M, N}_t(x) - {\cS}^{N}_t(\bar x)}^2 \right] 
        \leq C |\D_M|
    \]

    To conclude we just need to prove 
    \[
         \sup_{N \geq 1} \frac{1}{N}\sup\limits_{t\in [0,1]}\E\left[\norm{{\cS}^{N}_t(\bar x) - \cS^{N}_t(x)}^2\right]  \leq \bar{C} |\D_M|
    \]
    since
    \begin{equation*} \begin{gathered}
        \sup\limits_{t\in [0,1]}\E\left[\norm{{\cS}^{M, N}_t(x) - \cS^{N}_t(x)}^2\right]
        \leq
        \\
        2 \left( \sup\limits_{t\in [0,1]}\E\left[\norm{{\cS}^{M, N}_t(x) - \cS^{N}_t(\bar x)}^2\right]
        + 
        \sup\limits_{t\in [0,1]}\E\left[\norm{{\cS}^{N}_t(\bar x) - \cS^{N}_t(x)}^2\right] \right) \leq
        \\
        2 \left( \E\left[\sup\limits_{t\in [0,1]} \norm{{\cS}^{M, N}_t(x) - \cS^{N}_t(\bar x)}^2\right]
        + 
        \sup\limits_{t\in [0,1]}\E\left[\norm{{\cS}^{N}_t(\bar x) - \cS^{N}_t(x)}^2\right] \right) \leq
        \\
        2(C + \bar C) |\D_M|
    \end{gathered}\end{equation*}

    Then note how 
    \begin{equation*}
        \begin{gathered}
            \cS^{N}_t(x) =  \cS^{N}_0(x) + \sum_{k=1}^d  \int_0^t
        \frac{\sigma_A}{\sqrt{N}} \Dot{x}^k_s dW^k_s \varphi(\cS^N_s(x)) + \sigma_b \dot{x}^k_s dB^k_s
        \\
        =  \cS^{N}_0(x) + \sum_{k=1}^d  \sum_{j=1}^N \int_0^t \Dot{x}^k_s \left(
        \frac{\sigma_A}{\sqrt{N}} \varphi([\cS^N_s(x)]_j) d[W^k]^j_s + \sigma_b dB^k_s \right)
        \end{gathered}
    \end{equation*}
    thus
    \begin{align*}
        {\cS}^{N}_t(\bar x) - \cS^{N}_t(x) 
        &=
         \sum_{k=1}^d  \sum_{j=1}^N \int_0^t (\Dot{\bar x}^k_s - \Dot{x}^k_s) \left(
        \frac{\sigma_A}{\sqrt{N}} \varphi([\cS^N_s(\bar x)]_j) d[W^k]^j_s + \sigma_b dB^k_s \right)
        \\
        & + 
         \sum_{k=1}^d  \sum_{j=1}^N \int_0^t \Dot{x}^k_s \frac{\sigma_A}{\sqrt{N}} \left(
         \varphi([\cS^N_s(\bar x)]_j)- \varphi([\cS^N_s(x)]_j)\right) d[W^k]^j_s 
    \end{align*}
    which leads to
    \begin{align*}
        \E\left[\norm{{\cS}^{N}_t(\bar x) - \cS^{N}_t(x)}^2\right]
        & = 
        N \sum_{k=1}^d \int_0^t  (\Dot{\bar x}^k_s - \Dot{x}^k_s)^2  \left(
        \frac{\sigma_A^2}{N} \E[\norm{\varphi(\cS^N_s(\bar x))}^2] + \sigma_b^2 \right) ds
        \\
        &+
        N \sum_{k=1}^d   \int_0^t (\Dot{x}^k_s)^2 \frac{\sigma_A^2}{N} \norm{
         \E[\varphi(\cS^N_s(\bar x))- \varphi(\cS^N_s(x))}^2] ds 
        \\
        & \leq 
        NK \sum_{k=1}^d \int_0^t  (\Dot{\bar x}^k_s - \Dot{x}^k_s)^2  \left(
        \frac{\sigma_A^2}{N} \E[\norm{\cS^N_s(\bar x)}^2] + \sigma_b^2 \right) ds
        \\
        &+
        NK \sum_{k=1}^d   \int_0^t (\Dot{x}^k_s)^2 \frac{\sigma_A^2}{N} \E[\norm{
         \cS^N_s(\bar x)- \cS^N_s(x)}^2 ] ds  
         \\
        & \leq 
        NK \norm{\Dot{\bar x} - \Dot{x}}^2_{\infty, [0,1]} \sum_{k=1}^d \int_0^t  \left(
        \frac{\sigma_A^2}{N} \E[ \norm{\cS^N_s(\bar x)}^2] + \sigma_b^2 \right) ds
        \\
        &+
        NK \norm{\Dot{x}}_{\infty, [0,1]}^2 \sum_{k=1}^d   \int_0^t \frac{\sigma_A^2}{N} \E[ \norm{ \cS^N_s(\bar x)- \cS^N_s(x)}^2] ds  
    \end{align*}

    Using the fact that, again \cite{kloeden1992numerical}[10.2.2],
    \[
    \frac{1}{N} \E[ \norm{\cS^N_s(\bar x)}^2] \leq \tilde K_{\bar x}
    \]
    with $\tilde K_{\bar x}$ only depending on $\norm{\dot{\bar x}}_{\infty}$
    we obtain, via Gronwall, a bound of type
    \begin{equation}
        \E\left[\norm{{\cS}^{N}_t(\bar x) - \cS^{N}_t(x)}^2\right] \leq \norm{\Dot{\bar x} - \Dot{x}}^2_{\infty, [0,1]} \bar K
    \end{equation}
    where $\bar K$ only depends on $\norm{\dot{\bar x}}_{\infty} + \norm{\dot{ x}}_{\infty}$. By our choice of $\bar x$ we conclude.
\end{proof}


We can finally state the main result:

\begin{theorem}
If the activation function is Lipschitz and $\varphi(0) = 0$ then there is a constant $C$ depending only on $\norm{\dot x}_{\infty, [0,1]}$ in an increasing fashion such that:  
    \[
    \sup_{N \geq 1} \sup_{t \in [0,1]} \mathcal{W}_1(\mu_{t}^{M, N}(x), \mu_{t}^{N}(x)) \leq \Bar{C} \sqrt{|\D_M|}
    \]
    where $\mu_{t}^{M,N}$ is the distribution of any coordinate of ${\cS}^{M, N}_t(x)$ and $\mu_{t}^{N}(x)$ that of any coordinate of ${\cS}^{N}_t(x)$ (they are identically distributed).
\end{theorem}

\begin{proof}
    Let $G: \R \to \R$ be 1-Lipschitz, we have
    \begin{equation*}
        \begin{gathered}
        \E[|G([{\cS}^{M, N}_t(x)]_1) - G([{\cS}^{N}_t(x)]_1)|] \leq
        \E\left[|[{\cS}^{M, N}_t(x)]_1 - [{\cS}^{N}_t(x)]_1|\right] \leq
        \\
        \left(\E\left[|[{\cS}^{M, N}_t(x)]_1 - [\tilde{\cS}^{M, N}_t(x)]_1|^2\right]\right)^{\frac{1}{2}} =
        \\
        \left(\frac{1}{N} \sum_{\alpha = 1}^N \E\left[|[{\cS}^{M, N}_t(x)]_{\alpha} - [{\cS}^{N}_t(x)]_{\alpha}|^2\right]\right)^{\frac{1}{2}}  \leq
        \\
        \left(\frac{1}{N} \E\left[ \norm{{\cS}^{M, N}_t(x) - {\cS}^{N}_t(x)}^2\right]\right)^{\frac{1}{2}}  \leq
        \\
        \left(\frac{1}{N} \sup_{t \in [0,1]} \E\left[\norm{{\cS}^{M, N}_t(x) - {\cS}^{N}_t(x)}^2\right]\right)^{\frac{1}{2}}  \leq
        \sqrt{{C}|\D_M|}
        \end{gathered}
    \end{equation*}
    hence 
    \[
    \mathcal{W}_1(\mu_{t}^{M, N}(x), \mu_{t}^{N}(x)) \leq \sqrt{{C}|\D_M|}
    \]
\end{proof}

Note that the entries have the same distribution as the normal projections since for $t \in \R$ 
\begin{align*}
    & \E[\exp\{i t \sprod{v^N}{\cS^{M,N}}\}] = \E \left[ \E[\exp\{i t \sprod{v^N}{\cS^{M,N}}\}| \cS^{M,N}] \right] 
    \\
    & = \E \left[ \E[\exp\{i \sprod{v^N}{t\cS^{M,N}}\}| \cS^{M,N}] \right] 
    = \E \left[\exp\{-\frac{t^2}{2N}\norm{\cS^{M,N}}^2\} \right] 
\end{align*}
and 
\begin{align*}
    & \E[\exp\{i t[\cS^{M,N}]_1\}] = \E \left[ \E[\exp\{i t \sprod{v^N}{\cS^{M,N}}\}| \cS^{M,N}] \right] 
    \\
    & = \E \left[ \E[\exp\{i \sprod{v^N}{t\cS^{M,N}}\}| \cS^{M,N}] \right] 
    = \E \left[\exp\{-\frac{t^2}{2N}\norm{\cS^{M,N}}^2\} \right] 
\end{align*}
thus giving us

\begin{theorem}
If the activation function is Lipschitz and $\varphi(0) = 0$ then there is a constant $C$ depending only on  $\norm{\dot x}_{\infty, [0,1]}$ in an increasing fashion such that:  
    \[
    \sup_{N \geq 1} \sup_{t \in [0,1]} \mathcal{W}_1(\mu_{t}^{M, N}(x), \mu_{t}^{N}(x)) \leq \Bar{C} \sqrt{|\D_M|}
    \]
    where $\mu_{t}^{M,N}$ is the distribution of $\sprod{v^N}{{\cS}^{M, N}_t(x)}$ and $\mu_{t}^{N}(x)$ that of $\sprod{v^N}{{\cS}^{N}_t(x)}$ for some independent vector with iid entries $[v^N]_{\alpha} \sim \mathcal{N}(0, \frac{1}{N})$.
\end{theorem}

\begin{remark}
With the same proof we have also the convergence of the laws of the rescaled processes $\frac{1}{\sqrt{N}}{\cS}^{M, N}_t(x)$.
\end{remark}

{
\begin{remark}
    Note that the exact arguments, being of $L^2$ type, can be repeated for the "stacked" vector $(\cS^{N,M}(x_1), \dots, \cS^{N,M}(x_N))$ extending (qualitatively) the bounds to the multi-input case.
\end{remark}
}

\subsubsection{Proof of Theorem \ref{thm:main-inhom-Kernels}: Part 2}

\begin{corollary}[Thm. \ref{thm:main-inhom-Kernels}]\label{app:cor:inhom_GP_second_part}
    If the activation function is Lipschitz and $\varphi(0) = 0$ then the limits in Thm. \ref{thm:main-inhom-Kernels_appendix} commute.
\end{corollary}

\begin{proof}
    By the classical Moore-Osgood theorem we need to prove that one of the two limits is uniform in the other, for example that the limit in distribution as $M \to \infty$ is uniform in $N$ in some metric which describes convergence in distribution. 
    But this is just the content of the previous results{, extended to the multi-input case.}
\end{proof}

\newpage

\section{Proofs for homogeneous controlled ResNets}

In this section we prove all the results for the homogeneous case. Recall that a \emph{randomly initialised, $1$-layer homogeneous controlled ResNet} $\Phi_\varphi^{M, N} : \bX \to \R$ is defined as follows
\begin{equation*}
    \Phi_\varphi^{M, N}(x) := \sprod{\phi}{S^{M, N}_{t_M}(x)}_{\R^N}
\end{equation*}
where $\phi \in \R^N$ is the random vector  $[\phi]_\alpha \iid \mathcal{N}(0,\frac{1}{N})$, and where the random functions $S^{M,N}_{t_i} : \bX \to \R^N$ satisfy the following recursive relation
\begin{equation*}
    S^{M,N}_{t_{i+1}} = S^{M,N}_{t_i} + \sum_{k=1}^d \big( A_{k} \varphi(S^{M,N}_{t_i}) + b_{k} \big) \Delta x^k_{t_{i+1}}
\end{equation*}
with initial condition $S_{t_0} = a$ with $[a]_{\alpha} \iid \mathcal{N}(0,\sigma_a^2)$, and Gaussian weights $A_{k} \in \R^{N \times N}$ and biases $b_{k} \in \R^N$ sampled independently according to
\begin{equation*}
    [A_{k}]_{\alpha}^{\beta} \iid \mathcal{N}\left(0,\frac{\sigma_{A}^2}{N}\right), \quad 
    [b_{k}]_{\alpha} \iid \mathcal{N}\left(0,\sigma_b^2\right)
\end{equation*}

This section too will be subdivided in three main parts: in the first we consider the infinite-width-then-depth limit, in the second we reverse the order and consider the infinite-depth-then-width limit and in the third one we prove that the limits can be exchanged.

\subsection{The infinite-width-then-depth regime}\label{appendix:hom-WD-limit}

The main goal is that of proving the first part of Theorem \ref{thm:phi-SigKer}, which we restate here:

\begin{theorem}
Let $\{\D_M\}_{M \in \N}$ be a sequence of partitions of $[0,1]$ such that $|\D_M| \to 0$ as $M \to \infty$. Let the activation function $\varphi$ be linearly bounded, absolutely continuous and with exponentially bounded derivative.
For any subset $\X = \{x_1, \dots, x_n\} \subset \bX$ of paths the following convergence in distribution holds
\begin{equation}
    \lim_{M \to \infty} \lim_{N \to \infty} \Phi_{\varphi}^{M, N}(\X) =
    \mathcal{N}(0,\K_\varphi(\X,\X))
\end{equation}
where the positive semidefinite kernel $\K_\varphi : \bX \times \bX \to \R$ is defined for any two paths $x,y \in \bX$ as $\K_\varphi(x,y) = \K_\varphi^{x,y}(1,1)$,  where $\K_\varphi^{x,y} : [0,1]^2 \to \R$ is the unique solution of the following differential equation
\begin{equation}\label{eqn:hom_kernel_appendix}
\partial_s \partial_t \K^{x,y}_{\varphi} = 
    \Big[        
    \sigma_A^2 V_\varphi 
    \left(\Sigma_\varphi^{x, y} \right) +    \sigma_b^2        
     \Big]  \sprod{\Dot{x}_{s}}{ \Dot{y}_{t}}
\end{equation}
where 
\[ 
    \Sigma_\varphi^{x, y}(s,t) = 
     \left(\begin{array}{c}
       \K_{\varphi}^{x, x}(s,s), \K_{\varphi}^{x, y}(s, t)\\
       \K_{\varphi}^{x, y}(s, t), \K_{\varphi}^{y, y}(t, t)
     \end{array}\right) 
\]
and with initial conditions for any $s,t \in [0, 1]$
$$\K_{\varphi}^{x,y}(0,0)=\K_{\varphi}^{x,y}(s,0)=\K_{\varphi}^{x,y}(0,t)=\sigma_a^2$$
\end{theorem}

Once again, it is clearer to subdivide the proof of this result in two parts:
\begin{enumerate}
    \item The infinite width-convergence of the finite dimensional distributions 
    \[
    \Phi_{\varphi}^{M, N}(\X)  \xrightarrow[]{N \to \infty} \mathcal{N}(0, \K_{\D_M \times \D_M}(\X,\X))
    \]
    to those of a Gaussian process $\mathcal{GP}(0,\K_{\D_M \times \D_M})$ defined by a Kernel computed as the final value of a difference equation on the partition $\D_M \times \D_M$ of $[0,1]\times[0,1]$. This will be done using Tensor Programs \cite{yang2019wide} in \ref{subsubsec:discrete-phi-SK}.
    
    \item The infinite-depth convergence of the discrete kernels $\K_{\D_M \times \D_M}$ to a limiting Kernel $\K_\varphi$ which solves the differential equation \ref{eqn:hom_kernel_appendix}. This will be done in \ref{subsubsec:conv-phi-SK}.
\end{enumerate}

\subsubsection{Infinite-width limit with fixed-depth}\label{subsubsec:discrete-phi-SK}

We prove convergence in the finite-depth, infinite-width limit to a GP endowed with discrete kernels. using the formalism of \emph{Tensor Programs}.

\begin{theorem}\label{thm:discrete-phi-SigKer}
Let $\varphi: \R \to \R$ be linearly bounded
\footnote{in the sense that $\exists C > 0.$ such that $|\phi(x)| \leq C(1+|x|)$.}.
For any subset $\X = \{x_1, \dots, x_n\} \subset \bX$ the following convergence in distribution holds
\[
    \lim_{N \to \infty} \Phi_\varphi^{M, N}(\X) 
    =
    \mathcal{N}(0,\K_{\D_M \times \D_M}(\X, \X))
\]
where the map $\K_{\D \times \D'}: \bX \times \bX \to \R$ is, given any two partitions $\D, \D'$ of $[0,1]$, defined for any two paths $x,y \in \bX$ as $\K_{\D \times \D'}(x,y) = \K^{x,y}_{\D \times \D'}(1,1)$ where $\K^{x,y}_{\D \times \D'} : \D \times \D' \to \R$ solves the following difference equation 
\begin{align}\label{eqn:hom-discrete-kernel}
    \K_{\D \times \D'}(s_m,t_n) &= \K_{\D \times \D'}(s_{m-1}, t_n) + \K_{\D \times \D'}(s_m,t_{n-1}) - \K_{\D \times \D'}(s_{m-1},t_{n-1})\nonumber
    \\ 
    &+ \Big( \sigma_A^2 
    V_\varphi\left(\Sigma^{x,y}_{\D \times \D'}(s_{m-1}, t_{n-1})\right)
    + \sigma_b^2 \Big)  \sprod{\Delta x_{s_m}}{\Delta y_{t_n} }_{\R^d}
\end{align}
with initial conditions 
$$\K^{x,y}_{\D \times \D'}({0,0}) = \K^{x,y}_{\D \times \D'}(0,t_n) = \K^{x,y}_{\D \times \D'}(s_m,0) = \sigma_{a}^2$$
and where 
\[
    \Sigma^{x,y}_{\D \times \D'}(s,t) =  
    \begin{pmatrix}
        \K^{x,x}_{\D \times \D}(s,s) & \K^{x,y}_{\D \times \D'}(s,t) \\
         \K^{x,y}_{\D \times \D'}(s,t) & \K^{y,y}_{\D' \times \D'}(s,s)
    \end{pmatrix}
\]
\end{theorem}

\begin{proof}
    We use \cite{yang2019wide}[Corollary 5.5] applied to the Tensor Program of Algorithm \ref{Algo:TimeDep} where the input variables are independently sampled according to
    \[
    [A_{k}]_{\alpha}^{\beta} \sim \mathcal{N}(0,\frac{\sigma_{A}^2}{N}), 
    [v]_{\alpha}\sim \mathcal{N}(0,{1}),
    [S_0]_{\alpha}\sim \mathcal{N}(0,\sigma_{a}^2)
    \hspace{5pt} and \hspace{5pt} 
    [b_{k}]_{\alpha} \sim \mathcal{N}(0,\sigma_b^2)
\]
Note how the sampling scheme follows \cite{yang2019wide}[Assumption 5.1] and how linearly bounded functions are \emph{controlled} in the sense of \cite{yang2019wide}[Definition 5.3] since for all $x \in \R$ one has $|\phi(x)| \leq C(1+|x|) \leq e^{|x| + log(C)}$.
Thus we are under the needed assumptions to apply \cite{yang2019wide}[Corollary 5.5].

There are two things to notice, done in order to satisfy the required formalism: 
\begin{itemize}
    \item In the program the output projector $v$ is sampled according to $\mathcal{N}(0,1)$ while the original $\phi \sim \mathcal{N}(0,\frac{1}{N})$. This does not pose any problems since the output of the formal programs uses ${v}/{\sqrt{N}}\sim \mathcal{N}(0,\frac{1}{N})$.
    \item The input paths $x_i$ enter program \ref{Algo:TimeDep} not as \emph{Inputs} but as coefficients of \emph{LinComb}, this means that for any choice of input paths we must formally consider different algorithms. In any case, for any possible choice, the result has always the same functional form; hence \emph{a posteriori} it is legitimate to think about \emph{one} algorithm.
\end{itemize}

Note that $S^{x}_{m}$ stands for the formal variable in program \ref{Algo:TimeDep}, not for $S^{M, N}_{t^{\D_M}_m}(x)$ even thogh this is the value which it "stores" for a fixed hidden dimension $N$.

The result in \cite{yang2019wide}[Corollary 5.5] tells us that the output vector converges in law, as $N \to \infty$, to a Gaussian distribution $\mathcal{N}(0, \Tilde{K}^M)$ where
\[
    [\Tilde{K}^M]_i^j = \E_{Z \sim \mathcal{N}(\mu, \Sigma)}
    \bigg[
    Z^{S^{x_i}_{\norm{\mathcal{D}_M}}} Z^{S^{x_j}_{\norm{\mathcal{D}_M}}}
    \bigg]
    =
    \Sigma(S^{x_i}_{\norm{\mathcal{D}_M}}, S^{x_j}_{\norm{\mathcal{D}_M}})
\]

with $\mu, \Sigma$ computed according to \cite{yang2019wide}[Definition 5.2] and defined on the set of all G-vars in the program \emph{i.e.}

\begin{equation*}
    \mu(g) = \begin{cases}
      \mu^{in}(g) & \text{if $g$ is \textbf{Input} G-var}\\
      \sum_k a_k \mu(g_k) & \text{if $g$ is introduced as $\sum_k a_k g_k$ via \textbf{LinComb}}\\
      0 & \text{otherwise}
    \end{cases}
\end{equation*}
\begin{equation*}
    \Sigma(g, g') =  \begin{cases}
      \Sigma^{in}(g, g') & \text{if both $g$ and $g'$ are \textbf{Input} G-var}\\
      \sum_k a_k \Sigma(g_k, g') & \text{if $g$ is introduced as $\sum_k a_k g_k$ via \textbf{LinComb}}\\
      \sum_k a_k \Sigma(g, g_k') & \text{if $g'$ is introduced as $\sum_k a_k g'_k$ via \textbf{LinComb}}\\
      \sigma_W^2 \E_{Z \sim \mathcal{N}(\mu, \Sigma)}[\phi(Z)\phi'(Z)]& \text{if $g = Wh$, $g' = Wh'$ via \textbf{MatMul} w/ same $W$}\\
      0 & \text{otherwise}
    \end{cases}
\end{equation*}

where $h = \phi((g_k)_{k=1}^{m})$ for some function $\phi$ and $\phi(Z) := \phi((Z^{g_k})_{k=1}^{m})$, similarly for $g'$.

In our setting $\mu^{in} \equiv 0$ since all \textbf{Input} variables are independent, from which $\mu \equiv 0$; furthermore $\Sigma^{in}(g, g') = 0$ except if $g = g'$ when it takes values in $\{ \sigma^2_{a}, \sigma^2_{b}, \sigma^2_{A}\}$ accordingly.

Following the rules of $\Sigma$, assuming $m_i, m_j \in \{1, \dots, \norm{\mathcal{D}_M} \}$, we obtain
\begin{eqnarray*}
  \Sigma (S^{x_i}_{m_i}, S^{x_j}_{m_j}) & = & \Sigma (S^{x_i}_{m_i - 1}, S^{x_j}_{m_j}) +
  \Sigma \left( \sum_{k = 1}^d \gamma_{m_i, k}^i \Delta (x_i)_{t_{m_i}}^k,
  S^{x_j}_{m_j} \right)\\
  & = & \Sigma (S^{x_i}_{m_i - 1}, S^{x_j}_{m_j}) + \Sigma \left( \sum_{k = 1}^d
  \gamma_{m_i, k}^i \Delta (x_i)_{t_{m_i}}^k, S^{x_j}_{m_j-1} \right)\\
  &  & + \Sigma \left( \sum_{k = 1}^d \gamma_{m_i, k}^i \Delta
  (x_i)_{t_{m_i}}^k, \sum_{l = 1}^d \gamma_{m_j, l}^j \Delta
  (x_j)_{t_{m_j}}^l \right)\\
  & = & \Sigma (S^{x_i}_{m_i - 1}, S^{x_j}_{m_j}) + \Sigma (S^{x_i}_{m_i}, S^{x_j}_{m_j-1})
  - \Sigma (S^{x_i}_{m_i - 1}, S^{x_j}_{m_j-1})\\
  &  & + \sum_{k, l = 1}^d \Sigma (\gamma_{m_i, k}^i, \gamma_{m_j, l}^j)
  \Delta (x_i)_{t_{m_i}}^k \Delta (x_j)_{t_{m_j}}^l
\end{eqnarray*}

Now
\[ 
    \Sigma (\gamma_{m_i, k}^i, \gamma_{m_j, l}^j) 
    = \delta_{k, l} \sigma_{A}^2
    \E[\varphi(Z_1)\varphi(Z_2)]
    +
    \Sigma(b_k, b_l)
    = \delta_{k l} \big[ 
    \sigma_{A}^2 
    \E[\varphi(Z_1)\varphi(Z_2)]
    + \sigma_b^2
    \big]
\]
   
where $[Z_{1,} Z_2]^{\top} \sim \mathcal{N} (0, \tilde{\Sigma}_{m_i - 1, m_j - 1}(x_i, x_j))$  with

\begin{eqnarray*}
  \tilde{\Sigma}_{m, m'}(x_i, x_j) & = 
  & \left(\begin{array}{c c}
    \Sigma (S^{x_i}_{m}, S^{x_i }_{m}) & 
    \Sigma (S^{x_i}_{m}, S^{x_j }_{m'})\\
    \Sigma (S^{x_i}_{m}, S^{x_j }_{m'}) & 
    \Sigma (S^{x_j }_{m'}, S^{x_j }_{m'})
  \end{array}\right)
\end{eqnarray*}

thus if we set, for $t_{m_i}, t_{m_j} \in \D_M$,  $\K_{\D_M \times \D_M}^{x_i, x_j}(t_{m_i}, t_{m_j}) := \Sigma(S^{x_i}_{m_i}, S^{x_j}_{m_j})$ then we get
\begin{align*}
  \K_{\D_M \times \D_M}^{x_i, x_j}(t_{m_i}, t_{m_j}) &=  \K_{\D_M \times \D_M}^{x_i, x_j}({t_{m_i - 1}, t_{m_j}}) + \K_{\D_M \times \D_M}^{x_i, x_j}(({t_{m_i}, t_{m_j - 1}}) \\
  &- 
  \K_{\D_M \times \D_M}^{x_i, x_j}(({t_{m_i - 1}, t_{m_j - 1}})\\
  &+ \sum_{k = 1}^d \big( \sigma_A^2 V_\varphi\left(\tilde{\Sigma}_{m_{i-1}, m_{j-1}}(x_i, x_j)\right)  + \sigma_b^2 \big) \Delta (x_i)_{t_{m_i}}^k \Delta (x_j)_{t_{m_j}}^k
\end{align*}

which is exactly what \cref{eqn:hom-discrete-kernel} states. 
Then note how 
\[
\Sigma (S^{x_i}_{m_i}, S^{x_j}_{0}) = \Sigma (S^{x_i}_{m_i - 1}, S^{x_j}_{0}) = \dots = \Sigma (S^{x_i}_{0}, S^{x_j}_{0}) = \sigma_{a}^2
\]

Thus finally we can conclude and write the entries of the matrix $\Tilde{K}^M$ as
\[
    [\Tilde{K}^M]_i^j = \K_{\D_M \times \D_M}^{x_i, x_j}(1,1)
\]

\end{proof}

\begin{algorithm}[tbh]
    \caption{$S^{M, N}_1$ as Nestor program}
    \label{Algo:TimeDep}
    \begin{algorithmic}
       \STATE {\bfseries Input:} $S_0: \Gtype(N)$ \hfill  $\triangleright$ \textit{initial value}
       \STATE {\bfseries Input:} $b_1,\dots, b_d: \Gtype(N)$ \hfill $\triangleright$ \textit{biases}
       \STATE {\bfseries Input:} $A_1, \dots, A_d: \Atype(N, N)$ \hfill $\triangleright$ \textit{matrices}
       \STATE {\bfseries Input:} $v: \Gtype(N)$ \hfill $\triangleright$ \textit{readout layer weights}
      \\
      \FOR{$i = 1, \dots, n$}
          \STATE {\it // Compute $S^{\D_M, N}_1 (x_i)$  (here $ S_0^{x_i}$ is to be read as $S_0$)}
          \\
          \FOR{$m = 1,\dots,\norm{\mathcal{D}_M}$}
              \FOR{$k = 1,\dots,d$}
                \STATE $\alpha^i_{m,k} := \varphi(S^{x_i}_{m-1}): \Htype(N)$ \hfill  $\triangleright$ \textit{by Nonlin;}
                \STATE $\beta^i_{m,k} := A_k \alpha^i_{m,k} : \Gtype(N)$ \hfill  $\triangleright$ \textit{by Matmul;}
                \STATE $\gamma^i_{m,k} := \beta^i_{m,k} + b_k  : \Gtype(N)$ \hfill  $\triangleright$ \textit{by LinComb;}
              \ENDFOR
              \STATE $S^{x_i}_{m} := S^{x_i}_{m-1} + \sum_{k=1}^d \gamma^i_{m,k} [(x_i)^k_{t_m} - (x_i)^k_{t_{m-1}}]  : \Gtype(N)$ \hfill  $\triangleright$ \textit{by LinComb;}
          \ENDFOR
          \\
      \ENDFOR
      \\
      \ENSURE $(v^T S_{\norm{\mathcal{D}_M}}^{x_i} / \sqrt{N})_{i=1,\dots,n}$
    \end{algorithmic}
\end{algorithm}

Actually we have proved the even stronger statement that 
\begin{corollary}
    For all $t,s \in \D_M$ one has the following distributional limit
    \[
     \sprod{\phi^N}{S^{M, N}_{t}(\X)}\sprod{\phi^N}{S^{M, N}_{s}(\X)}
    \xrightarrow[N \to \infty]{} 
    \mathcal{N}(0,K_{\D_M \times \D_M}^{\X,\X}(t,s))
    \]
    where $[\phi^N]_{\alpha} \sim \mathcal{N}(0,\frac{1}{N})$ are independently sampled.
\end{corollary}

Moreover with the same proof but considering different partitions, $\D$ and $\D'$, for different paths $x,y \in \bX$ in the algorithm we see that:
\begin{proposition}\label{app:prop:Sigma_in_PSD}
    For all $t \in \D$, for all $s \in \D'$, with a clear abouse of notation, one has the following distributional limit
    \[
    \sprod{\phi^N}{S^{\D, N}_{t}(x)}\sprod{\phi^N}{S^{\D', N}_{s}(y)}
    \xrightarrow[N \to \infty]{} 
    \mathcal{N}(0,K_{\D \times \D'}^{x,y}({t,s}))
    \]
    where $[\phi^N]_{\alpha} \sim \mathcal{N}(0,\frac{1}{N})$ are independently sampled
    and the matrices $\Sigma_{\D \times \D'}^{x,y}(t,s)$ are always in $PSD_2$.    
\end{proposition}

\begin{remark}
    The fact that the matrices $\Sigma_{\D \times \D'}^{x,y}(t,s)$ are always in $PSD_2$ will be crucial in the following.
\end{remark}

\subsubsection{Uniform convergence to neural signature kernels }\label{subsubsec:conv-phi-SK}

We now prove that the sequence of discrete kernels convergence uniformly to our neural signature kernels. First we need  definitions for extending the discrete kernels to $[0,1]^2$ similarly to the previous section.

\begin{definition}
  Fix two partitions $\D \assign \{ 0 = s_0 < \cdots < s_M = 1 \}$
  and $\D' \assign \{ 0 = t_0 < \cdots < t_{M'} = 1
  \}$. For any $x,y \in \bX$ define
  \[ \K^{x,y}_{\D \times \D'} : \D \times
     \D' \rightarrow \mathbb{R} \]
  as the map satisfying the following recursion
  \begin{align*}
    & \K^{x,y}_{\D \times \D'} (s_m, t_n) \\
    = &
    \K^{x,y}_{\D \times \D'} (s_{m-1}, t_n) +
    \K^{x,y}_{\D \times \D'} (s_m, t_{n-1}) -
    \K^{x,y}_{\D \times \D'} (s_m, t_{n-1}) \\
    &+ \sum_{k = 1}^d \left( \sigma_A^2 V_\varphi\left(\Sigma^{x,y}_{\D \times \D'}(s_{m-1}, t_{n-1})\right) + \sigma_b^2 \right)
    \Delta x_{s_m}^k \Delta y_{t_n}^k  
    \\
    = &  \sigma^2_{a} + \sum_{\tmscript{\begin{array}{c}
      0 \leq k_1 < m\\
      0 \leq k_2 < n
    \end{array}}} \sum_{k = 1}^d \left( \sigma_A^2 V_{\varphi} \left(
    \Sigma^{x,y}_{\D \times \D'} (s_{k_1}, t_{k_2})
    \right) + \sigma_b^2 \right) \Delta x_{{s_{k_1 + 1}} }^k
    \Delta y_{t_{k_2 + 1}}^k 
  \end{align*}
  where
  \[ 
     \K^{x,y}_{\D \times \D'}(0,t_n) =  \K_{\D \times \D'}(s_m,0) =
     \sigma^2_{a} 
     \]
  and
  \[ \Sigma^{x,y}_{\D \times \D'} (s_m, t_n) =
     \left(\begin{array}{c}
       \K^{x,x}_{\D \times \D'} (s_m, s_m),  \K^{x,y}_{\D \times \D'} (s_m, t_n)\\
       \K^{x,y}_{\D \times \D'} (s_m, t_n), \K^{y,y}_{\D \times \D'} (t_n, t_n)
     \end{array}\right) \]
\end{definition}

{
\begin{definition}
  We extend the map $\K^{x,y}_{\D \times \D'} : \D \times \D' \rightarrow \mathbb{R}$ to the whole square $[0,1] \times [0,1]$ in two ways: for any $(s,t) \in [s_{m-1}, s_m) \times [t_{n - 1}, t_n)$ 
  \begin{enumerate}
      \item ({integral interpolation}) using a slight abuse of notation that overwrites the previous one, define the map $\K^{x,y}_{\D \times \D'} : [0, 1] \times [0,1] \rightarrow \mathbb{R}$
      as
      \begin{align*}
         \K^{x,y}_{\D \times \D'} (s,t) &=
         \K^{x,y}_{\D \times \D'} (s_{m-1}, t_n) +
        \K^{x,y}_{\D \times \D'} (s_m, t_{n-1}) -
        \K^{x,y}_{\D \times \D'} (s_m, t_{n-1}) 
        \\
        & + \int^s_{\eta = s_{m - 1}} \int^t_{\tau  = t_{n - 1}}
        \left( \sigma_A^2 V_{\varphi} \left( \Sigma^{x,y}_{\D \times \D'} (s_{m - 1}, t_{n -
        1}) \right) + \sigma_b^2 \right) \sprod{\Dot{x}_{\eta}}{ \Dot{y}_{\tau }}_{\R^d}   d \eta d\tau
      \end{align*}
      We extend in a similar way the matrix $\Sigma^{x,y}_{\D \times \D'}$.
      \item (piecewise constant interpolation) define the map  $\Tilde\K^{x,y}_{\D \times \D'} : [0, 1] \times [0,1] \rightarrow \mathbb{R}$ as
      \begin{eqnarray*}
    \tilde\K_{\D \times \D'}^{x,y} (s,t) & = &
    \left\{\begin{array}{l}
      \K^{x,y}_{\D \times \D'} (s_m, t_n) \quad
      \tmop{if} \quad (s,t) \in \left[ s_m {, s_{m + 1}}  \right) \times
      [t_n, t_{n + 1})\\
      \K^{x,y}_{\D \times \D'} (1, t_n) \quad
      \tmop{if} \quad (s,t) \in \{ 1 \} \times [t_n, t_{n + 1})\\
      \K^{x,y}_{\D \times \D'} (s_{m}, 1) \quad
      \tmop{if} \quad (s,t) \in \left[ s_m {, s_{m + 1}}  \right) \times \{ 1
      \}\\
      \K^{x,y}_{\D \times \D'} (1, 1) \quad \tmop{if}
      \quad (s,t) = (1, 1)
    \end{array}\right.
  \end{eqnarray*}
       and similarly for the matrix $\Tilde\Sigma_{\D \times \D'}^{x,y}$.
   \end{enumerate}
\end{definition}
}


  

Next we prove that the discrete kernels converge in distribution to a limiting kernel that satisfies a two-parameters differential equation. 


\begin{theorem} \label{simple_conv}
  Let $\{\D_M\}_{M \in \N}$ be a sequence of partitions $[0,1]$ with $| \D_M | \rightarrow 0$.
  Then we have the following convergence in $(C^0([0,1]\times [0,1]; \R), \norm{\cdot}_{\infty})$
  \[ \K_{\D_M \times \D_M}^{x, y} (s,t) \rightarrow \K^{x, y}_\varphi
     (s,t) \]
  where $\K^{x, y} (s,t)$ satisfies the following differential equation
  \begin{equation}
\partial_s \partial_t \K^{x,y}_{\varphi} = 
    \Big[        
    \sigma_A^2 V_\varphi 
    \left(\Sigma_\varphi^{x, y} \right) +    \sigma_b^2        
     \Big]  \sprod{\Dot{x}_{s}}{ \Dot{y}_{t}}
\end{equation}
where 
\[ 
    \Sigma_\varphi^{x, y}(s,t) = 
     \left(\begin{array}{c}
       \K_{\varphi}^{x, x}(s,s), \K_{\varphi}^{x, y}(s, t)\\
       \K_{\varphi}^{x, y}(s, t), \K_{\varphi}^{y, y}(t, t)
     \end{array}\right) 
\]
and with initial conditions for any $s,t \in [0, 1]$
$$\K_{\varphi}^{x,y}(0,0)=\K_{\varphi}^{x,y}(s,0)=\K_{\varphi}^{x,y}(0,t)=\sigma_a^2$$

    The limit is independent of the chosen sequence and the convergence is, in $x,y$, uniform on bounded sets of $\bX$.
\end{theorem}

The proof will rely on classic arguments used to prove convergence of Euler schemes to CDE solutions (see \cite{FrizVictoir} and \cite{EulerGoursat}).
However, our setting is more complex than the ones we found in the literature; this is due to the presence of two "independent" driving signals and on the unique nature of the driving fields, which not only depend on the "present" but also on "past" and "future". 
This complex dependency structure of the driving fields makes them completely non-local and requires more involved arguments for the solution to be made sense of. Note that due to their definition as limiting kernels the $\tilde{\Sigma}_{\D_M \times \D_M}^{x,y}$ are always PSD. However we are going to need a more quantitative result, in the form of a uniform membership in some $PSD(R)$, in order to later leverage Lemma \ref{lemma:split_on_PSD_R}. We first prove such a result for the case $x = y$ and $\D = \D'$, the general case will follow easily. The first "ingredient" is a uniform, in $x \in \bX$, bound from below of the kernels for $t=s$.

\begin{lemma}
  $\forall x \in \bX. \forall t \in \D. \quad \Tilde{\K}_{\D \times \D}^{x, x}
  (t, t) \geq \sigma_{a}^2 $\quad for any choice of $\D$.
\end{lemma}

\begin{proof}
  We have
  \[ \K^{x,x}_{\D \times \D} (t_m, t_m) = \sigma^2_{a} +
     \sum_{\tmscript{\begin{array}{c}
       0 \leq k_1 < m\\
       0 \leq k_2 < m
     \end{array}}} \sum_{k = 1}^d \left( \sigma_A^2 V_{\varphi} \left(
     \Sigma^{\tmscript{x, x}}_{\D \times \D} (t_{k_1}, t_{k_2})
     \right) + \sigma_b^2 \right) \Delta x_{{t_{k_1 + 1}} }^k
     \Delta x_{t_{k_2 + 1}}^k \]
  and by Tensor Programs we know that
\begin{equation*}
\resizebox{\hsize}{!}{$
\begin{gathered}
\sum_{\tmscript{\begin{array}{c}
       0 \leq k_1 < m\\
       0 \leq k_2 < m
     \end{array}}} \sum_{k = 1}^d \left( \sigma_A^2 V_{\varphi} \left(
     \Sigma^{\tmscript{x, x}}_{\D \times \D} (t_{k_1}, t_{k_2})
     \right) + \sigma_b^2 \right) \Delta x_{{t_{k_1 + 1}} }^k
     \Delta x_{t_{k_2 + 1}}^k = \lim_{N \rightarrow \infty}
     \mathbb{E} \left[ \frac{\langle S_{t_m}^{\D, N} (x) - S^N_0, S_{t_m}^{\D, N} (x) - S^N_0 \rangle}{N} \right] \geq 0 
\end{gathered}
$}
\end{equation*}
  
\end{proof}

The second step is another bound, this time from above, which will be uniform on bounded subsets of $\bX$.

\begin{lemma}
  There is a constant $C_x > 0$ independent of $\D$ such that
  \[ \| \K_{\D \times \D}^{x \comma x} \|_{\infty,
     [0, 1] \times [0, 1]} \assign \sup\limits_{(s,t) \in [0, 1]^2} |
     \K_{\D \times \D}^{x \comma x} (s,t) | \leq C_x
  \]
  Moreover $C_x$ only depends on $\norm{x}_{1-var}$ and is increasing in it, thus there is a uniform bound on bounded subsets of $\bX$.
\end{lemma}

\begin{proof}
    Let us prove this result for $\Tilde{\K}^{x,x}$ first.
  Wlog $(s,t) \in \left[ s_m , s_{m + 1} \right) \times [t_n, t_{n + 1})$. Remember how
  \begin{eqnarray*}
    \Tilde{\K}_{\D \times \D}^{x \comma x} (s,t) & = &
    \K_{\D \times \D}^{x \comma x} (s_m, t_n)\\
    & = & \sigma^2_{a} +  \int^{s_m}_{\eta = 0} \int^{t_n}_{\tau  =
    0} \left( \sigma_A^2 V_{\varphi} \left(
    \tilde{\Sigma}^{\tmscript{x, x}}_{\D \times \D} (\eta, \tau) \right)
    + \sigma_b^2 \right)
    \sprod{\Dot{x}_{\eta}}{\Dot{x}_{\tau }}_{\R^d}   
    d\eta d\tau
  \end{eqnarray*}
  thus
  \begin{eqnarray*}
    | \tilde{\K}_{s, t} | & \leq & 
    \sigma^2_{a} + 
    \int^{s_m}_{\eta = 0} \int^{t_n}_{\tau = 0} 
    | \sigma_A^2 V_{\varphi} (\tilde{\Sigma}_{ \eta, \tau})| 
    |\sprod{\Dot{x}_{\eta}}{\Dot{x}_{\tau }}_{\R^d}|  
    d\eta d\tau
    + \int^{s_m}_{\eta = 0} \int^{t_n}_{\tau = 0} 
    | \sigma_b^2 | 
    |\sprod{\Dot{x}_{\eta}}{\Dot{x}_{\tau }}_{\R^d}|  
    d\eta d\tau \\
    & = & \sigma^2_{a} + 
    \sigma_b^2 \norm{x}_{1-var, [0,s_m]}\norm{x}_{1-var, [0,s_n]} +
    \sigma_A^2 \int^{s_m}_{\eta = 0} \int^{t_n}_{\tau = 0} 
    |V_{\varphi}(\tilde{\Sigma}_{ \eta, \tau})|  
    |\sprod{\Dot{x}_{\eta}}{\Dot{x}_{\tau }}_{\R^d}|  
    d\eta d\tau \\
    & \leq & \sigma^2_{a} + 
    \sigma_b^2 \| x \|^2_{1 - \tmop{var}, [0, 1]} +
    \sigma_A^2 \int^{s_m}_{\eta = 0} \int^{t_n}_{\tau = 0} 
    \tilde{M} 
    \tmscript{\left(
    1 + \sqrt{\Tilde{\K}_{\eta, \eta}} \right) \left( 1 +
    \sqrt{\Tilde{\K}_{\tau , \tau}} \right)} 
    |\sprod{\Dot{x}_{\eta}}{\Dot{x}_{\tau }}_{\R^d}|  
    d\eta d\tau\\
    & = & \sigma^2_{a} + \sigma_b^2 \| x \|^2_{1 - \tmop{var}, [0, 1]} +
    \sigma_A^2 \tilde{M} 
    \tmscript{\left( \int^{s_m}_{\eta = 0} \left( 1 +
    \sqrt{\Tilde{\K}_{\eta, \eta}} \right) |\Dot{x}_{\eta}| d\eta \right) 
    \left(\int^{t_n}_{\tau = 0} \left( 1 + \sqrt{\Tilde{\K}_{\tau , \tau}}
    \right)  |\Dot{x}_{\tau }| d\tau \right)}
  \end{eqnarray*}
  In particular
  \begin{eqnarray*}
    \Tilde{\K}_{s, s} = | \Tilde{\K}_{s, s} | & \leq & \sigma^2_{a} +
    \sigma_b^2 \| x \|^2_{1 - \tmop{var}, [0, 1]} + \sigma_A^2 \tilde{M}
    \left| \int^{s_m}_{\eta = 0} \left( 1 +
    \sqrt{\Tilde{\K}_{\eta, \eta}} \right) |\Dot{x}_{\eta}| d\eta
    \right|^2\\
    & \leq & \sigma^2_{a} + \sigma_b^2 \| x \|^2_{1 - \tmop{var}, [0, 1]} +
    \sigma_A^2 \tilde{M}  \| x \|_{1 - \tmop{var}, [0, t_m]}
    \tmscript{
    \left( \int^{s_m}_{\eta = 0} \left( 1 + \sqrt{\Tilde{\K}_{\eta, \eta}}
    \right)^2 |\Dot{x}_{\eta}| d\eta \right)}\\
    & \leq & \sigma^2_{a} + \sigma_b^2 \| x \|^2_{1 - \tmop{var}, [0, 1]} +
    2 \sigma_A^2 \tilde{M} \| x \|_{1 - \tmop{var}, [0, 1]} 
    \int^{s_m}_{\eta = 0} (1 + \Tilde{\K}_{\eta, \eta}) 
    |\Dot{x}_{\eta}| d\eta \\
    & \leq & \sigma^2_{a} + \sigma_b^2 \| x \|^2_{1 - \tmop{var}, [0, 1]} +
    2 \sigma_A^2 \tilde{M} \| x \|_{1 - \tmop{var}, [0, 1]}
    \int^s_{\eta = 0} (1 + \Tilde{\K}_{\eta, \eta}) 
    |\Dot{x}_{\eta}| d\eta
  \end{eqnarray*}

  and we can use Gronwall Inequality (\cite{FrizVictoir}, Lemma 3.2) to
  obtain
  \begin{eqnarray*}
    1 + \Tilde{\K}_{s, s} & \leq & (1 + \sigma^2_{a} + \sigma_b^2 \| x
    \|^2_{1 - \tmop{var}, [0, 1]}) \exp \{ 2 \sigma_A^2 \tilde{M} \| x \|^2_{1
    - \tmop{var}, [0, 1]} \}
  \end{eqnarray*}
  hence the thesis with
  \begin{eqnarray*}
    \Tilde{C}_x & = & (1 + \sigma^2_{a} + \sigma_b^2 \| x \|^2_{1 - \tmop{var}, [0,
    1]}) {e^{2 \sigma_A^2 \tilde{M} \| x \|^2_{1 - \tmop{var}, [0, 1]}}}  - 1
  \end{eqnarray*}

  In order to extend the bound to $\K^{x,x}$ notice that
  \begin{align*}
          & |\K_{\D \times \D}^{x,x}(s,t)| =  |\sigma^2_{a} +  \int^{s}_{\eta = 0} \int^{t}_{\tau  =
          0} \left( \sigma_A^2 V_{\varphi} \left(
          \tilde{\Sigma}^{\tmscript{x, x}}_{\D \times \D} (\eta, \tau) \right)
          + \sigma_b^2 \right)
          \sprod{\Dot{x}_{\eta}}{\Dot{x}_{\tau }}_{\R^d}   
          d\eta d\tau |
          \\
          \leq &
          \sigma^2_{a} +   \left( 2\sigma_A^2 \Tilde{M} (1 + \Tilde{C}_x) + \sigma_b^2 \right)
          \int^s_{\eta = 0} \int^t_{\tau = 0} |\sprod{\Dot{x}_{\eta}}{\Dot{x}_{\tau }}_{\R^d}| d\eta d\tau 
          \\
          \leq &
          \sigma^2_{a} +   \left( 2\sigma_A^2 \Tilde{M} (1 + \Tilde{C}_x) + \sigma_b^2 \right)
          \norm{x}_{1-var,[0,1]}^2 =: C_x 
      \end{align*}
\end{proof}

Thus finally we can conclude, as promised, that

\begin{lemma}
  There is an $R = R_x \in \mathbb{R}$ such that $\tilde{\Sigma}_{\D \times\D}^{x \comma x} \in \tmop{PSD_2} (R)$ for every $\D$.
  Moreover, as before, this $R_x$ only depends on $\norm{x}_{1-var}$ and is increasing in it. 
\end{lemma}

\begin{proof}
  The diagonal elements $\Tilde{\K}_{\D \times \D}^{x,x}(s , s)$ are
  bounded above by a common constant $C_x$ and below by $\sigma^2_{a}$
  We thus can simply choose $R_x = C_x \vee \sigma_{a}^{-2}$.
\end{proof}

We are ready to extend the result to the general case. 

\begin{proposition}\label{prop:bounded_set_bounds}
    Fix a $\alpha > 0$. There exist a constant $C_{\alpha}$ such that for all $x,y \in \bX$ with $\norm{x}_{1-var}, \norm{y}_{1-var} \leq \alpha$ and partitions $\D, \D'$ it holds
    \[
        \sup_{0 \leq t, s \leq 1} \nobracket | \nobracket \K_{\D \times \D'}^{x \comma y} (s,t) | \leq C_{\alpha} 
    \]
    Moreover for $R_{\alpha} := C_{\alpha} \vee \sigma_{a}^{-2}$ we get 
    \[
        \Tilde{\Sigma}_{\D \times \D'}^{x \comma y} (s,t) \in \tmop{PSD}(R_{\alpha}) 
    \]
\end{proposition}

\begin{proof}
Remember how the maps $\Tilde{\Sigma}_{\D \times \D'}^{x \comma y}$
are $PSD_2$ since they arise as such from Tensor limiting arguments, in particular by positive semidefinitiveness 
and the previous bounds
\begin{eqnarray*}
\nobracket | \nobracket \tilde{\K}_{\D \times \D'}^{x
\comma y} (s,t) | \leq \sqrt{\Tilde{\K}_{\D \times
\D}^{x \comma x} (s,s) \Tilde{\K}_{\D' \times
\D'}^{y \comma y} (t,t)} \leq \sqrt{\Tilde{C}_x \Tilde{C}_y} \leq \Tilde{C}_x \vee \Tilde{C}_y 
&  & 
\end{eqnarray*}

For $K_{x,y}(t)$ we proceed similarly to before:
\begin{align*}
          & |\K_{\D \times \D'}^{x,y}(s,t)| =  |\sigma^2_{a} +  \int^s_{\eta = 0} \int^{t}_{\tau  =
          0} \left( \sigma_A^2 V_{\varphi} \left(
          \tilde{\Sigma}^{\tmscript{x, y}}_{\D \times \D'} (\eta, \tau) \right)
          + \sigma_b^2 \right)
          \sprod{\Dot{x}_{\eta}}{\Dot{y}_{\tau }}_{\R^d}   
          d\eta d\tau |
          \\
          \leq &
          \sigma^2_{a} +   \left( 2\sigma_A^2 \Tilde{M} (1 + \Tilde{C}_x \vee \Tilde{C}_y ) + \sigma_b^2 \right)
          \int^s_{\eta = 0} \int^t_{\tau = 0} |\sprod{\Dot{x}_{\eta}}{\Dot{y}_{\tau }}_{\R^d}| d\eta d\tau 
          \\
          \leq & 
          \sigma^2_{a} +   \left( 2\sigma_A^2 \Tilde{M} (1 + \Tilde{C}_x \vee \Tilde{C}_y ) + \sigma_b^2 \right)
          \norm{x}_{1-var,[0,1]}\norm{y}_{1-var,[0,1]} \leq C_{\alpha} \\
\end{align*}

The second part follows from the definition of $\tmop{PSD}(R_{\alpha})$.
\end{proof}


\begin{proof}[Proof of \cref{simple_conv}]~\\
    
  The central idea will be proving that, given any sequence $\D_n \times \D'_n$ with $|\D_n| \vee |\D'_n| \downarrow 0$, the maps $\{ \K_{\D_M \times \D'_n}^{x, y} \}_{n}$ form a Cauchy sequence in the Banach space $(C^0 ([0, 1] \times [0, 1] ; \mathbb{R}), \| \cdot \|_{\infty})$. By completeness of this space the sequence will admit a limit $\K^{x, y}$, which we will prove to be independent from the chosen sequence. Finally we will prove that this limit is the sought after kernel.

  Note that trivially the set $\{x,y\}$ is bounded in $\bX$, we can thus fix two constants $C_{x,y}, R_{x,y}$ with the properties given in Proposition \ref{prop:bounded_set_bounds}. We will often, for ease of notation, refer to them as $C$ and $R$.

  \paragraph{Part I: how close are $\K_{\D \times \D'}$ and $\tilde{\K}_{\D \times \D'}$ ?}~\\
  
  We have
  \begin{align*}
    & | \K^{x,y}_{\D \times \D'} (s,t) -
    \tilde{\K}_{\D \times \D'}^{x, y} (s,t) | 
    \\
     = & |
    \nobracket \int^s_{\eta = 0} \int^t_{\tau = 0} (\sigma_A^2 V_{\varphi}
    (\widetilde{\Sigma }_{\D \times \D'}^{x, y} (\eta,\tau)) 
    + \sigma_b^2) \sprod{\Dot{x}_{\eta}}{\Dot{y}_{\tau }}_{\R^d}        d\eta d\tau
    \\
      & - \int^{s_m}_{\eta = 0} \int^{t_n}_{\tau = 0} (\sigma_A^2 V_{\varphi}
    (\widetilde{\Sigma }_{\D \times \D'}^{ x, y} (\eta,\tau)) 
    + \sigma_b^2) \sprod{\Dot{x}_{\eta}}{\Dot{y}_{\tau }}_{\R^d}        d\eta d\tau | \nobracket
    \\
    = & \left| \int_{\Omega_{s, t}} (\sigma_A^2 V_{\varphi} (\widetilde{\Sigma
    }_{\D \times \D'}^{ x, y} (\eta, \tau)) + \sigma_b^2)
    \sprod{\Dot{x}_{\eta}}{\Dot{y}_{\tau }}_{\R^d}        d\eta d\tau \right|
  \end{align*}
  with
  \[ \Omega_{s, t} \assign [0, s] \times [0, t] \setminus [0, s_m] \times [0,
     t_n] \]
  Note that
  \[ \Omega_{s, t} \subseteq ([s_{m,} s_{m + 1}] \times [0, 1]) \cup ([0, 1]
     \times [t_n, t_{n + 1}]) \]
  thus
  \[ \mathcal{L}^2 (\Omega_{s, t}) \leq | \D | + | \D' |
     \leq 2 (| \D | \vee | \D' |) \]
  hence
  \begin{align*}
    & | \K^{x,y}_{\D \times \D'} (s,t) -
    \tilde{\K}_{\D \times \D'}^{x, y} (s,t) |  \leq 
    \int_{\Omega_{s, t}} | \sigma_A^2 V_{\varphi} (\widetilde{\Sigma }_{\D
    \times \D'}^{ x, y} (\eta, \tau)) + \sigma_b^2 |  |\sprod{\Dot{x}_{\eta}}{\Dot{x}_{\tau }}_{\R^d}|       d\eta d\tau
    \\
     \leq  &
     \sigma_A^2 \int_{\Omega_{s, t}} \tilde{M} 
    \tmscript{\left( 1 +
    \sqrt{\K^{x,x}_{\D \times \D} (\eta \comma \eta)}
    \right) \left( 1 + \sqrt{\K^{y,y}_{\D' \times \D'}
    (\tau \comma \tau)} \right)} 
    |\sprod{\Dot{x}_{\eta}}{\Dot{y}_{\tau }}_{\R^d}|       d\eta d\tau
    \\
    & +  
    \sigma_b^2 \int_{\Omega_{s, t}} 
    |\sprod{\Dot{x}_{\eta}}{\Dot{y}_{\tau }}_{\R^d}|       
    d\eta d\tau
    \\
     \leq  &
     \sigma_A^2 \tilde{M} \left( 1 + \sqrt{C_{x, y}} \right)^2
    \int_{\Omega_{s, t}} |\sprod{\Dot{x}_{\eta}}{\Dot{y}_{\tau }}_{\R^d}|       
    d\eta d\tau 
    + \sigma_b^2
    \int_{\Omega_{s, t}} |\sprod{\Dot{x}_{\eta}}{\Dot{y}_{\tau }}_{\R^d}|      
    d\eta d\tau
    \\
     \leq &
     (2 \sigma_A^2 \tilde{M} (1 + C_{x, y}) + \sigma_b^2) \int_{[0, 1]
    \times [0, 1]} \mathbb{I}_{\Omega_{s, t}} (\eta, \tau) 
    |\sprod{\Dot{x}_{\eta}}{\Dot{y}_{\tau }}_{\R^d}|       
    d\eta d\tau
    \\
     \leq  &
     (2 \sigma_A^2 \tilde{M} (1 + C_{x, y}) + \sigma_b^2)
    \sqrt{\mathcal{L}^2 (\Omega_{s, t})}
    \sqrt{\int_{[0, 1] \times [0, 1]} 
    |\sprod{\Dot{x}_{\eta}}{\Dot{y}_{\tau }}_{\R^d}|^2       
    d\eta d\tau}
    \\
    \leq &
    (2 \sigma_A^2 \tilde{M} (1 + C_{x, y}) + \sigma_b^2)
    \sqrt{\mathcal{L}^2 (\Omega_{s, t})}
    \sqrt{\int_{[0, 1] \times [0, 1]} 
    |\Dot{x}_{\eta}|^2 |\Dot{y}_{\tau }|^2       
    d\eta d\tau}
    \\
     \leq &
     (2 \sigma_A^2 \tilde{M} (1 + C_{x, y}) + \sigma_b^2)
    \sqrt{\mathcal{L}^2 (\Omega_{s, t})}
    \sqrt{\norm{x}^2_{\bX}\norm{y}^2_{\bX}}
    \\
     \leq &
     2 (2 \sigma_A^2 \tilde{M} (1 + C_{x, y}) + \sigma_b^2) 
     \sqrt{|\D | \vee | \D' |} 
     \sqrt{\norm{x}^2_{\bX}\norm{y}^2_{\bX}}
  \end{align*}
  
  In particular as $(| \D | \vee | \D' |) \rightarrow 0$ we
  have
    \begin{equation*} \begin{gathered}
    \sup_{[0, 1] \times [0, 1]} | \K^{x,y}_{\D \times \D'}
     (s,t) - \tilde{\K}_{\D \times \D'}^{x, y} (s,t) |  \leq
    \\
     2 (2 \sigma_A^2 \tilde{M} (1 + C_{x, y}) + \sigma_b^2)   
     \sqrt{|\D | \vee | \D' |} 
     \sqrt{\norm{x}^2_{\bX}\norm{y}^2_{\bX}} \rightarrow 0
     \end{gathered} \end{equation*}

  Notice that, since we can repeat the argument for $\K_{\D \times
  \D}^{x, x}$ and $\K^{y,y}_{\D' \times \D'}$, we have
  \begin{equation*}  \begin{gathered}
  \sup_{[0, 1] \times [0, 1]} | \Sigma_{\D \times \D'}^{ 
     x, y} (s,t) - \tilde{\Sigma}_{\D \times \D'}^{  x, y}
     (s,t) |_{\infty} \leq 
     \\
     2 (2 \sigma_A^2 \tilde{M} (1 + C_{x, y}) +
     \sigma_b^2) 
     \sqrt{|\D | \vee | \D' |} 
     (\norm{x}^2_{\bX} \vee \norm{y}^2_{\bX}) =: \Gamma_{x,y}\sqrt{|\D | \vee | \D' |}
  \end{gathered} \end{equation*}

  \paragraph{Part II : Cauchy Bounds}
  ~\\
  Consider now another partition $\check{\D} \times
  \check{\D}'$, we have
  \begin{align*}
    & | \K_{\check{\D} \times \check{\D}'}^{x, y} (s,t) -
    \K^{x,y}_{\D \times \D'} (s,t) |  
    \\
    = & \left|
    \sigma_A^2 \int^s_{\eta = 0} \int^t_{\tau = 0} [V_{\varphi}
    (\tilde{\Sigma}_{\check{\D} \times \check{\D}'}^{  x, y}
    (\eta, \tau)) - V_{\varphi} (\tilde{\Sigma}_{\D \times \D'}^{
    x, y} (\eta, \tau))] \sprod{\Dot{x}_{\eta}}{\Dot{y}_{\tau }}_{\R^d}        d\eta d\tau 
    \right|
    \\
     \leq  &
     \sigma_A^2 \int^s_{\eta = 0} \int^t_{\tau = 0} | V_{\varphi}
    (\tilde{\Sigma}_{\check{\D} \times \check{\D}'}^{  x, y}
    (\eta, \tau)) - V_{\varphi} (\tilde{\Sigma}_{\D \times \D'}^{x, y} (\eta, \tau)) |  
    |\sprod{\Dot{x}_{\eta}}{\Dot{y}_{\tau }}_{\R^d}|       d\eta d\tau
    \\
     \leq  &
     \sigma_A^2 \int^s_{\eta = 0} \int^t_{\tau = 0} k_R |
    \tilde{\Sigma}_{\check{\D} \times \check{\D}'}^{  x, y}
    (\eta, \tau) - \tilde{\Sigma}_{\D \times \D'}^{  x, y}
    (\eta, \tau) |_{\infty}  |\sprod{\Dot{x}_{\eta}}{\Dot{y}_{\tau }}_{\R^d}|       d\eta d\tau
    \\
     \leq  &
     \sigma_A^2 k_R \int^s_{\eta = 0} \int^t_{\tau = 0} |
    \tilde{\Sigma}_{\check{\D} \times \check{\D}'}({  \eta,
    \tau}) - \Sigma_{\check{\D} \times \check{\D}'}({ 
    \eta, \tau}) |_{\infty}  |\sprod{\Dot{x}_{\eta}}{\Dot{y}_{\tau }}_{\R^d}|       d\eta d\tau
    \\
     &  + \sigma_A^2 k_R \int^s_{\eta = 0} \int^t_{\tau = 0} |
    \Sigma_{\check{\D} \times \check{\D}'}({  \eta, \tau})
    - \Sigma_{\D \times \D'}({  \eta, \tau}) |_{\infty} 
    |\sprod{\Dot{x}_{\eta}}{\Dot{y}_{\tau }}_{\R^d}|       d\eta d\tau
    \\
     &  + \sigma_A^2 k_R \int^s_{\eta = 0} \int^t_{\tau = 0} |
    \Sigma_{\D \times \D'}({  \eta, \tau}) -
    \tilde{\Sigma}_{\D \times \D'}({  \eta, \tau})
    |_{\infty} |\sprod{\Dot{x}_{\eta}}{\Dot{y}_{\tau }}_{\R^d}|       d\eta d\tau
    \\
     \leq  &
     \sigma_A^2 \Gamma_{x,y}\sqrt{|\check{\D} | \vee |\check{\D}'|}
    \int^s_{\eta = 0} \int^t_{\tau = 0}  
    |\sprod{\Dot{x}_{\eta}}{\Dot{y}_{\tau }}_{\R^d}|       d\eta d\tau
    \\
     &  + \sigma_A^2 k_R \int^s_{\eta = 0} \int^t_{\tau = 0} |
    \Sigma_{\check{\D} \times \check{\D}'}({  \eta, \tau})
    - \Sigma_{\D \times \D'}({  \eta, \tau}) |_{\infty} 
    |\sprod{\Dot{x}_{\eta}}{\Dot{y}_{\tau }}_{\R^d}|       d\eta d\tau
    \\
     &  + \sigma_A^2 \Gamma_{x,y} \sqrt{|\D | \vee | \D' |}
     \int^s_{\eta = 0} \int^t_{\tau = 0} 
     |\sprod{\Dot{x}_{\eta}}{\Dot{y}_{\tau }}_{\R^d}|       d\eta d\tau
     \\
     \leq & \sigma_A^2 2\Gamma_{x,y}\sqrt{|\D | \vee | \D' | \vee |\check{\D} | \vee |\check{\D}'|}
     (\norm{x}^2_{1 - \tmop{var}, [0, 1]} \vee \norm{y}^2_{1 - \tmop{var}, [0, 1]})
    \\
     &  + \tmcolor{black}{\sigma_A^2 k_R \int^s_{\eta = 0} \int^t_{\tau  =
    0} | \Sigma_{\check{\D} \times \check{\D}'}^{  x, y}
    (\eta, \tau) - \Sigma_{\D \times \D'}^{  x, y} (\eta,
    \tau) |_{\infty} |\sprod{\Dot{x}_{\eta}}{\Dot{y}_{\tau }}_{\R^d}|       d\eta d\tau}
  \end{align*}

  \vspace{3pt}
  Assume to be in the case $x=y$, $\D = \D'$, $\check{\D} = \check{\D}'$ and define the following quantity:
  \[
  \Xi_t \assign 
  \sup_{0 \leq \eta, \tau \leq t} | \Sigma_{\check{\D} \times \check{\D}}^{x,x}({\eta, \tau}) - \Sigma_{\D \times \D}^{x,x}({\eta, \tau}) |_{\infty} = 
  \sup_{0 \leq \eta, \tau \leq t} |\K_{\check{\D} \times \check{\D}}^{x,x}({\eta, \tau}) - \K_{\D \times \D}^{x,x}({\eta, \tau})|
  \]

  One has, from the previous inequality, that
  \begin{eqnarray*}
    \Xi_t &
    \leq &  \sigma_A^2 2 \Gamma_{x,x}
    \sqrt{|\check{\D} | \vee |\D |}
    \| x \|^2_{1 - \tmop{var}, [0, 1]} \\
    &  & + \tmcolor{black}{\sigma_A^2 k_R \int^s_{\eta = 0} \int^t_{\tau  =
    0} | \Sigma_{\check{\D} \times \check{\D}}({  \eta, \tau}) - \Sigma_{\D \times \D}({  \eta, \tau})
    |_{\infty} |\sprod{\Dot{x}_{\eta}}{\Dot{x}_{\tau }}_{\R^d}|       d\eta d\tau}
  \end{eqnarray*}
  and since
  \begin{align*}
    & \int^s_{\eta = 0} \int^t_{\tau = 0} | \Sigma_{\check{\D}
    \times \check{\D}}({  \eta, \tau}) - \Sigma_{\D \times
    \D}({  \eta, \tau}) |_{\infty} |\sprod{\Dot{x}_{\eta}}{\Dot{x}_{\tau }}_{\R^d}|d\eta d\tau
    \\
    \leq & 
    \int^s_{\eta = 0} \int^t_{\tau = 0} \Xi_{\eta \vee
    \tau} |\sprod{\Dot{x}_{\eta}}{\Dot{x}_{\tau }}_{\R^d}|d\eta d\tau  
    \\
    = &
    \int^s_{\eta = 0} \int^{\eta}_{\tau = 0} \Xi_{\eta} 
    |\sprod{\Dot{x}_{\eta}}{\Dot{x}_{\tau }}_{\R^d}|d\eta d\tau
    + \int^s_{\eta = 0} \int^t_{\tau  = \eta}
    \Xi_{\tau } |\sprod{\Dot{x}_{\eta}}{\Dot{x}_{\tau }}_{\R^d}| d\eta d\tau 
    \\
     \leq & 
     \int^s_{\eta = 0}  \| x \|_{1 - \tmop{var},
    [0, \eta]} \Xi_{\eta} |\Dot{x}_{\eta} | d\eta 
    +  \int^t_{\tau = 0} \int^{\sigma}_{\eta = 0} 
    \Xi_{\tau } 
    |\Dot{x}_{\eta} | |\Dot{x}_{\tau } | d\eta d\tau  
    \\
     = &
     \int^s_{\eta = 0}  \| x \|_{1 - \tmop{var},
    [0, \eta]} \Xi_{\eta} |\Dot{x}_{\eta} | d\eta 
    + \int^t_{\tau = 0}  \| x \|_{1 - \tmop{var},
    [0, \tau]} \Xi_{\tau } |\Dot{x}_{\tau } | d\tau   
    \\
    \leq & 
    2 \| x \|_{1 - \tmop{var}, [0, 1]} 
     \int^s_{\eta = 0} \Xi_{\eta} |\Dot{x}_{\eta} | d\eta & 
  \end{align*}
  we get
  \begin{eqnarray*}
    {\Xi_t}  & \leq &  \sigma_A^2 2\Gamma_{x,x}
    \| x \|^2_{1 - \tmop{var}, [0, 1]} 
    \sqrt{|\check{\D} | \vee |\D |}\\
    &  & + 2 \sigma_A^2 k_R \| x \|_{1 - \tmop{var}, [0, 1]} 
    \int^s_{\eta = 0} \Xi_{\eta} |\Dot{x}_{\eta} | d\eta 
  \end{eqnarray*}
  thus, by Gronwall,
  \[ \Xi_t \leq 2 \sigma_A^2 \Gamma_{x,x}
     \| x \|^2_{1 - \tmop{var}, [0, 1]} 
    \sqrt{|\check{\D} | \vee |\D |}
      \cdot e^{2 \sigma_A^2 k_R \| x \|_{1 - \tmop{var}, [0, 1]}^2 }
      \]

  \vspace{5pt}
  Coming back to the general case  and setting 
  \[
  \Lambda_{x,y} :=\norm{x}^2_{1 - \tmop{var}, [0, 1]} \vee \norm{y}^2_{1 - \tmop{var}, [0, 1]}
  \]
  we can now say that
  
  \begin{equation*} \begin{gathered}
    \Xi^{x, y}_t \assign \sup_{0 \leq \eta, \tau \leq t} |
    \Sigma_{\check{\D} \times \check{\D}'}({  \eta, \tau})
    - \Sigma_{\D \times \D'}({  \eta, \tau}) |_{\infty} 
    \\
    \leq  2 \sigma_A^2 \Gamma_{x,y}
    \sqrt{|\check{\D} | \vee | \check{\D}'| \vee | \D | \vee | \D' |} 
    \Lambda_{x,y}
    (1 +  e^{2 \sigma_A^2 k_R  \Lambda_{x,y}})
    \\
       + \tmcolor{black}{\sigma_A^2 k_R \int^s_{\eta = 0} \int^t_{\tau  =
    0} | \Sigma_{\check{\D} \times \check{\D}'}({  \eta,
    \tau}) - \Sigma_{\D \times \D'}({  \eta, \tau})|_{\infty} 
    |\sprod{\Dot{x}_{\eta}}{\Dot{y}_{\tau }}_{\R^d}|       d\eta d\tau}
  \end{gathered} \end{equation*}
  and since
  \begin{align*}
     & \int^s_{\eta = 0} \int^t_{\tau = 0} 
    | \Sigma_{\check{\D} \times \check{\D}'}({  \eta, \tau}) 
    - \Sigma_{\D \times \D'}({  \eta, \tau}) |_{\infty} 
    |\sprod{\Dot{x}_{\eta}}{\Dot{y}_{\tau }}_{\R^d}|       d\eta d\tau 
    \\
    \leq &
    \int^s_{\eta = 0} \int^t_{\tau = 0} {\Xi^{x, y}_{\eta
    \vee \tau}}  |\sprod{\Dot{x}_{\eta}}{\Dot{y}_{\tau }}_{\R^d}|       d\eta d\tau 
    \\
    & = \int^s_{\eta = 0} \int^{\eta}_{\tau = 0} 
    \Xi^{x, y}_{\eta} 
    |\sprod{\Dot{x}_{\eta}}{\Dot{y}_{\tau }}_{\R^d}|       d\eta d\tau 
    + \int^s_{\eta = 0} \int^t_{\tau  =\eta} 
    \Xi^{x, y}_{\tau }
    |\sprod{\Dot{x}_{\eta}}{\Dot{y}_{\tau }}_{\R^d}|       d\eta d\tau 
    \\
     \leq  &
    \int^s_{\eta = 0} \| y \|_{1 - \tmop{var}, [0, \eta]} 
    \Xi^{x, y}_{\eta} |\Dot{x}_{\eta}| d\eta 
    +  \int^t_{\tau = 0} \int^{\sigma}_{\eta = 0} 
    \Xi^{x, y}_{\tau } 
    |\Dot{x}_{\eta}| |\Dot{y}_{\tau }| d\eta  d\tau   
    \\
     = & 
     \int^s_{\eta = 0} \| y \|_{1 - \tmop{var}, [0, \eta]} 
    \Xi^{x, y}_{\eta} |\Dot{x}_{\eta}| d\eta 
    +  \int^t_{\tau = 0} \| x \|_{1 - \tmop{var}, [0, \tau]} 
    \Xi^{x,y}_{\tau } |\Dot{y}_{\tau }| d\tau 
    \\
    \leq & (\| x \|_{1 - \tmop{var}, [0, 1]} \vee \| y \|_{1 - \tmop{var},[0, 1]}) 
    \int^s_{\eta = 0} \Xi^{x, y}_{\eta} (|\Dot{x}_{\eta}| +|\Dot{y}_{\eta}|) d\eta   
  \end{align*}
  we get
  \begin{equation*}  \begin{gathered}
    \Xi^{x, y}_t  \leq  2 \sigma_A^2 \Gamma_{x,y} \Lambda_{x,y}
    \sqrt{|\check{\D} | \vee | \check{\D}'| \vee | \D | \vee | \D' |} 
    (1 +  e^{2 \sigma_A^2 k_R \Lambda_{x,y}})\\
     + 2 \sigma_A^2 k_R (\| x \|_{1 - \tmop{var}, [0, 1]} \vee \| y \|_{1
    - \tmop{var}, [0, 1]}) \int^s_{\eta = 0} \Xi^{x, y}_{\eta} (|\Dot{x}_{\eta}| +|\Dot{y}_{\eta}|) d\eta
  \end{gathered}  \end{equation*}
  thus, by Gronwall,
 \begin{equation*} \begin{gathered}
    \Xi_t  \leq  2 \sigma_A^2 \Gamma_{x,y} \Lambda_{x,y}
    \sqrt{|\check{\D} | \vee | \check{\D}'| \vee | \D | \vee | \D' |} 
    (1 +  e^{2 \sigma_A^2 k_R \Lambda_{x,y} })\\
     \cdot e^{2 \sigma_A^2 k_R (\| x \|_{1 - \tmop{var}, [0, 1]} \vee \|
    y \|_{1 - \tmop{var}, [0, 1]}) {(\| x \|_{1 - \tmop{var},
    [0, t]} + \| y \|_{1 - \tmop{var}, [0, t]})}}
  \end{gathered}  \end{equation*}
  hence
  \begin{equation*} \begin{gathered}
     \| \K_{\check{\D} \times \check{\D}'}^{x, y} (s,t) -
    \K^{x,y}_{\D \times \D'} (s,t) \|_{\infty, [0, 1]^2} \leq
    \\
    2 \sigma_A^2 \Gamma_{x,y} \Lambda_{x,y}
    \sqrt{|\check{\D} | \vee | \check{\D}'| \vee | \D | \vee | \D' |} 
    \cdot
    (1 +  e^{2 \sigma_A^2 k_R \Lambda_{x,y}})
    \cdot e^{{2 \sigma_A^2 k_R \Lambda_{x,y}}} 
  \end{gathered} \end{equation*}

  \paragraph{Part III : Existence and Uniqueness of limit}
  
  Given any sequence of partitions $\{ \D_n \times \D'_n \}$ with $| \D_n | \vee | \D_n' | \rightarrow 0$, due to the bounds we have just proven we have
  \[ 
     \{ \K_{\D_M \times \D'_n}^{x , y} \} \text{is a Cauchy sequence in } (C^0 ([0, 1] \times [0, 1] ;
     \mathbb{R}), \| \cdot \|_{\infty}) 
  \]
  and the limit $\K^{x, y}_{\varphi}$ does not depend on the sequence (i.e. the limit
  exists and is unique). 

  The limit is indeed unique: assume $K^{x, y}$ and $G^{x, y}$ are limits along two different sequences of partitions $\{ \D_n \times \D'_n \}$ and $\{ \mathcal{G}_n \times \mathcal{G}'_n \}$; then the sequence $\{ \mathcal{P}_n \times \mathcal{P}'_n \}$ such that $\mathcal{P}_{2n} \times \mathcal{P}'_{2n} := \D_n \times \D'_n $ and $\mathcal{P}_{2n + 1} \times \mathcal{P}'_{2n + 1} := \mathcal{G}_n \times \mathcal{G}'_n$ is still such that $| \mathcal{P}_n | \vee | \mathcal{P}'_n | \rightarrow 0$ hence the associated kernels have a limit which must be equal to both $K^{x, y}$ and $G^{x, y}$.
  
  Since PSD matrices form a closed set we moreover have that the matrices $\Sigma^{x, y}_{\varphi} (s,t)$
  obtained as limits using the previous result are all PSD. We can actually say more: they belong to $\tmop{PSD} (R)$.

  \paragraph{Part IV : Limit Kernel solves Equation (\ref{eqn:hom_kernel_appendix})} 

  We can finally conclude by proving that  $\K^{x, y}_{\varphi}$ is, in fact,  a solution of the PDE:
 \begin{align*}
    & \left| \K^{x, y}_{\varphi} (s,t) - \sigma^2_{a} + \int_{\eta = 0}^s
    \int_{\tau  = 0}^t (\sigma_A^2 V_{\varphi} (\Sigma^{x, y}_\varphi(\eta, \tau)) +
    \sigma_b^2) 
    \sprod{\Dot{x}_{\eta}}{\Dot{y}_{\tau }}_{\R^d}     d\eta d\tau   \right|  
    \\
    \leq &
    | \K^{x, y}_{\varphi} (s,t) - \K^{x,y}_{\D \times \D'} (s, t) | 
    \\
    & + \sigma_A^2 \int^s_{\eta = 0} \int^t_{\tau = 0} | V_{\varphi}
    (\tilde{\Sigma}_{\D \times \D'}^{  x, y} (\eta, \tau))
    - V_{\varphi} (\Sigma^{x, y}_\varphi(\eta, \tau)) |  
    | \sprod{\Dot{x}_{\eta}}{\Dot{y}_{\tau }}_{\R^d} |    d\eta d\tau  
    \\
    \leq & 
    o (1) + \sigma_A^2 k_R \int^s_{\eta = 0} \int^t_{\tau = 0} |
    \tilde{\Sigma}_{\D \times \D'}^{  x, y} (\eta, \tau) -
    \Sigma^{x, y}_\varphi(\eta, \tau) |_{\infty} 
    | \sprod{\Dot{x}_{\eta}}{\Dot{y}_{\tau }}_{\R^d} |    d\eta d\tau
    \\
    \leq &
    o (1) + \sigma_A^2 k_R \int^s_{\eta = 0} \int^t_{\tau = 0} |
    \tilde{\Sigma}_{\D \times \D'}^{  x, y} (\eta, \tau) -
    \Sigma_{\D \times \D'}^{  x, y} (\eta, \tau)
    |_{\infty}  |\sprod{\Dot{x}_{\eta}}{\Dot{y}_{\tau }}_{\R^d}|       d\eta d\tau  
    \\
    & + \sigma_A^2 k_R \int^s_{\eta = 0} \int^t_{\tau = 0} |
    \Sigma_{\D \times \D'}^{  x, y} (\eta, \tau) -
    \Sigma^{x, y}_\varphi(\eta, \tau) |_{\infty}  
    | \sprod{\Dot{x}_{\eta}}{\Dot{y}_{\tau }}_{\R^d} |    d\eta d\tau 
    \\
    = & o (1) + o (1) + o (1) = o (1)  
   \end{align*}
  thus
  \[ \left| \K^{x, y}_{\varphi} (s,t) - \sigma^2_{a} + \int_{\eta = 0}^s \int_{\tau 
     = 0}^s (\sigma_A^2 V_{\varphi} (\Sigma^{x, y}_\varphi(\eta, \tau)) + \sigma_b^2) 
     \sprod{\Dot{x}_{\eta}}{\Dot{y}_{\tau }}_{\R^d}     d\eta d\tau\right| = 0 \]
  i.e
  \[ \K^{x, y}_{\varphi} (s,t) = \sigma^2_{a} + \int_{\eta = 0}^s \int_{\tau  = 0}^t
     (\sigma_A^2 V_{\varphi} (\Sigma^{x, y}_\varphi(\eta, \tau)) + \sigma_b^2) 
     \sprod{\Dot{x}_{\eta}}{\Dot{y}_{\tau }}_{\R^d}     d\eta d\tau\]

 \textit{- Part V : Uniformity in $x,y$ on bounded sets} 
 
Note that, once again, the bounds we have just proven only depend on  the norms $\norm{\cdot}_{1-var, [0,1]}$ and $\norm{\cdot}_{\bX}$ of the paths,and they do so in an increasing manner. 
Since $\norm{\cdot}_{1-var, [0,1]} \leq \norm{\cdot}_{\bX}$ this means that we have uniform convergence rates on bounded sets on $\bX$.
\end{proof}

\begin{remark}
    We could have stated this result for $x \in W^{1,1}([0,1];\R^d)$ thus requiring the derivative to only be in $L^1$, this is done in an analogous way to the proof of Theorem (\ref{thm:main-inhom-Kernels}) using Ascoli-Arzela. The cost to pay for the decreased regularity of the driving path is the loss of uniform convergence bounds.
\end{remark}

\begin{remark}
    Being careful one could maintain uniform bounds, at the cost of slower convergence, and state the result for $x \in W^{1,1+\epsilon}([0,1];\R^d)$ for any $\epsilon > 0$. More specifically the bound would be proportional to $|\D|^{\frac{\epsilon}{1+\epsilon}}$.
\end{remark}

\vspace{15pt}
\begin{proposition}[Uniqueness]\label{prop:hom_uniqueness}
    Under the previous assumptions on the activation function, fix $x,y \in \bX$. 
    Then any two triples 
    $$K({s,t}) := (K^{x, x}({s, t}),K^{x, y}({s, t}),K^{y, y}({s, t}))$$ 
    $$G({s, t}) :=(G^{x, x}(s,t),G^{x, y}(s,t),G^{y, y}(s,t))$$ defined on $[0,1] \times [0,1]$,satisfying Equation (\ref{eqn:hom_kernel_appendix}) and such that 
    $$K^{x, x}({s, s}),K^{y, y}({t, t}),G^{x,x}({s, s}),G^{y,y}({t, t})>0$$ 
    for all $t \in [0,1]$ must be equal. 
\end{proposition}
\begin{proof}
    We will write $|K|_{\infty} := |K^{x, x}|\vee|K^{x, y}|\vee|K^{y, y}|$.
    To satisfy equation (\ref{eqn:hom_kernel_appendix}) the associated covariance matrices $\Sigma_K({\eta, \tau}), \Sigma_G({\eta, \tau})$ must be always $PSD_2$.
    Using the assumed bound from below the one from above given by continuity we can assume that they uniformly in time are contained in some $PSD_2(\Bar{R})$. This true for all choices $(x,x),(x,y),(y,y)$.
    \begin{align*}
    & |K^{x, y}({s, t}) - G^{x,y}({s, t})|  
        \\
        = &
     |\int_{\eta = 0}^s \int_{\tau  = 0}^t
     \sigma_A^2 (V_{\varphi} (\Sigma_K^{x, y}({\eta, \tau})) - V_{\varphi} (\Sigma_G^{x,y}({\eta, \tau}))) 
     \sprod{\Dot{x}_{\eta}}{\Dot{y}_{\tau }}_{\R^d}  d\eta d\tau| 
     \\
     \leq &
     \int_{\eta = 0}^s \int_{\tau  = 0}^t \sigma_A^2 
     |V_{\varphi} (\Sigma_K^{x, y}({\eta, \tau})) - V_{\varphi} (\Sigma_G^{x,y}({\eta, \tau}))|
     | \sprod{\Dot{x}_{\eta}}{\Dot{y}_{\tau }}_{\R^d}|  d\eta d\tau 
     \\
     \leq &
     \int_{\eta = 0}^s \int_{\tau  = 0}^t \sigma_A^2 
     k_{\Bar{R}}|\Sigma_K^{x, y}({\eta, \tau}) - \Sigma_G^{x,y}({\eta, \tau})|_{\infty}
     | \sprod{\Dot{x}_{\eta}}{\Dot{y}_{\tau }}_{\R^d}|  d\eta d\tau
     \\
     \leq &
      \sigma_A^2 k_{\Bar{R}} \int_{\eta = 0}^s \int_{\tau  = 0}^t
    (\sup_{0\leq t_1, s_1 \leq \eta \vee \tau }|K(t_1, s_1) - G(t_1, s_1)|_{\infty})
     | \sprod{\Dot{x}_{\eta}}{\Dot{y}_{\tau }}_{\R^d}|  d\eta d\tau
    \end{align*}
    Moreover this holds substituting $(x,y)$ with $(x,x)$ and $(y,y)$.
    
    Let $\Xi_t := \sup_{0\leq t_1, s_1 \leq \eta t }|K({t_1, s_1}) - G({t_1, s_1})|_{\infty})$ then
    \begin{equation*} \begin{gathered}
        |K({s, t}) - G({s, t})|_{\infty} \leq \sigma_A^2 k_{\Bar{R}}
     \int_{\eta = 0}^s \int_{\tau  = 0}^t  \Xi_{\eta \vee \tau}
     | \sprod{\Dot{x}_{\eta}}{\Dot{y}_{\tau }}_{\R^d}|  d\eta d\tau 
     \\
     \leq 
     \int_{\eta = 0}^s \int_{\tau  = 0}^t  \Xi_{\eta \vee \tau}
     |\Dot{x}_{\eta}||\Dot{y}_{\tau }|  d\eta d\tau
    \end{gathered} \end{equation*}
    thus
    \begin{align*}
     & \Xi_t \leq
     \int_{\eta = 0}^s \int_{\tau  = 0}^t  \Xi_{\eta \vee \tau}
     |\Dot{x}_{\eta}||\Dot{y}_{\tau }|  d\eta d\tau 
     \\
     \leq &
     \int_{\eta = 0}^s \int_{\tau  = 0}^\eta  \Xi_{\eta}
     |\Dot{x}_{\eta}||\Dot{y}_{\tau }|  d\eta d\tau +
     \int_{\eta = 0}^s \int_{\tau  = \eta}^t  \Xi_{\tau }
     |\Dot{x}_{\eta}||\Dot{y}_{\tau }|  d\eta d\tau
     \\
     \leq &
     \norm{y}_{1-var, [0,1]} \int_{\eta = 0}^s  \Xi_{\eta}
     |\Dot{x}_{\eta}|d\eta +
     \norm{x}_{1-var, [0,1]} \int_{\tau  = 0}^t  \Xi_{\tau }
     |\Dot{y}_{\tau }|  d\tau 
     \\
     \leq &
     (\norm{x}_{1-var, [0,1]} \vee \norm{y}_{1-var, [0,1]}) \int_{\eta = 0}^s  \Xi_{\eta}
     (|\Dot{x}_{\eta}| + |\Dot{y}_{\eta}| ) d\eta 
    \end{align*}

    By Gronwall we finally conclude that $\Xi_t = 0$ for all $t$, which concludes the proof.z
\end{proof}

\subsubsection{Proof of Theorem \ref{thm:phi-SigKer}: Part 1}

It is finally time to prove the first part of the main result of the paper, which we restate below for the reader's convenience.

Fix partitions $\{\D_M\}_{M \in \N}$ of $[0,1]$ with $|\D_M| \downarrow 0$.
Write $\Phi_{\varphi}^{\D_M, N}(\cdot) = \sprod{v^N}{S^{\D_M,N}_1(\cdot)}$ for the ResNet initialized with the \emph{time-homogeneous} scheme.

\begin{theorem}\label{thm:phi-SigKer-appendix}
Let the activation function $\varphi: \R \to \R$ be a linearly bounded, absolutely continuous map with exponentially bounded derivative.
For any subset $\X = \{x_1, \dots, x_n\} \subset \bX$ the following convergence in distribution holds
\[
    \lim_{M \to \infty} \lim_{N \to \infty} \Phi_{\varphi}^{M, N}(\X) 
    =
    \mathcal{N}(0,\K_{\varphi}(\X, \X))
\]
where the map $\K_{\varphi} : \bX \times \bX \to \R$ is given by the unique final values $\K_{\varphi}^{x,y}(1, 1)$ of the following integro-differential equation
\begin{equation}\label{eqn:hom_kernel-appendix}
\begin{gathered}
      \K_{\varphi}^{x, y}(s,t) =  \sigma_{a }^2  +  
      \int_{\eta = 0}^s \int_{\tau  = 0}^t
      \Big[ 
      \sigma_A^2 \mathbb{E}_{Z \sim \mathcal{N} (0, {\Sigma}^{x, y}_{\eta, \tau})} 
      [ \varphi(Z_1) \varphi(Z_2) ] 
      + 
      \sigma_b^2 
      \Big] 
     \sprod{\Dot{x}_{\eta}}{ \Dot{y}_{\tau }}_{\R^d}   d \eta d\tau
\end{gathered}
\end{equation}

with 
\[ 
    {\Sigma}^{x, y}_{s, t} = 
     \left(\begin{array}{c}
       \K_{\varphi}^{x, x}(s, s), \K_{\varphi}^{x, y}(s,t)\\
       \K_{\varphi}^{x, y}(s,t), \K_{\varphi}^{y, y}(t, t)
     \end{array}\right) 
\]
\end{theorem}

\begin{proof}
    The proof is now just a matter of combining Theorem \ref{thm:discrete-phi-SigKer} and Theorem \ref{simple_conv}.

    Under our hypotheses Theorem \ref{thm:discrete-phi-SigKer} tells us that, for any subset $\X = \{x_1, \dots, x_n\} \subset \bX$, we have in distribution
\[
    \lim_{N \to \infty} \Phi_{\varphi}^{\D_M, N}(\X) 
    =
    \mathcal{N}(0,\K_{\D_M \times \D_M}(\X, \X))
\]
thus to conclude we just have to prove that, still in distribution, it holds
\[
    \lim_{M \to \infty} \mathcal{N}(0,\K_{\D_M \times \D_M}(\X, \X))
    =
    \mathcal{N}(0,\K_{\varphi}(\X, \X))
\]
or, equivalently, that in $R^{n\times n}$ it holds $\lim_{M \to \infty} \K_{\D_M \times \D_M}(\X, \X) = \K_{\varphi}(\X, \X)$.

This last needed limit follows from Theorem \ref{simple_conv} with the sequence $\{\D_M \times \D_M\}_{M \in \N}$.
The uniqueness follows from Proposition \ref{prop:hom_uniqueness}.
\end{proof}

Like in the inhomogeneous case we can explicitly write the Kernel for the simplest case:
\begin{corollary}
    If $\sigma = id$, then 
    \begin{equation}\label{eqn:app_hom_id_case}
    \K_{id}^{x, y} (s,t) = \big( \sigma_{a}^2 + \frac{\sigma_b^2}{\sigma_A^2}\big) k_{sig}^{\sigma_A x,\sigma_A y}(s,t) - \frac{\sigma_b^2}{\sigma_A^2}
    \end{equation}
\end{corollary}

\begin{proof}
    We readily see that
    \begin{equation}
    \begin{gathered}
         \K_{id}^{x, y} (s,t) =  \sigma_{a }^2  +  
          \int_{\eta = 0}^s \int_{\tau  = 0}^t
          \Big[ 
          \sigma_A^2 \K_{id}^{x, y} (\eta, \tau)
          + 
          \sigma_b^2 
          \Big] 
         \sprod{\Dot{x}_{\eta}}{ \Dot{y}_{\tau }}_{\R^d}   d \eta d\tau
    \end{gathered}
    \end{equation}

    By substituting (\ref{eqn:app_hom_id_case}) for $\K_{id}^{x, y} (\eta, \tau)$ in the integral and using \cite{salvi2021signature}[Theorem 2.5]
    we have
    \begin{align*}
          & \sigma_{a }^2  +  
          \int_{\eta = 0}^s \int_{\tau  = 0}^t
          \Big[ 
          \sigma_A^2 
          \big\{ \big( \sigma_{a}^2 + \frac{\sigma_b^2}{\sigma_A^2}\big)  k_{sig}^{\sigma_A x,\sigma_A y}({\eta,\sigma}) - \frac{\sigma_b^2}{\sigma_A^2} \big\}
          + 
          \sigma_b^2 
          \Big] 
         \sprod{\Dot{x}_{\eta}}{ \Dot{y}_{\tau }}_{\R^d}   d \eta d\tau 
         \\
         = & 
         \sigma_{a }^2  +  
          \int_{\eta = 0}^s \int_{\tau  = 0}^t
          \sigma_A^2 
           \big( \sigma_{a}^2 + \frac{\sigma_b^2}{\sigma_A^2}\big)  k_{sig}^{\sigma_A x,\sigma_A y}({\eta,\sigma}) 
         \sprod{\Dot{x}_{\eta}}{ \Dot{y}_{\tau }}_{\R^d}   d \eta d\tau 
         \\
         = &
         \sigma_{a }^2  +  \big( \sigma_{a}^2 + \frac{\sigma_b^2}{\sigma_A^2}\big)
          \int_{\eta = 0}^s \int_{\tau  = 0}^t 
             k_{sig}^{\sigma_A x,\sigma_A y}({\eta,\sigma}) 
         \sprod{\sigma_A\Dot{x}_{\eta}}{ \sigma_A\Dot{y}_{\tau }}_{\R^d}   d \eta d\tau  
         \\
         = &
         \sigma_{a }^2  +  \big( \sigma_{a}^2 + \frac{\sigma_b^2}{\sigma_A^2}\big)
          (  k_{sig}^{\sigma_A x,\sigma_A y}({s, t}) - 1) 
          \\
          = &
          \big( \sigma_{a}^2 + \frac{\sigma_b^2}{\sigma_A^2}\big)
            k_{sig}^{\sigma_A x,\sigma_A y}({s, t})
          - \frac{\sigma_b^2}{\sigma_A^2}
    \end{align*}
    which means, by {uniqueness of solutions} (Proposition \ref{prop:hom_uniqueness}), that the thesis holds.
\end{proof}

\begin{remark}
    A proof similar to that given in the inhomogeneous case for $\varphi = ReLU$ does not work now since the covariance matrix is not, in general, degenerate for $x=y$ as in that case.
\end{remark}

\begin{lemma}\label{app:lemma:hom_simpler_form}
For all choices $(\sigma_a, \sigma_A, \sigma_b)$ and for all $\varphi$ as in Theorem \ref{thm:main-inhom-Kernels} we have, with abuse of notation and the obvious meaning, that
\[
    \K_{\varphi}^{x,y}(s,t; \sigma_a, \sigma_A, \sigma_b) = 
    \K_{\varphi}^{\sigma_A x, \sigma_A y}(s,t; \sigma_a, 1, \sigma_b \sigma_A^{-1})
\]
\end{lemma}
\begin{proof}
    Follow the exact same steps and arguments of \ref{app:lemma:inhom_simpler_form}.
\end{proof}

\subsection{The infinite-depth-then-width regime}

\subsubsection{The infinite-depth limit with finite-width}
\begin{proposition}
    Let $\{\D_M\}_{M \in \N}$ be a sequence of partitions of $[0,1]$ such that $|\D_M| \to 0$ as $M \to \infty$. Assume the activation function $\varphi$ is Lipschitz and linearly bounded. 
    Let $x \in \bX$ and let $\rho_M(t) := \sup \{ s \in \D_M : s \leq t\}$.
    Then, the the $\R^N$-valued process $t \mapsto S^{M,N}_{\rho_M(t)}(x)$ converges in distribution, as $M \to \infty$, to the solution $S^{N}(x)$ of the following Neural CDE
    \begin{equation}
        S^{N}_t(x) = a + \int_0^t \sum_{j=1}^d \big( A_j \varphi(S^{N}_s(x))  + b_j \big) dx_s^j
    \end{equation}
    where $A_j \in \R^{N \times N}$ and $b_j \in \R^N$ are sampled according in the definition of the homogeneous controlled ResNet. 
\end{proposition}

\begin{proof}
    Assume the $A_k$ and $b_k$ to be fixed for all choices of $\D_M$.
    The system has unique solution by \cite{FrizVictoir}[Theorems 3.7, 3.8] with 
    \[
        V_i(x) = A_i \varphi(x) + b_k
    \]
    noting that these are Lipschitz since composition of Lipschitz and linearly bounded.
    
    By reasoning as in the proof of Theorem \ref{thm:phi-SigKer} one gets that 
    \begin{equation*}
        \begin{gathered}
            \norm{S^{M, N}(x) - S^{N}(x)}_{\infty,[0,1]} \leq 
            \sqrt{|\D_M|}\exp\left\{K_{\norm{x}_{\bX}} (\norm{a}_{\R^N} + \sum_{k=1}^d \norm{A_k}_F + \norm{b_k}_F)\right\}
        \end{gathered}
    \end{equation*}
    for some constant $K_{\norm{x}_{\bX}}$ depending in an increasing fashion on the norm of the input path.
    
    Taking the expectation over $A_k$,$b_k$ and the initial condition leads to 
    \[
    \E[\norm{S^{M, N}(x) - S^{N}(x)}_{\infty,[0,1]}] \leq \Tilde{K}_{\norm{x}_{\bX}}\sqrt{|\D_M|} 
    \]
    for some other constant $\Tilde{K}_{\norm{x}_{\bX}}$ since the matrices and vectors are all distributed as Gaussians. One then concludes by portmanteau lemma considering Lipschitz functions on $C^0([0,1];\R^N)$.
    
\end{proof}

\begin{remark}
    Since the bounds $\Tilde{K}_{\norm{x}_{\bX}}$ depend increasingly on the norm of the input path they are uniform on bounded sets of $\bX$ and this result can be extended to arbitrary finite sets of paths $\X \subseteq \bX$
    with exactly the same portmanteau arguments.
\end{remark}

\begin{definition}[Randomized Signatures]
    We call \emph{Randomized Signatures} the solutions to 
    \[
    S^{N}_t(x) = a + \sum_{k=1}^d \int_0^t \big( A_k [ \varphi(S^{N}_{\tau}(x)) ] + b_k \big) dx_{\tau}^k
    \]
\end{definition}

These are the same objects defined in \cite{cuchiero2021expressive}.

\subsubsection{Infinite-depth-then-width limit: $\varphi = id$}

\begin{theorem}\label{thm:app_sig_id}
In the case $\varphi = id$ 
\[
    \frac{1}{N}\sprod{S^{N}_s(\X)}{S^{N}_t(\X)}_{\R^N} 
    \xrightarrow[N \to \infty]{\mathbb{L}^2}
    \K_{id}^{\X,\X}(s,t)
\]
moreover the convergence is of order $\mathcal{O}(\frac{1}{N})$
\end{theorem}

\begin{proof}
    This is proved in Appendix \ref{app:conv_sigKer}.
\end{proof}

\subsection{Commutativity of Limits}\label{app:sub:hom_commut}

In this section we are going to prove the commutativity of limits in the homogeneous case. The core arguments are the same as those employed for the inhomogeneous counterpart, this time however we won't be able to take advantage of \emph{ready-made} results from stochastic analysis, thus we are going to carefully obtain bounds in more direct ways.

Let $\{\D_M\}_{M \in \N}$ be a sequence of partitions of $[0,1]$ such that $|\D_M| \to 0$ as $M \to \infty$. 
Fix $N$ and the matrices $S_0, A_k, b_k$ for all partitions.  For any $x \in \bX$ let $S^{M,N}(x): \D_M \to \R^N$ be the homogeneous cResNet corresponding to the above quantities.
    
\begin{definition}
    Given $x \in \bX$ and a partition $\D_M$ define 
    \[
    \rho_M(t) := \sup \{ s \in \D_M : s \leq t\}
    \]
   Then define the piecewise constant extension $Z^{M,N}(x): [0,1] \to \R^N$
   \begin{equation}
       Z^{M,N}_t(x) := S^{M,N}_{\rho_M(t)}(x)
   \end{equation}
   and the integral extension $S^{M,N}(x): [0,1] \to \R^N$
   \begin{equation}
       S^{M,N}(x) = S_0^N + \sum_{k=1}^d \int_0^t  \big( A_k \varphi(Z^{M,N}_s(x))  + b_k \big) dx_s^k
   \end{equation}
\end{definition}

Note how the two extensions coincide on $\D_M$.

\begin{proposition}
Assume the activation function $\varphi$ is Lipschitz and linearly bounded.
There is a constant $ K_x>0$ independent of $N, M$ and increasing in $\norm{x}_{1-var}$ such that 
\[
\E\left[\norm{Z^{N,M}_{t}(x)}_{\infty, [0,1]}^2\right] 
\leq N K_x
\]
where the expectation is taken over the joint distribution of $S_0, \{A_k, b_k\}_{k=1,\dots, d}$
\end{proposition}

\begin{proof}
Let $t \in [t_m, t_{m+1})$ then
\begin{equation*}
\begin{gathered}
    |Z^{N,M}_{t}(x)| = |S^{N,M}_{t_m}(x)| \leq
    \\
    |S_0^N| + \sum_{i=1}^{d} |\int_0^t A^N_i \varphi(Z^{N,M}_{r}(x)) + b^N_i dx^i_r| \leq
    \\
    |S_0^N| + \sum_{i=1}^{d} \int_0^t |A^N_i \varphi(Z^{N,M}_{r}(x)) + b^N_i| |dx^i_r| \leq
    \\
    |S_0^N| + \sum_{i=1}^{d} \int_0^t \big( \norm{A^N_i}_{op} |\varphi(Z^{N,M}_{r}(x))| + |b^N_i| \big) |dx^i_r| \leq
    \\
    |S_0^N| + \sum_{i=1}^{d} \int_0^t \big( \norm{A^N_i}_{op}  \sqrt{2} C (\sqrt{N}+\norm{Z^{N,M}_{r}(x)}) + |b^N_i| \big) |dx^i_r| \leq
    \\
    \left[
    |S_0^N| +  (\sum_{i=1}^{d} \sqrt{2N} C \norm{A^N_i}_{op} + |b^N_i|) \norm{x}_{1-var, [0,1]} 
    \right]
    + \sqrt{2}C(\sum_{i=1}^{d}\norm{A^N_i}_{op})\sum_{i=1}^{d} \int_0^t |Z^{N,M}_{r}(x)|  |dx^i_r|
\end{gathered}
\end{equation*}
thus by \emph{Lemma 3.2} of \cite{FrizVictoir} we obtain
\begin{equation*}
\begin{gathered}
    |Z^{N,M}_{t}(x)| \leq
    \left[
    |S_0^N| +  (\sum_{i=1}^{d} \sqrt{2N} C \norm{A^N_i}_{op} + |b^N_i|) \norm{x}_{1-var, [0,1]} 
    \right]
    \exp{\big\{ \sqrt{2}C(\sum_{i=1}^{d}\norm{A^N_i}_{op}) \norm{x}_{1-var, [0,t]}  \big\}}
    \\
    \leq 
    \left[
    |S_0^N| +  (\sqrt{2N} C +  \sum_{i=1}^{d} |b^N_i|) \norm{x}_{1-var, [0,1]} 
    \right]
    \exp{\big\{ (\sum_{i=1}^{d}\norm{A^N_i}_{op})(1 + \sqrt{2}C \norm{x}_{1-var, [0,t]})  \big\}}
    \\
    \leq
    \left[
    |S_0^N| +  (\sqrt{2N} C +  \sum_{i=1}^{d} |b^N_i|) \norm{x}_{1-var, [0,1]} 
    \right]
    \prod_{i=1}^d \exp{\big\{(1 + \sqrt{2}C \norm{x}_{1-var, [0,1]})  \norm{A^N_i}_{op}  \big\}}
\end{gathered}
\end{equation*}

Hence using independence
\begin{equation*}
\begin{gathered}
    \E[ \sup_{t \in [0,1]}|Z^{N,M}_{t}(x)|^2 ] \leq
    \\
    2 \E \left[
    |S_0^N|^2 +  (2NC^2 +  \sum_{i=1}^{d} |b^N_i|^2) \norm{x}^2_{1-var, [0,1]} 
    \right]
    \prod_{i=1}^d \E\left[ \exp{\big\{2(1 + \sqrt{2}C \norm{x}_{1-var, [0,1]})  \norm{A^N_i}_{op}  \big\}}\right]
\end{gathered}
\end{equation*}

{
Note that $\norm{A^N_i}_{op} = \sqrt{\rho((A^N_i)^T A^N_i)}$ is the square root of the biggest eigenvalue $l^N$ of $(A^N_i)^T A^N_i$. The distribution of eigenvalues $\frac{1}{\sigma_A^2}(A^N_i)^T A^N_i$ is well understood as converging to the Marchenko–Pastur distribution (a classical distribution supported on the interval $[0,4]$), and $l^N$ converges almost surely to $4\sigma_A^2$ \cite{ Geman, Johnstone} with fluctuations around this limit having exponential tails \cite{Vergassola, Johansson}. It follows that $\E[\exp\{\lambda\norm{A^N_i}_{op}\}]$ is uniformly bounded in $N$, with the bound depending on $\sigma_A$ and $\lambda$. 
}

Moreover by H\"older
\[
\E[|S_0^N|^2] = {\E[\sum_{i=1}^N |[S_0^N]_i|^2]} =  \sigma^2_a {N}
\]
and similarly $\E[|b^N_i|^2] \leq \sigma^2_{b_i} {N}$.

Hence
\begin{equation*}
\begin{gathered}
    \E[ \sup_{t \in [0,1]}|Z^{N,M}_{t}(x)|^2 ] \leq
    \\
    2{N}  \left( \sigma^2_a +  (2 C^2 + \sum_{i=1}^d \sigma_{b_i}^2) \norm{x}^2_{1-var, [0,1]}\right)
    \E\left[ \exp{\big\{2(1 + \sqrt{2}C \norm{x}_{1-var, [0,1]})  \norm{A^N_1}_{op}  \big\}}\right]^d
    \\
    =: N K_x
\end{gathered}
\end{equation*}
\end{proof}

\begin{proposition}
Assume the activation function $\varphi$ is Lipschitz and linearly bounded.
$$\E\left[ \norm{S^{N,M}_{t}(x) - Z^{N,M}_{t}(x)}_{\infty, [0,1]}^2 \right] \leq {N|\D_M|}\Tilde{K}_x$$
where $\Tilde{K}_x$ is a universal constant depending only on $\norm{x}_{\bX}$ in an increasing manner.
\end{proposition}

\begin{proof}
Let $t \in [t_m, t_{m+1})$ then, using the bound in Remark \ref{app:rem:linear_bound_N},
\begin{equation*}
\begin{gathered}
    |S^{N,M}_{t}(x) - Z^{N,M}_{t}(x)| = 
    |\sum_{i=1}^d \int_{t_m}^t \big( A^N_i \varphi(S^{N,M}_{t_m}(x)) + b^N_i \big) dx^i_r | \leq
    \\
    \sum_{i=1}^d \int_{t_m}^t \big( \norm{A^N_i}_{op} |\varphi(S^{N,M}_{t_m}(x))| + |b^N_i| \big) |dx^i_r| \leq
    \\
    \sum_{i=1}^d \int_{t_m}^t \big( \norm{A^N_i}_{op} \sqrt{2}C(\sqrt{N} + |S^{N,M}_{t_m}(x)|) + |b^N_i| \big) |dx^i_r| \leq
    \\
    \big[ \sum_{i=1}^d  \norm{A^N_i}_{op} \sqrt{2}C(\sqrt{N} + |Z^{N,M}_{t_m}(x)|) + |b^N_i| \big]  \norm{x}_{1-var, [t_m,t]}  \leq 
    \\
    \big[ \sum_{i=1}^d  \norm{A^N_i}_{op} \sqrt{2}C(\sqrt{N} + |Z^{N,M}_{t_m}(x)|) + |b^N_i| \big] \norm{x}_{\bX} \sqrt{|\D_M|}
\end{gathered}
\end{equation*}
where we have used 
\[
\norm{x}_{1-var, [s,t]} = \int_s^t |\dot x_r| dr \leq \sqrt{\int_0^1 |\dot x_r| |t-s| dr} \leq \norm{x}_{\bX} \sqrt{t-s}
\]

thus
\begin{equation*}
\begin{gathered}
    \E\left[ \sup\limits_{t \in [0,1]}|S^{N,M}_{t}(x) - Z^{N,M}_{t}(x)|^2 \right] \leq
    \\
    2 \E\left[ \sup\limits_{t_m \in \D_M}  \big[ \sum_{i=1}^d  \norm{A^N_i}^2_{op} 2C^2({N} + |Z^{N,M}_{t_m}(x)|^2) + |b^N_i|^2 \big]  \norm{x}_{\bX}^2 {|\D_M|}
    \right]
    \\
    \leq {N}\Bar{K}_x\norm{x}^2_{\bX} {|\D_M|}
\end{gathered}
\end{equation*}
\end{proof}

\begin{proposition}\label{app:prop:comm_final_bound}
Assume the activation function $\varphi$ is Lipschitz and linearly bounded.
$$\E\left[ \norm{S^{N,M}_{t}(x) - S^{N,M'}_{t}(x)}_{\infty, [0,1]}^2 \right] \leq {N}\Tilde{K}_x ({|D_M|\vee|\D_{M'}|})$$
where $\Tilde{K}_x$ is a universal constant depending only on $\norm{x}_{\bX}$ in an increasing manner.
\end{proposition}

\begin{proof}
Let $t \in [t_m, t_{m+1})$ then
\begin{equation*}
\begin{gathered}
    |S^{N,M}_{t}(x) - S^{N,M'}_{t}(x)| \leq
    |\sum_{i=1}^d \int_{0}^t |A^N_i \big( \varphi(Z^{N,M}_{t}(x)) - \varphi(Z^{N,M'}_{t}(x)) \big)| |dx^i_r| \leq
    \\
    \sum_{i=1}^d \int_{0}^t \big( \norm{A^N_i}_{op} |\varphi(Z^{N,M}_{t}(x)) - \varphi(Z^{N,M'}_{t}(x))| \big) |dx^i_r| \leq
    \\
    \sum_{i=1}^d \int_{0}^t \norm{A^N_i}_{op} K |Z^{N,M}_{t}(x) - Z^{N,M'}_{t}(x)|  |dx^i_r| =
    \\
    K \sum_{i=1}^d   \norm{A^N_i}_{op} \int_{0}^t  |Z^{N,M}_{t}(x) - Z^{N,M'}_{t}(x)|  |dx^i_r|
\end{gathered}
\end{equation*}

Using triangle inequality we obtain
\begin{equation*}
\begin{gathered}
    |S^{N,M}_{t}(x) - S^{N,M'}_{t}(x)| \leq
    \\
    K \sum_{i=1}^d    \norm{A^N_i}_{op} \int_{0}^t  |Z^{N,M}_{t}(x) - S^{N,M}_{t}(x)|  |dx^i_r| +
    \\
    K  \sum_{i=1}^d  \norm{A^N_i}_{op} \int_{0}^t  |S^{N,M}_{t}(x) - S^{N,M'}_{t}(x)|  |dx^i_r| +
    \\
    K  \sum_{i=1}^d  \norm{A^N_i}_{op}  \int_{0}^t  |S^{N,M'}_{t}(x) - Z^{N,M'}_{t}(x)|  |dx^i_r| 
\end{gathered}
\end{equation*}

Note how 
\begin{equation*}
\begin{gathered}
    K  \sum_{i=1}^d  \norm{A^N_i}_{op}  \int_{0}^t  |Z^{N,M}_{t}(x) - S^{N,M}_{t}(x)|  |dx^i_r| \leq
    \\
    K \sum_{i=1}^d  \norm{A^N_i}_{op} \int_{0}^t  \big[ \sum_{k=1}^d  \norm{A^N_k}_{op} \sqrt{2}C(\sqrt{N} + |Z^{N,M}_{t_m}(x)|) + |b^N_k| \big]  \norm{x}_{\bX} \sqrt{|\D_M|}   |dx^i_r| \leq
    \\
    K ( \sum_{i=1}^d  \norm{A^N_i}_{op} )
    \big[ \sum_{i=k}^d  \norm{A^N_k}_{op} \sqrt{2}C(\sqrt{N} + \Gamma_x) + |b^N_k| \big] \norm{x}_{1-var}   \norm{x}_{\bX} \sqrt{|\D_M|} 
    =:  \Lambda_x \sqrt{|\D_M|} 
\end{gathered}
\end{equation*}
with 
\[
\Gamma_x  = \left[
    |S_0^N| +  (\sqrt{2N} C +  \sum_{i=1}^{d} |b^N_i|) \norm{x}_{1-var, [0,1]} 
    \right]
    \prod_{i=1}^d \exp{\big\{(1 + \sqrt{2}C \norm{x}_{1-var, [0,1]})  \norm{A^N_i}_{op}  \big\}}
\]

Hence it follows that 
\begin{equation*}
\begin{gathered}
    |S^{N,M}_{t}(x) - S^{N,M'}_{t}(x)| \leq
    \\
    (\sqrt{|\D_M|} + \sqrt{|\D_{M'}|}) \Lambda_x +
     K \sum_{i=1}^d  \norm{A^N_i}_{op} \int_{0}^t  |S^{N,M}_{t}(x) - S^{N,M'}_{t}(x)|  |dx^i_r| 
\end{gathered}
\end{equation*}
and with Gronwall 
\[
\norm{S^{N,M}_{t}(x) - S^{N,M'}_{t}(x)}_{\infty} \leq 
2\sqrt{|\D_M|\vee|\D_{M'}|} \Lambda_x \exp\{K \big( \sum_{i=1}^d  \norm{A^N_i}_{op} \big)\norm{x}_{1-var}\}
\]

Taking expectations of the squares and proceeding as before we have the thesis.
\end{proof}

We can finally prove the main bound which will allow the exchange of limits:

\begin{theorem}
    Let $\{\D_M\}_{M \in \N}$ be a sequence of partitions of $[0,1]$ such that $|\D_M| \to 0$ as $M \to \infty$. Assume the activation function $\varphi$ is Lipschitz and linearly bounded. Fix the matrices $S_0, A_k, b_k$ for all partitions. 
    Let $S^{N}(x)$ be the solution of the following Neural CDE
    \begin{equation}
        S^{N}_t(x) = S_0^N + \sum_{k=1}^d \int_0^t  \big( A_k \varphi(S^{N}_s(x))  + b_k \big) dx_s^k
    \end{equation}
    Then there exist a constant $\mathbf{K}_x$ independent of $N,M$ and increasing function of $\norm{x}_{\bX}$ such that 
    \begin{equation}
        \E[\sup\limits_{t \in [0,1]} \norm{S^{M,N}_t(x) - S^N_t(x)}^2_{\R^N}] \leq {N{|D_M|}} \mathbf{K}_x 
    \end{equation}
\end{theorem}

\begin{proof}
    Note first that $S^N(x)$ is well defined since the system has unique solution by \cite{FrizVictoir}[Theorems 3.7, 3.8] with 
    \[
        V_k(x) = A_k \varphi(x) + b_k
    \]
    noting that these are Lipschitz since composition of Lipschitz and linearly bounded.

    Moreover the proof tells us that $S^{N,M}_{t}(x)$ is a cauchy sequence in $C^0([0,1]; \R^N)$ hence for any $M > 0$, eventually in $M'$ it holds
    \[
    \sup\limits_{t \in [0,1]} \norm{S^{M,N}_t(x) - S^N_t(x)}_{\R^N} \leq 2\sup\limits_{t \in [0,1]} \norm{S^{M,N}_t(x) - S^{M',N}_t(x)}_{\R^N}
  \]
  Proposition (\ref{app:prop:comm_final_bound}) then gives the sought after bound.
    
\end{proof}

\begin{theorem}
Assume the activation function $\varphi$ is Lipschitz and linearly bounded.
There is a constant $C$ depending only on $\norm{x}_{\bX}$ in an increasing fashion such that:  
    \[
    \sup_{N \geq 1} \sup_{t \in [0,1]} \mathcal{W}_1(\mu_{t}^{M, N}(x), \mu_{t}^{N}(x)) \leq \Bar{\mathbf{K}}_x \sqrt{|\D_M|}
    \]
    where $\mu_{t}^{M,N}$ is the distribution of $\sprod{v^N}{{S}^{M, N}_t(x)}$ and $\mu_{t}^{N}(x)$ that of  $\sprod{v^N}{{S}^{N}_t(x)}$ for some independent vector with iid entries $[v^N]_{\alpha} \sim \mathcal{N}(0, \frac{1}{N})$.
\end{theorem}

\begin{proof}
    Let $G: \R \to \R$ be 1-Lipschitz, we have
    \[
    \E[|G(\sprod{v^N}{{S}^{M, N}_t(x)}) - G(\sprod{v^N}{{S}^{N}_t(x)})|] \leq  \E\left[|\sprod{v^N}{{S}^{M, N}_t(x)} - \sprod{v^N}{{S}^{N}_t(x)}|\right] 
    \]
    which leads to
    \begin{align*}
        \E\left[|\sprod{v^N}{{S}^{M, N}_t(x)} - \sprod{v^N}{{S}^{N}_t(x)}|\right] 
        & =  \E\left[|\sprod{v^N}{{S}^{M, N}_t(x) - {S}^{N}_t(x)}|\right] 
        \\
        & \leq \left( \E\left[|\sprod{v^N}{{S}^{M, N}_t(x) - {S}^{N}_t(x)}|^2 \right] \right)^{\frac{1}{2}} 
        \\
        & =  \left(\frac{1}{N} \E\left[ \norm{{S}^{M, N}_t(x) - {S}^{N}_t(x)}^2\right]\right)^{\frac{1}{2}}  
        \\
        & \leq  \left(\frac{1}{N} \E\left[\sup_{t \in [0,1]}\norm{{S}^{N}_t(x) - {S}^{M, N}_t(x)}^2\right]\right)^{\frac{1}{2}}  
        \\
        & \leq \sqrt{\mathbf{K}_x|\D_M|}
    \end{align*}
    hence 
    \[
    \mathcal{W}_1(\mu_{t}^{M, N}(x), \mu_{t}^{N}(x)) \leq \sqrt{\mathbf{K}_x |\D_M|}
    \]
\end{proof}

{
\begin{remark}
    Note that the exact arguments, being of $L^2$ type, can be repeated for the "stacked" vector $(S^{N,M}(x_1), \dots, S^{N,M}(x_N))$ extending (qualitatively) the bounds to the multi-input case.
\end{remark}
}

\begin{corollary}[Thm. \ref{thm:phi-SigKer}]\label{app:cor:hom_GP_second_part}
    The limits in Thm. \ref{simple_conv} commute.
\end{corollary}

\begin{proof}
    By the classical Moore-Osgood theorem  we need to prove that one of the two limits is uniform in the other, for example that the limit in distribution as $M \to \infty$ is uniform in $N$ in some metric which describes convergence in distribution. But this is just the content of the previous result{, extended to the multi-input case.}
\end{proof}

\subsection{An alternative proof for the case $\varphi=id$}\label{app:conv_sigKer}

In this last section we prove Theorem \ref{thm:app_sig_id}.
between \emph{Randomized Signature Kernels} and the original \emph{Signature Kernel}.

We are going to consider words $I = (i_1, \cdots, i_k) \in \mathcal{A}_d^k$ in the alphabet $\mathcal{A}_d := \{1,\dots, d\}$; the space of all words will be denoted by $\mathbb{W}_d$, the length of a word by $|I| = k$ and the sum of its elements by $\norm{I} = \sum_{m=1}^k i_m$.

\vspace{5pt}
Recall we consider randomized Signatures
\[
   S^{N}_t(x) = S_0 + \sum_{k=1}^d \int_0^t \big( A_k S^{N}_{\tau}(x)  + b_k \big) dx_{\tau}^k
\]
where $x \in \bX$ and 
\[
    [A_{k}]_{\alpha}^{\beta} \sim \mathcal{N}(0,\frac{\sigma_{A}^2}{N}) 
    \hspace{15pt}
    [S_0]_{\alpha}\sim \mathcal{N}(0,\sigma_{S_0}^2)
    \hspace{15pt}
    [b_{k}]_{\alpha} \sim \mathcal{N}(0,\sigma_b^2)
\]

\vspace{10pt}
Our goal is that of proving the following result:
\begin{theorem}
    In the assumptions stated above
    \[
        \lim_{N \to \infty} \E\Big[
        \frac{1}{N}\langle S^N_s(x), S^N_t(y) \rangle_{\R^N}\Big] = 
        \big( \sigma_{S_0}^2 + \frac{\sigma_b^2}{\sigma_A^2}\big)   k_{sig}^{\sigma_A x,\sigma_A y}({s, t}) - \frac{\sigma_b^2}{\sigma_A^2}
    \]
    and the variance around the limit is of order $O(\frac{1}{N})$.
\end{theorem}

\vspace{15pt}
We know, see \cite{10.1214/EJP.v17-2136}[Remark 2.10], that it is possible to write a closed form for $S^N_t(x)$
which \emph{decouples} the effects of the vector fields 
and those of the driving control using the \emph{Signature}:

\begin{equation}\label{eqn:decoupling}
    S^N_t(x) = \sum_{I \in \mathbb{W}_d} V_{I}(S^N_0) \hspace{3pt} Sig^{I}_{0,s}(x)
\end{equation}

where $\mathbb{W}_d$ is the set of words in the alphabet $\{1,\dots,d\}$ and if $I = (i_1,\dots,i_k)$ then $[V_I(z)]_{\alpha} := V_{i_1} \cdots V_{i_k} \langle e_{\alpha}, \cdot \rangle (z)$ with $V_i(z) := A_i z + b_i$ and $Vf(z) = := df_z [V_j(z)]$ for $f \in C^{\infty}(\R^N; \R)$.
Notice that $V_{I}(S^N_0) \in \R^N$ and $Sig^{I}_{0,s}(x) \in \R$.

\vspace{5pt}
In fact the "Taylor expansion" with respect of the signature of $f(S^N_t(x))$ for $f \in C^{\infty}(\R^N; \R)$ is 

\[
    f(S^N_t(x)) = \sum_{I \in \mathbb{W}_d} V_{I}f(S^N_0) \hspace{3pt} Sig^{I}_{0,s}(x)
\]
where, with $I$ as above, $V_If(x) := V_{i_1} ( V_{i_2} \cdots  ( V_{i_k} f) \cdots )(x)$ with $$V_jf(x) := df_x [V_j(x)]$$

\vspace{15pt}

From Equation \ref{eqn:decoupling} we also get
\begin{align*}
    \langle S^N_s(x), S^N_t(y) \rangle_{\R^N} 
    &= \sum_{I \in \mathbb{W}_d} \sum_{J \in \mathbb{W}_d} 
    \langle V_{I}f(S^N_0) , V_{J}f(S^N_0)  \rangle_{\R^N} \hspace{3pt} Sig^{I}_{0,s}(x) Sig^{J}_{0,t}(y) 
\end{align*}

where $y$ is another control.
If we could exchange expectation with the series we would thus get

\begin{equation*}
    \mathbb{E}[\langle S^N_s(x), S^N_t(y) \rangle_{\R^N}]
    =
    \sum_{I \in \mathbb{W}_d} \sum_{J \in \mathbb{W}_d} 
    \E [ \langle V_{I}f(S^N_0) , V_{J}f(S^N_0)  \rangle_{\R^N} ]
    \hspace{3pt} Sig^{I}_{0,s}(x) Sig^{J}_{0,t}(y)
\end{equation*}

Thus we have to study the expectations 
$\E [ \langle V_{I}f(S^N_0) , V_{J}f(S^N_0)  \rangle_{\R^N} ]$.

Note that if $I = ()$, the empty word, then $V_{I}f(S^N_0) = S^N_0$.

\subsubsection{Products of Gaussian matrices}

\vspace{20pt}
The most important result which will make our plan succeed is the following classical theorem:

\begin{theorem}[Isserlis]
 Let $(X_1, \dots, X_N)$ be a zero mean multivariate normal vector, then
 \[
    \E[X_1 \cdots X_N] = \sum_{p \in P^2_n} \prod_{\{i,j\} \in p} \E[X_i X_j]
 \]
 where the sum is over all distinct ways of partitioning $\{1,\ldots,N\}$ into pairs $\{i,j\}$, and the product is over the pairs contained in $p$.
\end{theorem}

\begin{proposition}\label{thm:randomized_sig_kernel}
 Assume $I,J \in \mathbb{W}_d$ such that $|I|+|J| > 0$.
 In the hypotheses stated above 
 \[
    \frac{1}{N}\E [ \langle V_{I}f(S^N_0) , V_{J}f(S^N_0)  \rangle_{\R^N} ]
    =
    (\sigma_{S_0}^2 + \frac{\sigma_{b}^2}{\sigma_A^2})\sigma_A^{|I|+|J|}  \big[ \delta_I^J + O(\frac{1}{N}) \big]
 \]

\end{proposition}

\begin{proof}

    Let us first consider the case $\sigma_b=0$.
    
    Taking $f(z) = \langle v, z \rangle_{\R^N}$ for some $v \in \R^N$ and $V_j(z) := A_jz $ we get
    \[
        V_jf(x) := d(\langle v, \cdot \rangle_{\R^N})_x [A_jx] = v^TA_jx = \langle A_j^Tv, x \rangle_{\R^N}
    \]
    hence by induction 
    \[
        V_If(x) = \langle A_{i_1}^T \cdots A_{i_k}^T v, x \rangle_{\R^N}
        = \langle v, A_{i_k} \cdots A_{i_1} x \rangle_{\R^N}
    \]
    and in our notation
    \[
    V_I(x) = A_Ix = A_{i_k} \cdots A_{i_1}x
    \]
    
    \vspace{5pt}
    We are thus interested in the quantities $\E [ (S^N_0)^T A_I^T A_J S^N_0 ]$.
    
    \vspace{15pt}
    
    Let $\alpha, \beta \in \{1, \dots, N\}$ and given a matrix $M \in \R^{N \times N}$ write $[M]_{\alpha}^{\beta}$ for its component in row $\alpha$ and column $\beta$. Remember how 
    \[
        [M_1 M_2]_{\alpha}^{\beta} = \sum_{\gamma = 1}^N \hspace{3pt} [M_1]_{\alpha}^{\gamma} [M_2]_{\gamma}^{\beta}
    \]
    
    Thus we have 
    \begin{equation*}
        \begin{gathered}
            \E\Big[ (S^N_0)^T A^T_I A_J S^N_0 \Big] 
            = \sum_{n=1}^N\sum_{m=1}^N \E\Big[ [S^N_0]_n [S^N_0]_m [A^T_I A_J]^m_n \Big]
            \\=
            \sum_{n=1}^N\sum_{m=1}^N \E\Big[ [S^N_0]_n [S^N_0]_m \Big] \E\Big[[A^T_I A_J]^m_n \Big]
        \end{gathered}
    \end{equation*}

    Moreover, for $I = (i_1, \dots, i_k)$, we obtain
    \begin{align*}
        [A_I]_{\alpha}^{\beta} 
        &= \sum_{\delta_1 = 1}^N \hspace{3pt} [A_{(i_2,\dots,i_{k})}]_{\alpha}^{\delta_1} [A_{i_1}]_{\delta_1}^{\beta}\\
        &= \sum_{\delta_{k-1} = 1}^N \dots \sum_{\delta_1 = 1}^N \hspace{3pt} [A_{i_k}]_{\alpha}^{\delta_{k-1}} [A_{i_{l-1}}]_{\delta_{k-1}}^{\delta_{k-2}} \cdots [A_{i_{2}}]_{\delta_{2}}^{\delta_1}[A_{i_1}]_{\delta_1}^{\beta}\\
        &= \sum_{\bar{\delta} \in \Lambda^{N,k}_{\alpha,\beta}} 
        \prod_{n=1}^k \hspace{3pt} [A_{i_n}]_{\delta_{n}}^{\delta_{n-1}}
    \end{align*}
    
    where $\Lambda^{N,k}_{\alpha,\beta} := \{ (\delta_0, \dots, \delta_{k}) \in \{1,\dots,N\}^{k+1} : \delta_k = \alpha \text{ and } \delta_{0} = \beta\}$. With this notation we can write 
    
    \begin{align*}
        [A_I^T A_J]_{\alpha}^{\beta} 
        &= \sum_{\gamma = 1}^N \hspace{3pt} [A_I^T]_{\alpha}^{\gamma} [A_J]_{\gamma}^{\beta}
        = \sum_{\gamma = 1}^N \hspace{3pt} [A_I]^{\alpha}_{\gamma} [A_J]_{\gamma}^{\beta}\\
        &= \sum_{\gamma = 1}^N \sum_{\bar{\delta} \in \Lambda^{N,|I|}_{\gamma,\alpha}} \sum_{\bar{\epsilon} \in \Lambda^{N,|J|}_{\gamma,\beta}} \prod_{n=1}^{|I|} \prod_{m=1}^{|J|}\hspace{3pt} [A_{i_n}]_{\delta_{n}}^{\delta_{n-1}} [A_{j_m}]_{\epsilon_{m}}^{\epsilon_{m-1}}\\
    \end{align*}
    
    thus 
    \begin{align}
        \E\Big[[A_I^T A_J]_{\alpha}^{\beta} \Big] &= 
        \sum_{\gamma = 1}^N \sum_{\bar{\delta} \in \Lambda^{N,|I|}_{\gamma,\alpha}} \sum_{\bar{\epsilon} \in \Lambda^{N,|J|}_{\gamma,\beta}} \E \Big[ \prod_{n=1}^{|I|} \prod_{m=1}^{|J|}\hspace{3pt} [A_{i_n}]_{\delta_{n}}^{\delta_{n-1}} [A_{j_m}]_{\epsilon_{m}}^{\epsilon_{m-1}} \Big]
    \end{align}

    \vspace{5pt}
    The time is ripe for the application of \emph{Isserlis' Theorem}.

    \vspace{10pt}
    First of all notice how the sum in Isserlis runs over the possible pairings of the index set which in our case is the set of elements of the concatenation 
    $$I*J = (i*j_1, \dots, i*j_{|I|+|J|}) = (i_1, \dots, i_{|I|}, j_1, \dots, j_{|J|})$$ of $I$ and $J$. Then
    
    \begin{equation*}
        \begin{gathered}
            \E\Big[[A_I^T A_J]_{\alpha}^{\beta} \Big] \\ 
        =\sum_{\gamma = 1}^N \sum_{\bar{\delta} \in \Lambda^{N,|I|}_{\gamma,\alpha}} \sum_{\bar{\epsilon} \in \Lambda^{N,|J|}_{\gamma,\beta}} \sum_{p \in P^2_{|I|+|J|}} \prod_{\{a,b\} \in p} \E \Big[[A_{(i*j)_a}]_{(\delta*\epsilon)_{a}}^{(\delta*\epsilon)'_{a}} [A_{(i*j)_b}]_{(\delta*\epsilon)_{b}}^{(\delta*\epsilon)'_{b}} \Big]\\
        = \sum_{p \in P^2_{|I|+|J|}} \sum_{\gamma = 1}^N \sum_{\bar{\delta} \in \Lambda^{N,|I|}_{\gamma,\alpha}} \sum_{\bar{\epsilon} \in \Lambda^{N,|J|}_{\gamma,\beta}} \prod_{\{a,b\} \in p} \E \Big[[A_{(i*j)_a}]_{(\delta*\epsilon)_{a}}^{(\delta*\epsilon)'_{a}} [A_{(i*j)_b}]_{(\delta*\epsilon)_{b}}^{(\delta*\epsilon)'_{b}} \Big]
        \end{gathered}
    \end{equation*}

    where $$(\delta*\epsilon) = (\delta_1,\dots,\delta_{|I|},\epsilon_1,\dots,\epsilon_{|J|})$$
    and  
    $$(\delta*\epsilon)' = (\delta_0,\dots,\delta_{|I|-1},\epsilon_0,\dots,\epsilon_{|J|-1})$$
    
    \vspace{10pt}
    In particular this sum is $0$ by default if $I*J$ has an odd number of elements. This means that
    \[
         \E\Big[[A_I^T A_J]_{\alpha}^{\beta} \Big] = 0
    \]
    whenever $|I| + |J|$ is odd.
    
    Moreover, even if $|I| + |J|$ is even, the pairings must be in such a way that no factor of the product vanishes; since we are working with matrices with independent normal entries this is equivalent to requiring for each $\{a,b\} \in p \in P^2_{|I|+|J|}$ that $(i*j)_a = (i*j)_b$, $(\delta*\epsilon)_{a} = (\delta*\epsilon)_{b}$, and $(\delta*\epsilon)'_{a} = (\delta*\epsilon)'_{b}$. 
    
    \vspace{5pt}
    Evidently then $I*J$ must be at least "pairable" with pairs of identical indices! In this case we will say that $I*J$ is \emph{twinable} and we will get a factor $\sigma_A^2 N^{-1}$ out of every $\E \Big[[A_{(i*j)_a}]_{(\delta*\epsilon)_{a}}^{(\delta*\epsilon)'_{a}} [A_{(i*j)_b}]_{(\delta*\epsilon)_{b}}^{(\delta*\epsilon)'_{b}} \Big]$ thus
    
    \[
        \E\Big[[A_I^T A_J]_{\alpha}^{\beta} \Big] = \Big(\frac{\sigma_A^2}{N}\Big)^{\frac{|I|+|J|}{2}} \Xi(I,J,N,\alpha,\beta)
    \]
    
    where 
    \begin{equation*}
        \begin{gathered}
             \Xi(I,J,N,\alpha,\beta) 
             \\ =
            |\{(p,\gamma,\bar{\delta},\bar{\epsilon}) :  \forall \{a,b\} \in p. \hspace{3pt} \E \Big[[A_{(i*j)_a}]_{(\delta*\epsilon)_{a}}^{(\delta*\epsilon)'_{a}} [A_{(i*j)_b}]_{(\delta*\epsilon)_{b}}^{(\delta*\epsilon)'_{b}} \Big]\neq 0\}|
        \end{gathered}
    \end{equation*}

    \vspace{10pt}
    Define, given $p \in P^2_{|I|+|J|}$,
    \begin{equation*}
        \begin{gathered}
            \omega(p,I,J,N,\alpha,\beta) 
            \\ = 
            |\{(\gamma,\bar{\delta},\bar{\epsilon}) :  \forall \{a,b\} \in p. \hspace{3pt} \E \Big[[A_{(i*j)_a}]_{(\delta*\epsilon)_{a}}^{(\delta*\epsilon)'_{a}} [A_{(i*j)_b}]_{(\delta*\epsilon)_{b}}^{(\delta*\epsilon)'_{b}} \Big]\neq 0\}|
        \end{gathered}
    \end{equation*}
    
    so that
    \[
    \Xi(I,J,N,\alpha,\beta) = \sum_{p \in P^2_{|I|+|J|}} \omega(p,I,J,N,\alpha,\beta)
    \]
    
    \vspace{10pt}
    First of all notice how
    \[
        \omega(p,I,J,N,\alpha,\beta) \leq N^{\frac{|I|+|J|}{2}}
    \]
    
        in fact considering only the constraints given by $\alpha$ and $\beta$ we have $N^{|I|}$ ways to choose $\bar{\delta} \in \{1,\dots,N\}^{|I|+1}$ and $N^{|J|}$ ways to choose $\bar{\epsilon} \in \{1,\dots,N\}^{|J|+1}$, but since they must come in pairs as dictated by $p$ we actually have $N^{\frac{|I|+|J|}{2}}$ possible choices \emph{i.e.} $N$ per pair.
    
    \vspace{5pt}
    Notice however how we have equality if and only if $I=J$, $\alpha = \beta$ and the pairings are such that $p \ni \{a,b\} = \{a,|I|+a\}$ for $a \in \{1,\dots,|I|\}$
    \emph{i.e.} every element in $I$ is paired to the corresponding one in $J = I$.
    
    The \emph{if} part is easy to see: the full constraints are just
    $\delta_{|I|} = \epsilon_{|I|} = \gamma$, $\delta_a = \epsilon_a$ for every $1 < a \leq |I|$ and
    $\alpha = \delta_{0} = \epsilon_{0} = \beta$ thus there are $$N \times N^{|I| - 1} \times 1 = N^{|I|} = N^\frac{|I|+|J|}{2}$$ choices to make.
    
    The \emph{only if} is more complicated and follows from the constraint $\delta_{|I|} = \epsilon_{|J|}$. Fix $\gamma$ and assume $i_{|I|}$ is not paired with $j_{|J|}$, then the choices for $(\delta*\epsilon)_a = (\delta*\epsilon)_b$ for 2 out of the $\frac{|I|+|J|}{2}$ pairs are constrained to be $\gamma$, all in all we have $N$ choices for $\gamma$ and at most $N^{\frac{|I|+|J|}{2}-2}$ for the other entries, thus at most $N^{\frac{|I|+|J|}{2}-1}$ in total. But then we must require $i_{|I|}$ and $j_{|J|}$ to be paired. The same argument can now be repeated with $i_{|I|-1}$ and $j_{|I|-1}$ since we have established that $\{|I|,|I|+|J|\} \in p$, thus not only $\delta_{|I|} = (\delta*\epsilon)_{|I|} = (\delta*\epsilon)_{|I|+|J|} = \epsilon_{|J|}$ but also $\delta_{|I|-1} = (\delta*\epsilon)'_{|I|} = (\delta*\epsilon)'_{|I|+|J|} = \epsilon_{|J|-1}$. 
    This goes on until, without loss of generality, we run out of elements in $I$.
    If the same happens for $J$ (\emph{i.e.} $|I|=|J|$) we are done since we have proved that $\forall a \in 1, \dots, |I|$ we have $i_a = j_a$ and $\alpha = (\delta*\epsilon)'_{1} = (\delta*\epsilon)'_{|I| + 1} = \beta$.
    Otherwise $|J| \geq 2 + |I|$ and $(j_{|J|-|I|},\dots,j_{1})$ are paired between themselves. 
    But since $\{1,|I|+(|J|-|I|+1)\}\in p$ we have $\alpha = (\delta*\epsilon)'_{1} = (\delta*\epsilon)'_{|I|+(|J|-|I|+1)} = (\delta*\epsilon)_{|I|+(|J|-|I|+1) - 1} = (\delta*\epsilon)_{|I|+(|J|-|I|)} = \epsilon_{|J|-|I|}$. Which means that there is no free choice one of the remaining couples (\emph{i.e.} the one containing $|I|+(|J|-|I|)$ corresponding to $j_{|J|-|I|}$) thus, reasoning just as before, we cut the number of choices of at least a factor $N$.
    

    \vspace{10pt}
    We have just shown that, given a pairing $p$,
    $$\Big(\frac{\sigma_A^2}{N}\Big)^{\frac{|I|+|J|}{2}} \omega(p,I,J,N,\alpha,\beta) \leq  \\ \Big(\frac{\sigma_A^2}{N}\Big)^{\frac{|I|+|J|}{2}} N^{\frac{|I|+|J|}{2} - 1} = \frac{\sigma_A^{|I|+|J|}}{N} $$
    except when $I=J$, $\alpha = \beta$ and the pairings are such that $p \ni \{a,b\} = \{a,|I|+a\}$ for $a \in \{1,\dots,|I|\}$, in which case
    $$\Big(\frac{\sigma_A^2}{N}\Big)^{\frac{|I|+|J|}{2}} \omega(p,I,J,N,\alpha,\beta) =  \Big(\frac{\sigma_A^2}{N}\Big)^{\frac{|I|+|J|}{2}} N^{\frac{|I|+|J|}{2}} = \sigma_A^{|I|+|J|} $$
    
    This means that 
    \[
        \E\Big[[A_I^T A_J]_{\alpha}^{\beta} \Big] = \Big(\frac{\sigma_A^2}{N}\Big)^{^{\frac{|I|+|J|}{2}}} \hspace{3pt} \Xi(I,J,N,\alpha,\beta) = \sigma_A^{|I|+|J|}  \big[ \delta_{\alpha}^{\beta}\delta_{I}^{J} + \frac{1}{N}\psi(I,J,N,\alpha,\beta) \big]
    \]
    
    where $0 < \psi(I,J,N,\alpha,\beta) \leq (|I|+|J|)!!$ i.e.  $\psi(I,J,N)$ is a positive constant bounded above by the maximal number of pairings (which occur only when $I$ and $J$ are made up of the same one index). Notice how this bound depends only on $|I|$ and $|J|$ and not on $N$!
    
    \vspace{10pt}
    Finally, if $S^N_0$ is normally distributed then

    \begin{equation*}
        \begin{gathered}
            \E\Big[ (S^N_0)^T A^T_I A_J S^N_0 \Big] 
            =
            \sum_{n=1}^N\sum_{m=1}^N \E\Big[ [S^N_0]_n [S^N_0]_m \Big] \E\Big[[A^T_I A_J]^m_n \Big]
            \\=
            \sum_{n=1}^N {\sigma_{S_0}^2} \E\Big[[A^T_I A_J]^n_n \Big] 
            = N \sigma_{S_0}^2\sigma_A^{|I|+|J|} \big[ \delta_I^J + \frac{1}{N}\psi(I,J,N) \big] = N \sigma_{S_0}^2\sigma_A^{|I|+|J|} \big[ \delta_I^J + O(\frac{1}{N}) \big]
        \end{gathered}
    \end{equation*}
    
    with $0 \leq \psi(I,J,N) := \frac{1}{N} \sum_{n=1}^N \psi(I,J,N,n,n) \leq (|I|+|J|)!!$.
    
    \vspace{20pt}
    
    Let us now look at the case with $\sigma_b > 0$.  
    
    Let $V_j(x) := A_jx + b_j$ 
    \vspace{3pt}
    Take $f(x) = \langle v, x \rangle_{\R^N}$ for some $v \in \R^N$. Then 
    \[
        V_jf(x) = d(\langle v, \cdot \rangle_{\R^N})_x [A_jx + b_j] = v^T(A_jx + b_j) = \langle A_j^Tv, x \rangle_{\R^N}
        + \langle v, b_j \rangle_{\R^N}
    \]
    hence by induction 
    \[
        V_If(x) = \langle v, A_{i_k} \cdots A_{i_2} (A_{i_1} x + b_{i_1}) \rangle_{\R^N}
    \]
    now we have to study, with $\hat{I} := (i_2, \dots, i_{|I|})$, the terms
    \begin{equation*}
        \begin{gathered}
            \E\Big[ \langle A_I S^N_0 + A_{\hat{I}}b_{i_1}, A_J S^N_0 + A_{\hat{J}}b_{j_1} \rangle_{\R^N}  \Big]
            \\ =
            \E\Big[ \langle A_I S^N_0 , A_J S^N_0 \rangle_{\R^N}  \Big] 
            + 
            \E\Big[ \langle A_{\hat{I}} b_{i_{1}} , A_{\hat{J}} b_{j_{1}} \rangle_{\R^N}  \Big]
            \\ +
            \E\Big[ \langle A_{\hat{I}}b_{i_1} , A_J S^N_0 \rangle_{\R^N}  \Big]
            +
            \E\Big[ \langle A_I S^N_0 , A_{\hat{J}}b_{j_1} \rangle_{\R^N}  \Big]
            \\ = 
            \E\Big[ (S^N_0)^T A^T_I A_J S^N_0 \Big]
            + 
            \E\Big[b_{i_{1}}^T A^T_{\hat{I}} A_{\hat{J}} b_{j_{1}} \Big]
            \\ +
            \E\Big[ b^T_{i_1} A^T_{\hat{I}} A_J S^N_0 \Big]
            + 
            \E\Big[ (S^N_0)^T A^T_I A_{\hat{J}} b_{j_{1}} \Big]
            \\ =
            \E\Big[ (S^N_0)^T A^T_I A_J S^N_0 \Big]
            + 
            \E\Big[b_{i_{1}}^T A^T_{\hat{I}} A_{\hat{J}} b_{j_{1}} \Big]
        \end{gathered}
    \end{equation*}
    
    where in the last equality we have used the independence of the terms and their 0 mean.
    
    \vspace{5pt}
    We already know that the first term is
    \[
       \E\Big[ (S^N_0)^T A^T_I A_J S^N_0 \Big] = N \sigma_{S_0}^2\sigma_A^{|I|+|J|} \big[ \delta_I^J + O(\frac{1}{N}) \big]
    \]
    
    \vspace{5pt}
    Concerning the second term: using independence we readily see how we must have $i_{1} = j_{1}$, then 
    
    \begin{equation*}
        \begin{gathered}
            \E\Big[b_{i_{1}}^T A^T_{\hat{I}} A_{\hat{J}} b_{j_{1}} \Big] 
            = 
            \sigma_b^2 \delta_{i_1}^{j_1} \cdot \sum_{n=1}^N  
            \E\Big[[A^T_{\hat{I}} A_{\hat{J}}]^n_n \Big]
            \\ =
            \sigma_b^2\delta_{i_1}^{j_1} \cdot \sum_{n=1}^N  \sigma_A^{|\hat{I}|+|\hat{J}|}
            \big( 
                \delta_{\hat{I}}^{\hat{J}} + \frac{1}{N}\psi(I,J,N,n,n)
            \big) 
            =
            N \sigma_b^2 \sigma_A^{|\hat{I}|+|\hat{J}|} \big[ \delta_I^J + O(\frac{1}{N}) \big]
        \end{gathered}
    \end{equation*}

    Hence
    \[
        \E\Big[ \langle A_I S^N_0 + A_{\hat{I}}b_{i_1}, A_J S^N_0 + A_{\hat{J}}b_{j_1} \rangle_{\R^N}  \Big]
        =
        N(\sigma_{S_0}^2 + \frac{\sigma_{b}^2}{\sigma_A^2})\sigma_A^{|I|+|J|}  \big[ \delta_I^J + O(\frac{1}{N}) \big]
    \]
    
\end{proof}

\subsubsection{Convergence to the signature kernel}

Assume now that the exchange of series with limits and expectation are justified, which we will prove later, then we would like to study the variance of the \emph{expected signature kernels} around their limits.

\begin{proposition}
    The coefficients in the expansion of the variance 
    \[
        \E\Big[
        \Big(\frac{\langle S^N_s(x), S^N_t(y) \rangle_{\R^N}}{N} - 
        \sigma_{S_0}^2\langle Sig(\sigma_A x)_{0,s}, Sig(\sigma_A y)_{0,t} \rangle_{T((\R^d))}\Big)^2
        \Big]
    \]
    
    are all $O(\frac{1}{N})$ when $\sigma_b = 0$.
\end{proposition}

\begin{proof}

We have
\begin{equation*}
    \begin{gathered}
        \frac{\langle S^N_s(x), S^N_t(y) \rangle_{\R^N}}{N} - 
        \sigma_{S_0}^2\langle Sig(\sigma_A x)_{0,s}, Sig(\sigma_A y)_{0,t}\rangle_{T((\R^d))}
        \\
        = \sum_{I, J \in \mathbb{W}_d} 
        \big[ \frac{1}{N}(S^N_0)^T A^T_{I}A_{J} S^N_0 - \sigma_{S_0}^2\sigma_{A}^{|I|+|J|}\delta_I^J \big]\hspace{3pt} Sig^{I}_{0,s}(x) Sig^{J}_{0,t}(y)
    \end{gathered}
\end{equation*}
thus 
\begin{equation*}
    \begin{gathered}
        \Big(
        \frac{\langle S^N_s(x), S^N_t(y) \rangle_{\R^N}}{N} - 
        \sigma_{S_0}^2\langle Sig(\sigma_A x)_{0,s}, Sig(\sigma_A y)_{0,t}\rangle_{T((\R^d))}
        \Big)^2
        = 
        \\
        \sum_{I, J, K, L \in \mathbb{W}_d} 
        \big[ \frac{1}{N} (S^N_0)^T A^T_{I}A_{J} S^N_0 - \sigma_{S_0}^2\sigma_{A}^{|I|+|J|}\delta_I^J \big]\cdot
        \\
        \cdot
        \big[ \frac{1}{N} (S^N_0)^T A^T_{K}A_{L} S^N_0 - \sigma_{S_0}^2\sigma_{A}^{|L|+|K|}\delta_K^L \big] 
        \hspace{3pt} Sig^{I}_{0,s}(x) Sig^{J}_{0,t}(y)Sig^{K}_{0,s}(x) Sig^{K}_{0,t}(y)
    \end{gathered}
\end{equation*}

\vspace{5pt}
Let us study
\begin{equation*}
\begin{gathered}
    \E\Big[ 
    \big[ \frac{1}{N} (S^N_0)^T A^T_{I}A_{J} S^N_0 - \sigma_{S_0}^2\sigma_{A}^{|I|+|J|}\delta_I^J \big]
    \big[ \frac{1}{N} (S^N_0)^T A^T_{K}A_{L} S^N_0 - \sigma_{S_0}^2\sigma_{A}^{|L|+|K|}\delta_K^L \big] 
    \Big] 
    \\
    = \frac{1}{N^2} \E\Big[ (S^N_0)^T A^T_{I}A_{J} S^N_0  (S^N_0)^T A^T_{K}A_{L}S^N_0 \Big]\\
    - \frac{\sigma_{S_0}^2\sigma_{A}^{2|I|}}{N} \delta_I^J \E\Big[ (S^N_0)^T A^T_{K}A_{L}S^N_0 \Big]
    - \frac{\sigma_{S_0}^2\sigma_{A}^{2|L|}}{N} \delta_K^L \E\Big[ (S^N_0)^T A^T_{I}A_{J}S^N_0 \Big]
    + (\sigma_{S_0}^2\sigma_{A}^{|L|+|I|})^2\delta_I^J\delta_K^L 
    \\
    = \frac{1}{N^2}\E\Big[ (S^N_0)^T A^T_{I}A_{J} S^N_0  (S^N_0)^T A^T_{K}A_{L}S^N_0 \Big]\\
    - \frac{\sigma_{S_0}^2\sigma_{A}^{2|I|}}{N}\delta_I^J N \sigma_{S_0}^2\sigma_A^{|K|+|L|} \big[ \delta_K^L + O(\frac{1}{N}) \big]
    - \frac{\sigma_{S_0}^2\sigma_{A}^{2|L|}}{N}\delta_J^K N \sigma_{S_0}^2\sigma_A^{|I|+|J|} \big[ \delta_I^J + O(\frac{1}{N}) \big]
    + (\sigma_{S_0}^2\sigma_{A}^{|L|+|I|})^2\delta_I^J\delta_K^L\\
    = \frac{1}{N^2} \E\Big[ (S^N_0)^T A^T_{I}A_{J} S^N_0  (S^N_0)^T A^T_{K}A_{L}S^N_0 \Big] 
    - (\sigma_{S_0}^2\sigma_{A}^{|L|+|I|})^2 \delta_I^J\delta_K^L + O(\frac{1}{N})\\
\end{gathered}
\end{equation*}

from our previous results.

\vspace{5pt}
We have 
\begin{equation*}
\begin{gathered}
    (S^N_0)^T A^T_{I}A_{J} S^N_0  (S^N_0)^T A^T_{K}A_{L}S^N_0 =\\
    \sum_{\alpha = 1}^N [S^N_0]_{\alpha} [A^T_{I}A_{J} S^N_0  (S^N_0)^T A^T_{K}A_{L}S^N_0]_{\alpha} =\\
    \sum_{\alpha, \delta = 1}^N [S^N_0]_{\alpha} [A^T_{I}A_{J} S^N_0  (S^N_0)^T A^T_{K}A_{L}]_{\alpha}^{\delta}[S^N_0]_{\delta} =\\
    \sum_{\alpha, \beta, \gamma, \delta = 1}^N [S^N_0]_{\alpha} [A^T_{I}A_{J}]_{\alpha}^{\beta}[S^N_0  (S^N_0)^T]_{\beta}^{\gamma}[ A^T_{K}A_{L}]_{\gamma}^{\delta}[S^N_0]_{\delta} =\\
    \sum_{\alpha, \beta, \gamma, \delta = 1}^N [S^N_0]_{\alpha} [A^T_{I}A_{J}]_{\alpha}^{\beta}[S^N_0]_{\beta}[S^N_0]_{\gamma}[ A^T_{K}A_{L}]_{\gamma}^{\delta}[S^N_0]_{\delta} =\\
    \sum_{\alpha, \beta, \gamma, \delta = 1}^N [S^N_0]_{\alpha}[S^N_0]_{\beta}[S^N_0]_{\gamma}[S^N_0]_{\delta} [A^T_{I}A_{J}]_{\alpha}^{\beta}[A^T_{K}A_{L}]_{\gamma}^{\delta}\\
\end{gathered}
\end{equation*}
hence by independence
\begin{equation*}
\begin{gathered}
    \E\Big[(S^N_0)^T A^T_{I}A_{J} S^N_0  (S^N_0)^T A^T_{K}A_{L}S^N_0\Big] =\\
    \sum_{\alpha, \beta, \gamma, \delta = 1}^N \E\Big[[S^N_0]_{\alpha}[S^N_0]_{\beta}[S^N_0]_{\gamma}[S^N_0]_{\delta}\Big] \E\Big[[A^T_{I}A_{J}]_{\alpha}^{\beta}[A^T_{K}A_{L}]_{\gamma}^{\delta}\Big]\\
\end{gathered}
\end{equation*}

\vspace{10pt}
If $S_0^N$ is sampled from a Normal distribution as before then
\[
    \E\Big[[S^N_0]_{\alpha}[S^N_0]_{\beta}[S^N_0]_{\gamma}[S^N_0]_{\delta}\Big]
\]
is equal to $\sigma_{S_0}^4$ if $|\{\alpha, \beta, \gamma, \delta\}| = 2$, to
$3\sigma_{S_0}^4$ if $|\{\alpha, \beta, \gamma, \delta\}| = 1$ and to $0$ otherwise.

\vspace{10pt}
Remember how
\[
[A_I^T A_J]_{\alpha}^{\beta} = \sum_{\gamma = 1}^N \sum_{\bar{\delta} \in \Lambda^{N,|I|}_{\gamma,\alpha}} \sum_{\bar{\epsilon} \in \Lambda^{N,|J|}_{\gamma,\beta}} \prod_{i=1}^{|I|} \prod_{j=1}^{|J|}\hspace{3pt} [A_{I_i}]_{\delta_{i}}^{\delta_{i-1}} [A_{J_j}]_{\epsilon_{j}}^{\epsilon_{j-1}}
\]
thus 
\begin{equation*}
\begin{gathered}
    [A^T_{I}A_{J}]_{\alpha}^{\beta}
    [A^T_{K}A_{L}]_{\gamma}^{\delta}
    = \\
    \sum_{\epsilon, \phi = 1}^N 
    \sum_{\bar{\iota} \in \Lambda^{N,|I|}_{\epsilon,\alpha}} 
    \sum_{\bar{\zeta} \in \Lambda^{N,|J|}_{\epsilon,\beta}}
    \sum_{\bar{\kappa} \in \Lambda^{N,|K|}_{\phi,\gamma}}
    \sum_{\bar{\lambda} \in \Lambda^{N,|L|}_{\phi,\delta}}
    \prod_{i=1}^{|I|} \prod_{j=1}^{|J|}
    \prod_{k=1}^{|K|} \prod_{l=1}^{|L|}
    [A_{I_i}]_{\iota_{i}}^{\iota_{i-1}}
    [A_{J_j}]_{\zeta_{j}}^{\zeta_{j-1}}
    [A_{K_k}]_{\kappa_{k}}^{\kappa_{k-1}}
    [A_{L_l}]_{\lambda_{l}}^{\lambda_{l-1}}
\end{gathered}
\end{equation*}

Let us then define $T = I*J*K*L$, $\theta = (\iota*\zeta*\kappa*\lambda)$
and $\theta' = (\iota*\zeta*\kappa*\lambda)'$ just like before, setting $\mathcal{P} := P^2_{|I|+|J|+|K|+|L|}$ we get

\begin{equation*}
\begin{gathered}
    \E\Big[
    [A^T_{I}A_{J}]_{\alpha}^{\beta}[A^T_{K}A_{L}]_{\gamma}^{\delta}
    \Big]=\\
    \sum_{p \in \mathcal{P}}
    \sum_{\epsilon, \phi = 1}^N 
    \sum_{\bar{\iota} \in \Lambda^{N,|I|}_{\epsilon,\alpha}} 
    \sum_{\bar{\zeta} \in \Lambda^{N,|J|}_{\epsilon,\beta}}
    \sum_{\bar{\kappa} \in \Lambda^{N,|K|}_{\phi,\gamma}}
    \sum_{\bar{\lambda} \in \Lambda^{N,|L|}_{\phi,\delta}}
    \prod_{a,b \in p}
    \E\Big[
    [A_{T_a}]_{\theta_{a}}^{\theta'_{a}}
    [A_{T_b}]_{\theta_{b}}^{\theta'_{b}}
    \Big]\\
    = \big( \frac{\sigma_A^2}{N} \big)^\frac{|I|+|J|+|K|+|L|}{2}\sum_{p \in \mathcal{P}} \omega(p,I,J,K,L,\alpha,\beta,\gamma,\delta)
\end{gathered}
\end{equation*}

and once again we need to analyze these $\omega$s which, just as before, must satisfy the constraint
\[
    \omega(p,I,J,K,L,\alpha,\beta,\gamma,\delta) 
    \leq N^\frac{|I|+|J|+|K|+|L|}{2}
\]

Since we are interested in the behavior for 
$N \to \infty$ and $|\mathcal{P}|$ is independent from $N$ we just need to discover when the previous inequality is an equality.
This happens, as we have previously discovered, when we the pairing does not add to the possible choices of $\bar{\iota}, \bar{\zeta}, \bar{\kappa}, \bar{\lambda}$ any more constraints than the unavoidable ones \emph{i.e.} 
$\theta_{|I|} = \theta_{|I|+|J|}$, $\theta_{|I|+ |J| + |K|} = \theta_{|I| + |J| + |K| + |L|}$, $\theta'_{1} = \alpha$,
$\theta'_{|I| + 1} = \beta$, $\theta'_{|I| + |J| + 1} = \gamma $ and $\theta'_{|I| + |J| + |K| + 1} = \delta $.

\vspace{5pt}
Reasoning exactly as before this can happen if and only if 
$I = J$, $\alpha = \beta$, $K = L$, $\gamma = \delta$, 
$\theta_a = \theta_{|I| + a}$ for $a = 1, \dots, |I|$ and 
$\theta_{2|I| + a} = \theta_{2|I| + |K| + a}$ for $a = 1, \dots, |K|$.

\vspace{5pt}
This means that 
\begin{equation*}
    \begin{gathered}
        \E\Big[
    [A^T_{I}A_{J}]_{\alpha}^{\beta}[A^T_{K}A_{L}]_{\gamma}^{\delta}
    \Big]
    =
    (\sigma_A)^{|I|+|J|+|K|+|L|} \big[ \delta_I^J\delta_{\alpha}^{\beta}\delta_K^L\delta_{\gamma}^{\delta} + O(\frac{1}{N}) \big] \leq  
    \\ 
    (\sigma_A)^{|I|+|J|+|K|+|L|} \big[ \delta_I^J\delta_{\alpha}^{\beta}\delta_K^L\delta_{\gamma}^{\delta} + \frac{1}{N} \sigma(I,J,K,L) \big]
    \end{gathered}
\end{equation*} 
where $\sigma(I,J,K,L)$ is a positive constant corresponding to the maximal number of non zero pairings.

Since 
\begin{equation*}
\begin{gathered}
    \E\Big[(S^N_0)^T A^T_{I}A_{J} S^N_0  (S^N_0)^T A^T_{K}A_{L}S^N_0\Big] \\
    = \sum_{\alpha, \beta, \gamma, \delta = 1}^N \E\Big[
    [S^N_0]_{\alpha}[S^N_0]_{\beta}[S^N_0]_{\gamma}[S^N_0]_{\delta}
    \Big] 
    \E\Big[
    [A^T_{I}A_{J}]_{\alpha}^{\beta}[A^T_{K}A_{L}]_{\gamma}^{\delta}
    \Big]\\
    = \sum_{\alpha= 1}^N 3 \sigma_{S_0}^4
    \E\Big[
    [A^T_{I}A_{J}]_{\alpha}^{\alpha}[A^T_{K}A_{L}]_{\alpha}^{\alpha}
    \Big]
    + \sum_{\substack{\alpha, \beta=1 \\ \alpha \neq \beta}}^N \sigma_{S_0}^4
    \E\Big[
    [A^T_{I}A_{J}]_{\alpha}^{\alpha}[A^T_{K}A_{L}]_{\beta}^{\beta}
    \Big]\\
    + \sum_{\substack{\alpha, \beta=1 \\ \alpha \neq \beta}}^N \sigma_{S_0}^4
    \E\Big[
    [A^T_{I}A_{J}]_{\alpha}^{\beta}[A^T_{K}A_{L}]_{\alpha}^{\beta}
    \Big]
    + \sum_{\substack{\alpha, \beta=1 \\ \alpha \neq \beta}}^N \sigma_{S_0}^4
    \E\Big[
    [A^T_{I}A_{J}]_{\alpha}^{\beta}[A^T_{K}A_{L}]_{\beta}^{\alpha}
    \Big]\\
\end{gathered}
\end{equation*}

we end up with 
\begin{equation*}
\begin{gathered}
    \E\Big[(S^N_0)^T A^T_{I}A_{J} S^N_0  (S^N_0)^T A^T_{K}A_{L}S^N_0\Big] = \\
     \sigma_{S_0}^4(\sigma_A)^{|I|+|J|+|K|+|L|} \big\{ N 3  (\delta_I^J\delta_K^L + O(\frac{1}{N})) +
     \\
    + (N^2 - N)(\delta_I^J\delta_K^L +  O(\frac{1}{N}))
    + (N^2 - N)O(\frac{1}{N}) + (N^2 - N) O(\frac{1}{N}) \big\}\\
    = \sigma_{S_0}^4 (\sigma_A)^{|I|+|J|+|K|+|L|} N^2 \big[ \delta_I^J\delta_K^L +  O(\frac{1}{N}) \big]
    = (\sigma_{S_0}^2\sigma_{A}^{|L|+|I|})^2 N^2 \big[ \delta_I^J\delta_K^L +  O(\frac{1}{N}) \big]
\end{gathered}
\end{equation*}

In the end 
\begin{equation*}
    \begin{gathered}
        \E\Big[ 
        \big[ \frac{1}{N} (S^N_0)^T A^T_{I}A_{J} S^N_0 - \sigma_{S_0}^2\sigma_{A}^{|I|+|J|}\delta_I^J \big]
        \big[ \frac{1}{N} (S^N_0)^T A^T_{K}A_{L} S^N_0 - \sigma_{S_0}^2\sigma_{A}^{|L|+|K|}\delta_K^L \big] 
        \Big] =
        \\
        \frac{1}{N^2} \E\Big[ (S^N_0)^T A^T_{I}A_{J} S^N_0  (S^N_0)^T A^T_{K}A_{L}S^N_0 \Big] 
        - (\sigma_{S_0}^2\sigma_{A}^{|L|+|I|})^2 \delta_I^J\delta_K^L + O(\frac{1}{N}) =
        \\
        (\sigma_{S_0}^2\sigma_{A}^{|L|+|I|})^2 \big[ \delta_I^J\delta_K^L +  O(\frac{1}{N}) \big] - (\sigma_{S_0}^2\sigma_{A}^{|L|+|I|})^2 \delta_I^J\delta_K^L + O(\frac{1}{N}) = O(\frac{1}{N})
    \end{gathered}
\end{equation*}

\end{proof}

\begin{proposition}
        The coefficients in the expansion of the variance 
    \[
        \E\Big[
        \Big(\frac{\langle S^N_s(x), S^N_t(y) \rangle_{\R^N}}{N} - 
        \big( \sigma_{S_0}^2 + \frac{\sigma_b^2}{\sigma_A^2}\big)\langle Sig(\sigma_A x)_{0,s}, Sig(\sigma_A y)_{0,t}\rangle_{T((\R^d))} + \frac{\sigma_b^2}{\sigma_A^2}\Big)^2
        \Big]
    \]
    
    are all $O(\frac{1}{N})$ with $\sigma_b > 0$.
\end{proposition}

\begin{proof}
Now  
\begin{equation*}
\begin{gathered}
    \langle S^N_s(x), S^N_t(y) \rangle_{\R^N} 
    = \\
    \sum_{I, J \in \mathbb{W}_d} 
        \big[ 
        (S^N_0)^T A^T_{I}A_{J} S^N_0
        + (b_{I_{1}})^T A^T_{\hat{I}}A_{J} S^N_0
        + (S^N_0)^T A^T_{{I}}A_{\hat{J}} b_{J_{1}}
        + (b_{I_{1}})^T A^T_{\hat{I}}A_{\hat{J}} b_{J_{1}}
        \big] \cdot
        \\
        \cdot
        Sig^{I}_{0,s}(x) Sig^{J}_{0,t}(y)
\end{gathered}
\end{equation*}

\vspace{10pt}
Thus to study 
\[
        \E\Big[
        \Big(\frac{\langle S^N_s(x), S^N_t(y) \rangle_{\R^N}}{N} - 
        \big( \sigma_{S_0}^2 + \frac{\sigma_b^2}{\sigma_A^2}\big)\langle 
        Sig(\sigma_A x)_{0,s}, Sig(\sigma_A y)_{0,t}\rangle_{T((\R^d))} + 
        \frac{\sigma_b^2}{\sigma_A^2}\Big)^2
        \Big]
    \]
we need to study the terms 
\begin{equation*}
\begin{gathered}
        \frac{1}{N^2} \E\Big[
        \big[ 
        (S^N_0)^T A^T_{I}A_{J} S^N_0
        + (b_{I_{1}})^T A^T_{\hat{I}}A_{J} S^N_0
        + (S^N_0)^T A^T_{{I}}A_{\hat{J}} b_{J_{1}}
        + (b_{I_{1}})^T A^T_{\hat{I}}A_{\hat{J}} b_{J_{1}}
        - N\big( \sigma_{S_0}^2 + \frac{\sigma_b^2}{\sigma_A^2}\big)\sigma_A^{|I|+|J|}\delta_I^J
        \big] \cdot\\
        \cdot \big[ 
        (S^N_0)^T A^T_{K}A_{L} S^N_0
        + (b_{K_{1}})^T A^T_{\hat{K}}A_{L} S^N_0
        + (S^N_0)^T A^T_{{K}}A_{\hat{L}} b_{L_{1}}
        + (b_{K_{1}})^T A^T_{\hat{K}}A_{\hat{L}} b_{L_{1}}
        - N\big( \sigma_{S_0}^2 + \frac{\sigma_b^2}{\sigma_A^2}\big)\sigma_A^{|K|+|L|}\delta_K^L
        \big]
        \Big]
\end{gathered}
\end{equation*}

To ease the notation let us write $i, j, k$ and $l$ instead of, respectively, $I_{1}, J_{1}, K_{1}$ and $L_{1}$.

\vspace{10pt}
We have 
\begin{equation*}
\begin{gathered}
        \E\Big\{
        \big[ 
        (S^N_0)^T A^T_{I}A_{J} S^N_0
        \big]\cdot
        \\
        \cdot 
        \big[ 
        (S^N_0)^T A^T_{K}A_{L} S^N_0
        + (b_{k})^T A^T_{\hat{K}}A_{L} S^N_0
        + (S^N_0)^T A^T_{{K}}A_{\hat{L}} b_{l}
        + (b_{k})^T A^T_{\hat{K}}A_{\hat{L}} b_{l}
        - N \big( \sigma_{S_0}^2 + \frac{\sigma_b^2}{\sigma_A^2}\big)\sigma_A^{|K|+|L|} \delta_K^L
        \big]
        \Big\}
        \\
        = 
        \E\Big[
        (S^N_0)^T A^T_{I}A_{J} S^N_0
        (S^N_0)^T A^T_{K}A_{L} S^N_0
        \Big]\\
        +\E\Big[
        (S^N_0)^T A^T_{I}A_{J} S^N_0
        (b_{k})^T A^T_{\hat{K}}A_{L} S^N_0
        \Big]\\
        +\E\Big[
        (S^N_0)^T A^T_{I}A_{J} S^N_0
        (S^N_0)^T A^T_{{K}}A_{\hat{L}} b_{l}
        \Big]\\
        +\E\Big[
        (S^N_0)^T A^T_{I}A_{J} S^N_0
        (b_{k})^T A^T_{\hat{K}}A_{\hat{L}} b_{l}
        \Big]\\
         - N \big( \sigma_{S_0}^2 + \frac{\sigma_b^2}{\sigma_A^2}\big)\sigma_A^{|K|+|L|}\delta_K^L\E\Big[
        (S^N_0)^T A^T_{I}A_{J} S^N_0
        \Big]
\end{gathered}
\end{equation*}
which by previous results is equal to 
\begin{equation*}
\begin{gathered}
        (\sigma_{S_0}^2\sigma_{A}^{|L|+|I|})^2 N^2 \big[ \delta_I^J\delta_K^L +  O(\frac{1}{N}) \big]\\
        +\E\Big[
        (S^N_0)^T A^T_{I}A_{J} S^N_0
        (b_{k})^T A^T_{\hat{K}}A_{L} S^N_0
        \Big]\\
        +\E\Big[
        (S^N_0)^T A^T_{I}A_{J} S^N_0
        (S^N_0)^T A^T_{{K}}A_{\hat{L}} b_{l}
        \Big]\\
        +\E\Big[
        (S^N_0)^T A^T_{I}A_{J} S^N_0
        (b_{k})^T A^T_{\hat{K}}A_{\hat{L}} b_{l}
        \Big]\\
         - \big( \sigma_{S_0}^2 + \frac{\sigma_b^2}{\sigma_A^2}\big)\sigma_A^{|K|+|L|}\delta_K^L \sigma_{S_0}^2\sigma_A^{|I|+|J|} 
         N^2 \big[ \delta_I^J + O(\frac{1}{N}) \big]
\end{gathered}
\end{equation*}

Note then how 
\begin{equation*}
\begin{gathered}
        \E\Big[
        (S^N_0)^T 
        A^T_{I}A_{J} 
        S^N_0
        (b_{k})^T 
        A^T_{\hat{K}}A_{L} 
        S^N_0
        \Big]
        \\
        = \sum_{\alpha, \beta, \gamma, \delta = 1}^N 
        \E\Big[
        [S^N_0]_{\alpha}
        [S^N_0]_{\beta}
        [b_{k}]_{\gamma}
        [S^N_0]_{\delta}
        \Big] 
        \E\Big[
        [A^T_{I}A_{J}]_{\alpha}^{\beta}
        [A^T_{K}A_{L}]_{\gamma}^{\delta}
        \Big]
        \\
        = \sum_{\alpha, \beta, \gamma, \delta = 1}^N 
        \E\Big[
        [S^N_0]_{\alpha}
        [S^N_0]_{\beta}
        [S^N_0]_{\delta}
        \Big]
        \E\Big[
        [b_{k}]_{\gamma}
        \Big]
        \E\Big[
        [A^T_{I}A_{J}]_{\alpha}^{\beta}
        [A^T_{K}A_{L}]_{\gamma}^{\delta}
        \Big]
        = 0
        \\
\end{gathered}
\end{equation*}
given that 
$\E\Big[[b_{k}]_{\gamma}\Big] = 0$.
Analogously we have
\begin{equation*}
\begin{gathered}
        \E\Big[
        (S^N_0)^T A^T_{I}A_{J} S^N_0
        (S^N_0)^T A^T_{{K}}A_{\hat{L}} b_{l}
        \Big]
        = 0
\end{gathered}
\end{equation*}

Finally
\begin{equation*}
\begin{gathered}
        \E\Big[
        (S^N_0)^T 
        A^T_{I}A_{J} 
        S^N_0
        (b_{k})^T 
        A^T_{\hat{K}}A_{\hat{L}} 
        b_{l}
        \Big]
        \\
        = \sum_{\alpha, \beta, \gamma, \delta = 1}^N 
        \E\Big[
        [S^N_0]_{\alpha}
        [S^N_0]_{\beta}
        [b_{k}]_{\gamma}
        [b_{l}]_{\delta}
        \Big] 
        \E\Big[
        [A^T_{I}A_{J}]_{\alpha}^{\beta}
        [A^T_{\hat{K}}A_{\hat{L}}]_{\gamma}^{\delta}
        \Big]
        \\
        = \sum_{\alpha, \beta, \gamma, \delta = 1}^N 
        \E\Big[
        [S^N_0]_{\alpha}
        [S^N_0]_{\beta}
        \Big]
        \E\Big[
        [b_{k}]_{\gamma}
        [b_{l}]_{\delta}
        \Big]
        \E\Big[
        [A^T_{I}A_{J}]_{\alpha}^{\beta}
        [A^T_{\hat{K}}A_{\hat{L}}]_{\gamma}^{\delta}
        \Big]
        \\
        = \sigma_{S_0}^2 \sigma_{b}^2 \sigma_A^{|I|+|J|+|\hat{K}|+|\hat{L}| } \sum_{\alpha, \beta, \gamma, \delta = 1}^N 
        {\delta_{\alpha}^{\beta}}
        {\delta_{\gamma}^{\delta}\delta_{k}^{l}}
        (\delta_{I}^{J}\delta_{\hat{K}}^{\hat{L}}\delta_{\alpha}^{\beta}\delta_{\gamma}^{\delta} + O(\frac{1}{N}) )
        \\ 
        = \sigma_{S_0}^2 \sigma_{b}^2 \sigma_A^{|I|+|J|+|\hat{K}|+|\hat{L}| }
        \sum_{\alpha, \beta, \gamma, \delta = 1}^N 
        \delta_{\alpha}^{\beta}
        \delta_{\gamma}^{\delta}
        \big(
        \delta_{k}^{l}
        \delta_{I}^{J}
        \delta_{\hat{K}}^{\hat{L}}
        + O(\frac{1}{N})
        \big)
        \\
        = \sigma_{S_0}^2 \sigma_{b}^2 \sigma_A^{|I|+|J|+|\hat{K}|+|\hat{L}| } 
        \sum_{\alpha, \beta, \gamma, \delta = 1}^N 
        \delta_{\alpha}^{\beta}
        \delta_{\gamma}^{\delta}
        \big(
        \delta_{I}^{J}
        \delta_{K}^{L}
        + O(\frac{1}{N})
        \big)
        \\
        = \sigma_{S_0}^2 \sigma_{b}^2 \sigma_A^{|I|+|J|+|\hat{K}|+|\hat{L}| }
        \sum_{\alpha, \gamma = 1}^N 
        \big(
        \delta_{I}^{J}
        \delta_{K}^{L}
        + O(\frac{1}{N})
        \big)
        = \sigma_{S_0}^2 \sigma_{b}^2 \sigma_A^{|I|+|J|+|\hat{K}|+|\hat{L}| } N^2 
        \big[ \delta_{I}^{J} \delta_{K}^{L} + O(\frac{1}{N}) \big]
\end{gathered}
\end{equation*}
where once again we use the fact that all the $O(\frac{1}{N})$ are uniformly bounded above by some $O(\frac{1}{N})$.

Putting everything together 
\begin{equation*}
\begin{gathered}
        \E\Big\{
        \big[ 
        (S^N_0)^T A^T_{I}A_{J} S^N_0
        \big]
        \cdot
        \\
        \cdot
        \big[ 
        (S^N_0)^T A^T_{K}A_{L} S^N_0
        + (b_{k})^T A^T_{\hat{K}}A_{L} S^N_0
        + (S^N_0)^T A^T_{{K}}A_{\hat{L}} b_{l}
        + (b_{k})^T A^T_{\hat{K}}A_{\hat{L}} b_{l}
        - N \big( \sigma_{S_0}^2 + \frac{\sigma_b^2}{\sigma_A^2}\big)\sigma_A^{|K|+|L|} \delta_K^L
        \big]
        \Big\}\\
        = 
        (\sigma_{S_0}^2\sigma_{A}^{|L|+|I|})^2 N^2 \big[ \delta_I^J\delta_K^L +  O(\frac{1}{N}) \big]
        + 0
        + 0
        +\sigma_{S_0}^2 \sigma_{b}^2 \sigma_A^{|I|+|J|+|\hat{K}|+|\hat{L}| } N^2 
        \big[ \delta_{I}^{J} \delta_{K}^{L} + O(\frac{1}{N}) \big] \\
         -  \big( \sigma_{S_0}^2 + \frac{\sigma_b^2}{\sigma_A^2}\big)\sigma_A^{|K|+|L|}\delta_K^L \sigma_{S_0}^2\sigma_A^{|I|+|J|} 
         N^2 \big[ \delta_I^J + O(\frac{1}{N}) \big] = N^2{O(\frac{1}{N})}
\end{gathered}
\end{equation*}

\vspace{10pt}
With analogous arguments we obtain 
\begin{equation*}
\begin{gathered}
        \E\Big[
        \big[ 
        (b_i)^T A^T_{\hat{I}}A_{J} S^N_0
        \big]
        \cdot 
        \\
        \cdot
        \big[ 
        (S^N_0)^T A^T_{K}A_{L} S^N_0
        + (b_{k})^T A^T_{\hat{K}}A_{L} S^N_0
        + (S^N_0)^T A^T_{{K}}A_{\hat{L}} b_{l}
        + (b_{k})^T A^T_{\hat{K}}A_{\hat{L}} b_{l}
        - N \big( \sigma_{S_0}^2 + \frac{\sigma_b^2}{\sigma_A^2}\big)\sigma_A^{|K|+|L|}\delta_K^L
        \big]
        \Big]\\
        = 0 
        + \E\Big[
        ((b_i)^T A^T_{\hat{I}}A_{J} S^N_0
        (b_{k})^T A^T_{\hat{K}}A_{L} S^N_0
        \Big]
        +\E\Big[
        (b_i)^T A^T_{\hat{I}}A_{J} S^N_0
        (S^N_0)^T A^T_{{K}}A_{\hat{L}} b_{l}
        \Big]
        + 0  - N \big( \sigma_{S_0}^2 + \frac{\sigma_b^2}{\sigma_A^2}\big)\sigma_A^{|K|+|L|}\delta_K^L \cdot 0
\end{gathered}
\end{equation*}
and since 
\begin{equation*}
\begin{gathered}
        \E\Big[
        (b_i)^T 
        A^T_{\hat{I}}A_{J} 
        S^N_0
        (b_{k})^T 
        A^T_{\hat{K}}A_{L} 
        S^N_0
        \Big]
        \\
        = \sum_{\alpha, \beta, \gamma, \delta = 1}^N 
        \E\Big[
        [S^N_0]_{\beta}
        [S^N_0]_{\delta}
        \Big] 
        \E\Big[
        [b_{i}]_{\alpha}
        [b_{k}]_{\gamma}
        \Big] 
        \E\Big[
        [A^T_{\hat{I}}A_{J}]_{\alpha}^{\beta}
        [A^T_{\hat{K}}A_{{L}}]_{\gamma}^{\delta}
        \Big]
        \\
        = \sigma_{S_0}^2 \sigma_b^2  \sigma_A^{|\hat{I}|+|J|+|\hat{K}|+|{L}| }
        \sum_{\alpha, \beta, \gamma, \delta = 1}^N 
        {\delta_{\beta}^{\delta}}
        {\delta_{\alpha}^{\gamma}\delta_{i}^{k}}
        (
        \delta_{\hat{I}}^{J}
        \delta_{\hat{K}}^{L}
        \delta_{\alpha}^{\beta}
        \delta_{\gamma}^{\delta}
        +
        O(\frac{1}{N})
        )
        \\
        =  \sigma_{S_0}^2 \sigma_b^2  \sigma_A^{|\hat{I}|+|J|+|\hat{K}|+|{L}|}
        \big[
        \sum_{\alpha, \beta, \gamma, \delta = 1}^N 
        \delta_{\beta}^{\delta}
        \delta_{\alpha}^{\gamma}
        \delta_{i}^{k}
        \delta_{\hat{I}}^{J}
        \delta_{\hat{K}}^{L}
        \delta_{\alpha}^{\beta}
        \delta_{\gamma}^{\delta}
        + 
        \sum_{\alpha, \beta, \gamma, \delta = 1}^N 
        \delta_{\beta}^{\delta}
        \delta_{\alpha}^{\gamma}\delta_{i}^{k}
        O(\frac{1}{N})
        \big]
        \\
        = \sigma_{S_0}^2 \sigma_b^2  \sigma_A^{|\hat{I}|+|J|+|\hat{K}|+|{L}|}
        \big[
        \sum_{\alpha = 1}^N 
        \delta_{i}^{k}
        \delta_{\hat{I}}^{J}
        \delta_{\hat{K}}^{L}
        + \sum_{\alpha, \beta = 1}^N 
        \delta_{i}^{k}
        O(\frac{1}{N}) 
        \big]
        \\ 
       \sigma_{S_0}^2 \sigma_b^2  \sigma_A^{|\hat{I}|+|J|+|\hat{K}|+|{L}|}
        \big[
         N 
        \delta_{i}^{k}
        \delta_{\hat{I}}^{J}
        \delta_{\hat{K}}^{L}
        +
        N^2 
        \delta_{i}^{k}
        O(\frac{1}{N})
        \big]
        =
        N^2O(\frac{1}{N})
\end{gathered}
\end{equation*}

and similarly 
\[
    \E\Big[
    (b_i)^T A^T_{\hat{I}}A_{J} S^N_0
    (S^N_0)^T A^T_{{K}}A_{\hat{L}} b_{l}
    \Big]
    = N^2O(\frac{1}{N})
\]
we finally obtain 
\begin{equation*}
\begin{gathered}
        \E\Big\{
        \big[ 
        (b_i)^T A^T_{\hat{I}}A_{J} S^N_0
        \big]
        \cdot
        \\
        \cdot 
        \big[ 
        (S^N_0)^T A^T_{K}A_{L} S^N_0
        + (b_{k})^T A^T_{\hat{K}}A_{L} S^N_0
        + (S^N_0)^T A^T_{{K}}A_{\hat{L}} b_{l}
        + (b_{k})^T A^T_{\hat{K}}A_{\hat{L}} b_{l}
        - N \big( \sigma_{S_0}^2 + \frac{\sigma_b^2}{\sigma_A^2}\big)\sigma_A^{|K|+|L|}\delta_K^L
        \big]
        \Big\}\\
        = N^2O(\frac{1}{N})
\end{gathered}
\end{equation*}

Using the same arguments developed up to here all other terms in the product end up as being $N^2O(\frac{1}{N})$ thus dividing finally by $N^2$ we have the thesis.
\end{proof}

We will just do the case $(\sigma_{S_0},\sigma_A,\sigma_B) = (1,1,0)$.
The arguments for the general case are the same.

\begin{proposition}
    \begin{equation*}
        \begin{gathered}
            \lim_{N \to \infty} 
            \sum_{I \in \mathbb{W}_d}\sum_{J \in \mathbb{W}_d}
            \frac{1}{N}\E[(S_0^N)^T A^T_I A_J S^N_0]
            Sig^{I}_{0,s}(x) Sig^{J}_{0,t}(y)
            \\ =
            \sum_{I \in \mathbb{W}_d}\sum_{J \in \mathbb{W}_d}
            \lim_{N \to \infty} 
            \frac{1}{N}\E[(S_0^N)^T A^T_I A_J S^N_0]
            Sig^{I}_{0,s}(x) Sig^{J}_{0,t}(y)
        \end{gathered}
    \end{equation*}
\end{proposition}

\begin{proof}
First of all, to be thoroughly rigorous, we need to define a probability space over which we take all the expectations, everything is numerable thus there is no issues with this.

\vspace{5pt}
To justify the exchange of sum and limit we want to use Lebesgue dominated convergence, we thus need to bound the 
$$|\frac{1}{N}\E[(S_0^N)^T A^T_I A_J S^N_0]Sig^{I}_{0,s}(x) Sig^{J}_{0,t}(y)|$$
from above.

We know, from previous considerations, that 
\[
    \frac{1}{N}\E[(S_0^N)^T A^T_I A_J S^N_0] =
    \sum_{n=1}^N \frac{1}{N} \E\Big[[A^T_I A_J]^n_n \Big]
\]

and that
\[
    \E\Big[[A^T_I A_J]^n_n \Big]
    =
    \Big(\frac{1}{N}\Big)^{\frac{|I|+|J|}{2}} \sum_{p \in P^2_{|I|+|J|}} \omega(p,I,J,N,n,n) 
\]
thus, since $\omega(p,I,J,N,\alpha,\beta)  \leq N^{\frac{|I|+|J|}{2}}$, we obtain
\[
    \E\Big[[A^T_I A_J]^n_n \Big] \leq |P^2_{|I|+|J|}| \leq (|I|+|J|)!! 
\]

We have also seen that 
$\E\Big[[A^T_I A_J]^n_n \Big] = 0$ if $|I|+|J|$ is odd, thus we can consider $|I|+|J|$ to be even; then, writing $i,j$ instead of $|I|,|J|$ for ease of reading, we get
\[
    (|I|+|J|)!! = \bigl( 2 (\frac{i+j}{2}) \bigr) !!
    = 2^{\frac{i+j}{2}} \bigl( \frac{i+j}{2} \bigr)!
\]

Putting everything together we have found 
\[
    |\frac{1}{N}\E[(S_0^N)^T A^T_I A_J S^N_0]| \leq 2^{\frac{i+j}{2}} \bigl( \frac{i+j}{2} \bigr)!
\]
and 
\[
    2 \nmid i+j \implies \frac{1}{N}\E[(S_0^N)^T A^T_I A_J S^N_0] = 0
\]

\vspace{5pt}
Finally remember how, by factorial decay,
\[
    |Sig^{I}_{0,s}(x)| \leq \frac{\norm{x}^{|I|}_{1-var}}{|I|!}
\]

\vspace{5pt}
We have now to prove, by \cite{tao2016analysis}[8.2.1 and 8.2.2], that 
\begin{align*}
    &\sum_{I \in \mathbb{W}_d}\sum_{J \in \mathbb{W}_d}
     2^{\frac{|I|+|J|}{2}} \bigl( \frac{|I|+|J|}{2} \bigr)! 
     \hspace{5pt} 
     \mathbb{I}_{2 | |I|+|J|} 
     \hspace{5pt}
     \frac{\norm{x}^{|I|}}{|I|!}\frac{\norm{y}^{|J|}}{|J|!}
    \\ =
    &\sum_{i \in \N}\sum_{j \in \N} 
    d^i d^j
     2^{\frac{i+j}{2}} \bigl( \frac{i+j}{2} \bigr)! 
     \hspace{5pt} 
     \mathbb{I}_{2 | i+j} 
     \hspace{5pt}
     \frac{\norm{x}^i}{i!}\frac{\norm{y}^j}{j!}
    <
    \infty
\end{align*}

Once again by \cite{tao2016analysis}[8.2.1 and 8.2.2] we have to find a bijection 
$\phi : \N  \to \N \times \N$ such that, if 
\[
    f(i,j) = 
    d^i d^j
    2^{\frac{i+j}{2}} \bigl( \frac{i+j}{2} \bigr)! 
    \hspace{5pt} 
    \mathbb{I}_{2 | i+j} 
    \hspace{5pt}
    \frac{\norm{x}^i}{i!}\frac{\norm{y}^j}{j!}
\]
then 
\[
    \sum_{k \in \N} f(\phi(k)) < \infty
\]

\vspace{5pt}
As a first step assume $2 | i+j$ and, writing 
$i \wedge j := \min\{i,j\}$ and $i \vee j := \max\{i,j\}$, note how 
\begin{equation*}
    \resizebox{\hsize}{!}{$
    \begin{gathered}
        f(i,j) = f(i \wedge j , i \vee j)
        \\ \leq
        \frac{
        (d\norm{x})^i(d\norm{y})^j 2^{\frac{i+j}{2}} \bigl( \frac{i+j}{2} \bigr)!
        }{
        (i \wedge j)! (i \vee j)!
        }
        =
        \frac{
        (d\norm{x})^i(d\norm{y})^j 2^{\frac{i+j}{2}} \bigl( \frac{i+j}{2} \bigr) \cdots (i \wedge j +1)
        }{
        (i \vee j)!
        }
        =
        \frac{
        (d\norm{x})^i(d\norm{y})^j 2^{\frac{i+j}{2}} 
        }{
        (i \vee j) \cdots \bigl( \frac{i+j}{2}  + 1\bigr) \cdot (i \wedge j)!
        }
        \\ \leq 
        \frac{
        d^{i+j}\norm{x}^i\norm{y}^j 2^{\frac{i+j}{2}} 
        }{
        \bigl( \frac{i+j}{2}\bigr)! 
        }
        \leq 
        \frac{
        d^{i+j}\bigr[(1+\norm{x})(1+\norm{y})\bigr]^{i+j} 2^{\frac{i+j}{2}} 
        }{
        \bigl( \frac{i+j}{2}\bigr)! 
        }
    \end{gathered}
    $}
\end{equation*}

Consider then $\phi$ as the inverse of the map $(i,j) \mapsto \frac{1}{2}(i+j)(i+j+1) + j$ 
\emph{i.e.} $\phi$ is the map enumerating pairs $(i,j)$ starting from $(0,0)$ and proceeding
with diagonal motions of the form 
$$(m,0) \to (m-1,1) \to \dots \to (0,m)$$
and then $(0,m) \to (m+1,0)$.
Notice that such diagonal strides have length $m+1$ and comprise all couples
$(i,j)$ such that $i+j = m$.

\vspace*{5pt}
Since there are exactly $2k + 1$ choices of $(i,j)$ such that $\frac{i+j}{2} = k$,
corresponding to the couples $(i,2k - i)$ for $i = 0, \dots, 2k$, using the above $\phi$
it suffices to prove
\[
    \sum_{k \in \N} (2k + 1)
    \frac{\bigr[d(1+\norm{x})(1+\norm{y})\bigr]^{2k} 2^{k} }{k}
    =
    \sum_{k \in \N} (2k + 1)
    \frac{\bigr[2d^2(1+\norm{x})^2(1+\norm{y})^2\bigr]^k}{k!} 
    <
    \infty
\]

Since 
\begin{equation*}
    \begin{gathered}
        \sum_{k \in \N} (2k + 1)
        \frac{\bigr[2d^2(1+\norm{x})^2(1+\norm{y})^2\bigr]^k}{k!} 
        \\\leq
        \sum_{k \in \N} 2^{k+1}
        \frac{\bigr[2d^2(1+\norm{x})^2(1+\norm{y})^2\bigr]^k}{k!} 
        =
        2e^{\bigl[2d(1+\norm{x})(1+\norm{y})\bigr]^2} 
        <
        \infty
    \end{gathered}
\end{equation*}

we are done.

\end{proof}

\begin{proposition}
    The exchange of sums and expectations has always been justified.
\end{proposition}

\begin{proof}
    For the case with just two indices $I,J$ we need to use Fubini-Tonelli and prove that 
    \begin{equation*}
        \begin{gathered}
            \E \Biggl[ 
            \sum_{I \in \mathbb{W}_d}\sum_{J \in \mathbb{W}_d}
            \frac{1}{N}|(S_0^N)^T A^T_I A_J S^N_0
            Sig^{I}_{0,s}(x) Sig^{J}_{0,t}(y)|
            \Biggr]
            \\=
            \sum_{I \in \mathbb{W}_d}\sum_{J \in \mathbb{W}_d}
            \E[ \frac{1}{N}|(S_0^N)^T A^T_I A_J S^N_0 |]
            |Sig^{I}_{0,s}(x) Sig^{J}_{0,t}(y)|
            <
            \infty  
        \end{gathered}  
    \end{equation*}

    \vspace{5pt}
    We have 
    \begin{equation*}
        \begin{gathered}
            \frac{1}{N} \E[ |(S_0^N)^T A^T_I A_J S^N_0 |] 
            \leq 
            \sqrt{\frac{1}{N^2}\E[ ((S_0^N)^T A^T_I A_J S^N_0)^2] }
        \end{gathered}
    \end{equation*}

    Remember how 
    \begin{equation*}
        \begin{gathered}
            \frac{1}{N^2}\E[ (S_0^N)^T A^T_I A_J S^N_0 (S_0^N)^T A^T_I A_J S^N_0] \\
            = \sum_{\alpha= 1}^N \frac{3}{N^2}
            \E\Big[
            [A^T_{I}A_{J}]_{\alpha}^{\alpha}[A^T_{I}A_{J}]_{\alpha}^{\alpha}
            \Big]
            + \sum_{\substack{\alpha, \beta=1 \\ \alpha \neq \beta}}^N \frac{1}{N^2}
            \E\Big[
            [A^T_{I}A_{J}]_{\alpha}^{\alpha}[A^T_{I}A_{J}]_{\beta}^{\beta}
            \Big]\\
            + \sum_{\substack{\alpha, \beta=1 \\ \alpha \neq \beta}}^N \frac{1}{N^2}
            \E\Big[
            [A^T_{I}A_{J}]_{\alpha}^{\beta}[A^T_{I}A_{J}]_{\alpha}^{\beta}
            \Big]
            + \sum_{\substack{\alpha, \beta=1 \\ \alpha \neq \beta}}^N \frac{1}{N^2}
            \E\Big[
            [A^T_{I}A_{J}]_{\alpha}^{\beta}[A^T_{I}A_{J}]_{\beta}^{\alpha}
            \Big]\\
        \end{gathered}
    \end{equation*}
    
    and how 
    \begin{equation*}
        \begin{gathered}
            \E\Big[
            [A^T_{I}A_{J}]_{\alpha}^{\beta}[A^T_{K}A_{L}]_{\gamma}^{\delta}
            \Big]=\\
            = \frac{1}{N^\frac{|I|+|J|+|K|+|L|}{2}}\sum_{p \in \mathcal{P}} \omega(p,I,J,K,L,\alpha,\beta,\gamma,\delta)
        \end{gathered}
    \end{equation*}
    where
    \[
        \omega(p,I,J,K,L,\alpha,\beta,\gamma,\delta) 
        \leq N^\frac{|I|+|J|+|K|+|L|}{2}
    \]
    thus 
    \begin{equation*}
        \begin{gathered}
            \E\Big[
            [A^T_{I}A_{J}]_{\alpha}^{\beta}[A^T_{K}A_{L}]_{\gamma}^{\delta}
            \Big]
            \leq
            |P^2_{|I|+|J|+|K|+|L|}| 
            \leq
            (|I|+|J|+|K|+|L|)!!
        \end{gathered}
    \end{equation*}
    
    \vspace{5pt}
    In our case $K = I$ and $J = L$, hence 
    \[
    (|I|+|J|+|K|+|L|)!! = (2|I|+2|J|)!! \leq 2^{|I|+|J|}(|I|+|J|)!
    \]

    \vspace{5pt}
    Putting all together
    \begin{equation*}
        \begin{gathered}
            \frac{1}{N} \E[ |(S_0^N)^T A^T_I A_J S^N_0 |] 
            \leq 
            \sqrt{\frac{1}{N^2}\E[ ((S_0^N)^T A^T_I A_J S^N_0)^2] }
            \\ \leq
            \sqrt{6 \cdot 2^{|I|+|J|}(|I|+|J|)!}
            \leq
            \sqrt{6} \cdot  2^{\frac{|I|+|J|}{2}}(|I|+|J|)!!
        \end{gathered}
    \end{equation*}
       
     where we have used
     \[
        (|I|+|J|)! = (|I|+|J|)!! (|I|+|J| - 1)!! \leq [(|I|+|J|)!!]^2
     \]
     
     \vspace{5pt}
     Finally
     
     \begin{equation*}
        \begin{gathered}
            \sum_{I \in \mathbb{W}_d}\sum_{J \in \mathbb{W}_d}
            \frac{1}{N} \E[ |(S_0^N)^T A^T_I A_J S^N_0 |]
            |Sig^{I}_{0,s}(x) Sig^{J}_{0,t}(y)|
            \\ \leq
            \sum_{I \in \mathbb{W}_d}\sum_{J \in \mathbb{W}_d}
             \sqrt{6} \cdot  2^{\frac{|I|+|J|}{2}}(|I|+|J|)!! 
             \frac{\norm{x}^{|I|}}{|I|!}\frac{\norm{y}^{|J|}}{|J|!}
        \end{gathered}  
    \end{equation*}
    
    which is proved to be $<\infty$ exactly as in the previous proof, this time taking care to consider also the case $|I|+|J|$ \emph{not} even.

    \vspace{10pt}
    The case with 4 words, \emph{i.e.} the variance case, goes similarly. 
\end{proof}

Putting all of this together we have finally proved the theorem:

\begin{theorem}
    Consider randomized Signatures of the type 
\[
   S^{N}_t(x) = S_0 + \sum_{k=1}^d \int_0^t \big( A_k S^{N}_{\tau}(x)  + b_k \big) dx_{\tau}^k
\]
where $x \in \bX$ and 
\[
    [A_{k}]_{\alpha}^{\beta} \sim \mathcal{N}(0,\frac{\sigma_{A}^2}{N}) 
    \hspace{15pt}
    [Y_0]_{\alpha}\sim \mathcal{N}(0,\sigma_{S_0}^2)
    \hspace{15pt}
    [b_{k}]_{\alpha} \sim \mathcal{N}(0,\sigma_b^2)
\]

Then
    \[
        \lim_{N \to \infty} \E\Big[
        \frac{1}{N}\langle S^N_s(x), S^N_t(y) \rangle_{\R^N}\Big] = 
        \big( \sigma_{S_0}^2 + \frac{\sigma_b^2}{\sigma_A^2}\big)   k_{sig}^{\sigma_A x,\sigma_A y}({s, t}) - \frac{\sigma_b^2}{\sigma_A^2}
    \]
    and the variance around the limit is of order $O(\frac{1}{N})$.
\end{theorem}

\subsubsection{Convergence to a Gaussian process}

Fix $\X := \{x_1, \dots, x_m\} \subseteq \bX$.
Let moreover $\phi \in \R^N$ be sampled from $\mathcal{N}(0,\frac{1}{N})$. 
Define the vectors
\[
    \Phi^N_{\X} := [ \sprod{\phi}{S^N_1(x_j)}_{\R^N} ]_{j=1,\dots,m} 
\]

\begin{proposition}
    $\Phi^N_{\X}$ converge in distribution to a $\mathcal{N}(0,\K_{id}^{\X})$.
\end{proposition}

\begin{proof}
    By \emph{Lévy's continuity theorem} it suffices to study the limiting behavior of the characteristic functions
    
    \begin{align*}
            & \varphi_N(u) 
            := 
            \E[\exp{\big\{ i u^T \Phi^N_{\X} \big\} }] 
            \\
            = &
            \E[\exp{\big\{ i \sum_{j=1}^m u_j [\Phi^N_{\X}]_j \big\} }]
            = 
            \E[\exp{\big\{ i \sum_{j=1}^m u_j \sprod{\phi}{S^N_1(x_j)} \big\} }]
            \\ = &
            \E[\exp{\big\{ i  \sprod{\sum_{j=1}^m u_j S^N_1(x_j)}{\phi} \big\} }]
            \\ = &
            \E[\E[\exp{\big\{ i  \sprod{\sum_{j=1}^m u_j S^N_1(x_j)}{\phi} \big\} }| A_1, b_1, \dots, A_d, b_d, S_0]]
    \end{align*}
    
    But if we know $S_0$ and the $A_k, b_k$ the randomized signatures are deterministic objects, and $\phi$ is normally distributed; thus
    
    \begin{align*}
            & \varphi_N(u) 
            =
            \E[\E[\exp{\big\{ i  \sprod{\sum_{j=1}^m u_j S^N_1(x_j)}{\phi} \big\} }| A_1, b_1, \dots, A_d, b_d, S_0]]
            \\ = &
            \E[\exp{\big\{ 
                -\frac{1}{2N} \sprod{\sum_{j=1}^m u_j S^N_1(x_j)}{\sum_{j=1}^m u_j S^N_1(x_j)}
            \big\}}]
            \\ = &
            \E[\exp{\big\{ 
                -\frac{1}{2} \sum_{i,j=1}^m u_i u_j \frac{1}{N}\sprod{ S^N_1(x_i)}{S^N_1(x_j)}
            \big\}}]
            \\ = &
            E[\exp{\big\{ 
                -\frac{1}{2} \sprod{u}{ \frac{1}{N}G^N_{\X} u}
            \big\}}]
        \end{align*}
    where 
    \[
        [G^N_{\X}]_i^j := \sprod{ S^N_1(x_i)}{S^N_1(x_j)}
    \]

\vspace{5pt}
    Now, since we have proven that
    \[
         \frac{1}{N}\sprod{ S^N_1(x_i)}{S^N_1(x_j)}
        \xrightarrow[N \to \infty]{\mathbb{L}^2} [\K_{id}^{\X}]_i^j 
    \]
    we have 
    \[
        \frac{1}{N}G^N_{\X}
        \xrightarrow[N \to \infty]{\mathbb{L}^2}
        \K_{id}^{\X} := \Sigma
    \]

    \vspace{5pt}
    Thus, since $\mathbb{L}^2$ convergence implies convergence in distribution, by the classical \emph{Portmanteau theorem} we must have 
    \[
        \E[h(\frac{1}{N}G^N_{\X})] 
        \xrightarrow[N \to \infty]{}  
        \E[h(\Sigma)] = h(\Sigma)
    \]
    for every continuous and bounded $h : \R^{m \times m} \to \R$

    \vspace{5pt}
    Fix $ u \in \R^m$ and consider 
    \[
        f_u : \R^{m \times m} \to \R 
        \hspace{3pt} \text{s.t.} \hspace{3pt}
        A \mapsto \exp\bigl\{ - \frac{1}{2} \sprod{u}{Au} \bigr\}
    \]
    
    We would like to take $h = f_u$, unfortunately if $u \neq 0$ then $f_u$ is \emph{not} bounded: taking $A = - \mathbb{I}_{m\times m}$ we have 
    \[
        f(kA) = \exp \{ \frac{k}{2} \norm{u}_2^2 \} \xrightarrow[k \to \infty]{} + \infty
    \]
    
    Fortunately we are only interested in evaluating $f_u$ on the $G^N_{\X}$ which are all positive semidefinite matrices:
    \[
        G^N_{\X} = (W^N_{\X})^T W^N_{\X}
    \]
    where $W^N_{\X} \in \R^{N \times m}$ is defined by
    \[
        [W^N_{\X}]_i^j := [S^N(x_j)]_i
    \]
    thus for any $u \in \R^m$ it holds that
    \[
        u^T \frac{1}{N} G^N_{\X} u = u^T \frac{1}{N}(W^N_{\X})^T W^N_{\X} u
        = 
        (\frac{1}{\sqrt{N}} W^N_{\X} u)^T \frac{1}{\sqrt{N}} W^N_{\X} u = \norm{\frac{1}{\sqrt{N}}W^N_{\X} u}_2^2 \geq 0
    \]

    Since the $G^N_{\X}$ are semidefinite we have 
    $$0 < f_u(G^N_{\X}) \leq 1 $$
    for any $u$, hence we always have
    \[
        f_u(G^N_{\X}) = \psi \circ f_u(G^N_{\X})
    \]
    
    where $\psi: \R \to \R$ is defined by $\psi(x) := \min\{x,2\}$.
    
    To end notice how $\psi \circ f_u : \R^{m \times m} \to \R$
    is continuous, being composition of continuous functions, and bounded, since 
    \[
     0 < \psi \circ f_u(A) \leq 2
    \]
    
    Then we have, for every $u \in \R^m$, that 
    \[
        \E[f_u(G^N_{\X})]  
        =
        \E[(\psi \circ f_u)(\frac{1}{N}G^N_{\X})] 
        \xrightarrow[N \to \infty]{}  
        (\psi \circ f_u)(\Sigma)
        =
        f_u(\Sigma)
    \]
    
    where we have used the semidefinitiveness of $\Sigma$.
    With this we finally conclude that
    \[
    \varphi_N(u) 
    \xrightarrow[N \to \infty]{} 
    \exp \bigl\{-\frac{1}{2} \sprod{u}{\Sigma u} \bigr\}
    \]

\end{proof}

We have just proved
\begin{proposition}
Let $\varphi = id$. For any subset $\X = \{x_1, \dots, x_n\} \subset \bX$ the following convergence in distribution holds
\[
    \lim_{N \to \infty} \lim_{M \to \infty}  \Psi^{M, N}_{id}(\X)
    = \lim_{M \to \infty} \lim_{N \to \infty}  \Psi^{M, N}_{id}(\X)
\]
\end{proposition}


\end{document}